\documentclass[10pt,reqno]{amsart}
\usepackage{latexsym}
\usepackage{graphicx}
\usepackage{fancyhdr,amssymb}

\usepackage{amsmath, amsthm, amssymb}
\usepackage{graphicx}
\usepackage{mathbbol}
\usepackage{txfonts}
\usepackage[usenames]{color}
\usepackage{srcltx} 
\usepackage{ifluatex,pdftexcmds}
\usepackage{xpatch}

\xpatchcmd{\proof}
  {\itshape}
  {\bfseries}
  {}
  {}

\theoremstyle{plain}
\numberwithin{equation}{section}
\newtheorem{thm}{Theorem}[section]

\newtheorem{ex}[thm]{Example}
\newtheorem{lem}[thm]{Lemma}   
\newtheorem{prop}[thm]{Proposition}
\newtheorem{defn}[thm]{Definition}
\newtheorem{rmrk}[thm]{Remark}   
\newtheorem{problem}[thm]{Problem}

\newcommand{\be}{\begin{equation}}
\newcommand{\ee}{\end{equation}}

\newcommand{\ptGHto}{\stackrel { \textrm{ptGH}}{\longrightarrow} }
\newcommand{\GHto}{\stackrel { \textrm{GH}}{\longrightarrow} }

\newcommand{\R}{\mathbb{R}}

\newcommand{\diam}{\operatorname{Diam}}

\newcommand{\dil}{\textrm{dil}}

\begin{document}
\fancyhead{}
\renewcommand{\headrulewidth}{0pt}
\fancyfoot{}
\fancyfoot[LE,RO]{\medskip \thepage}
\fancyfoot[LO]{\medskip MISSOURI J.~OF MATH.~SCI., SPRING 2019}
\fancyfoot[RE]{\medskip MISSOURI J.~OF MATH.~SCI., VOL.~31, NO.~1}

\setcounter{page}{1}

\title[Smocked Metric Spaces and their Tangent Cones]{Smocked Metric Spaces and their Tangent Cones}

\author[Sormani, Kazaras and students]{Prof.~Christina Sormani (Lehman and CUNYGC),\\ Dr.~Demetre Kazaras (Stony Brook),\\ and Students: \\
David Afrifa (Lehman) \\
Victoria Antonetti (Lehman) \\
Moshe Dinowitz (Stony Brook)\\
Hindy Drillick (Stony Brook) \\
Maziar Farahzad (Stony Brook)\\
Shanell George (Lehman) \\
Aleah Lydeatte Hepburn (Lehman) \\
Leslie Trang Huynh (Lehman) \\
Emilio Minichiello (Queens) \\
Julinda Mujo Pillati (Lehman)\\
Srivishnupreeth Rendla (Stony Brook) \\
Ajmain Yamin (Stony Brook)\\
}
\address{Department of Mathematics\\
                Lehman College\\
                Bronx, NY 10468}
                \email{sormanic@gmail.com}
 \address{Department of Mathematics\\
                CUNY Graduate Center\\
                365 Fifth Avenue\\
                NY NY 10016}  
 \address{Department of Mathematics\\
                Stony Brook University\\
                100 Nicolls Rd, \\ Stony Brook, NY 11794
                }                              
\email{demetre.kazaras@stonybrook.edu}

\thanks{We are grateful to SCGP and CUNYGC for hosting our meetings.
Prof.~Sormani's research was funded in part by NSF DMS 1612049. Dr.~Kazaras' research
was funded by SB and SCGP.  The students were unfunded volunteers completing the
work for research credit only. }

       

\begin{abstract} 
We introduce the notion of a smocked metric space and explore the
balls and geodesics in a collection of different smocked spaces.  We find their rescaled
Gromov-Hausdorff limits and prove these tangent cones at
infinity exist, are unique, and are normed spaces.  We close
with a variety of open questions suitable for advanced undergraduates, masters students,
and doctoral students.
\end{abstract}

\maketitle

\section{Introduction}

The asymptotic behavior of a metric space at infinity has been well studied by many
mathematicians by taking a sequence of rescalings of the metric spaces and finding their
Gromov-Hausdorff limit.  This limit is called the tangent cone at infinity.  This idea
has ancient roots.  For example, a hyperboloid rescaled in upon itself converges to a cone,
which is equivalent to saying it is asymptotic to a cone at infinity.   However for more abstract
metric spaces, it can be more difficult to understand what it means to take a limit
and it can be more difficult to find that limit.  Often the limit is not unique and there is no
reason for it to be a cone.   

Here we introduce a new class of metric spaces that we call {\em smocked
spaces} inspired by the craft of smocking fabric.  Each smocked space is defined by 
taking a Euclidean space with a pattern on it 
(as in Figure~\ref{fig:smocking-patterns}) and then pulling each stitch in the
pattern to a single point.  The notion of pulling a thread to a point has already been
explored by metric geometers as it provides interesting counter examples to questions
involving areas and perimeters (cf. \cite{Area-spaces} by Burago-Ivanov).  However this is the first time anyone has explored 
more complex patterns in which many threads are drawn to points.  We discover that
indeed we obtain some rather surprising tangent cones at infinity when studying these spaces.

\begin{figure}[h]
\includegraphics[width=.3 \textwidth]{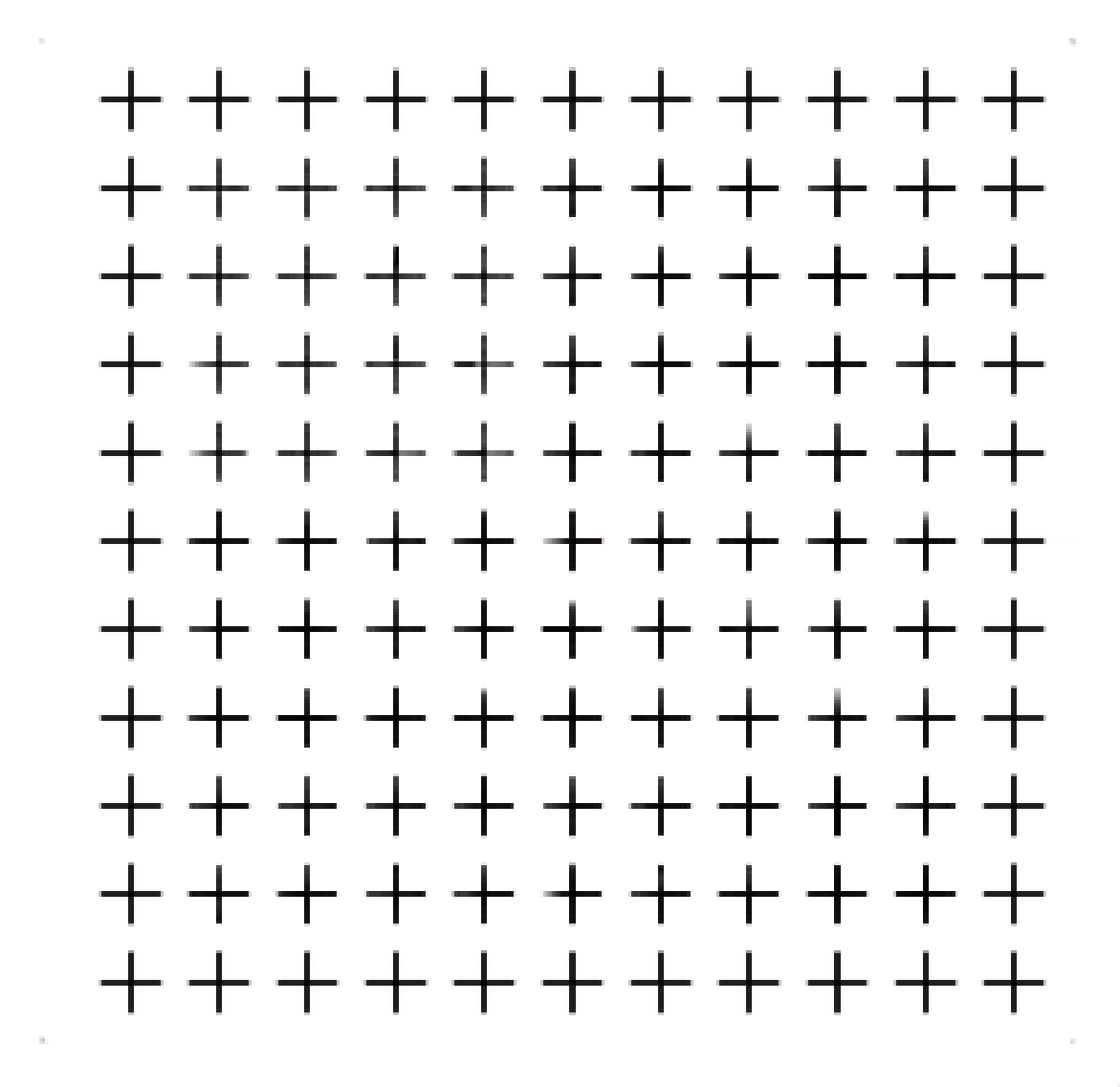} 
\includegraphics[width=.3 \textwidth]{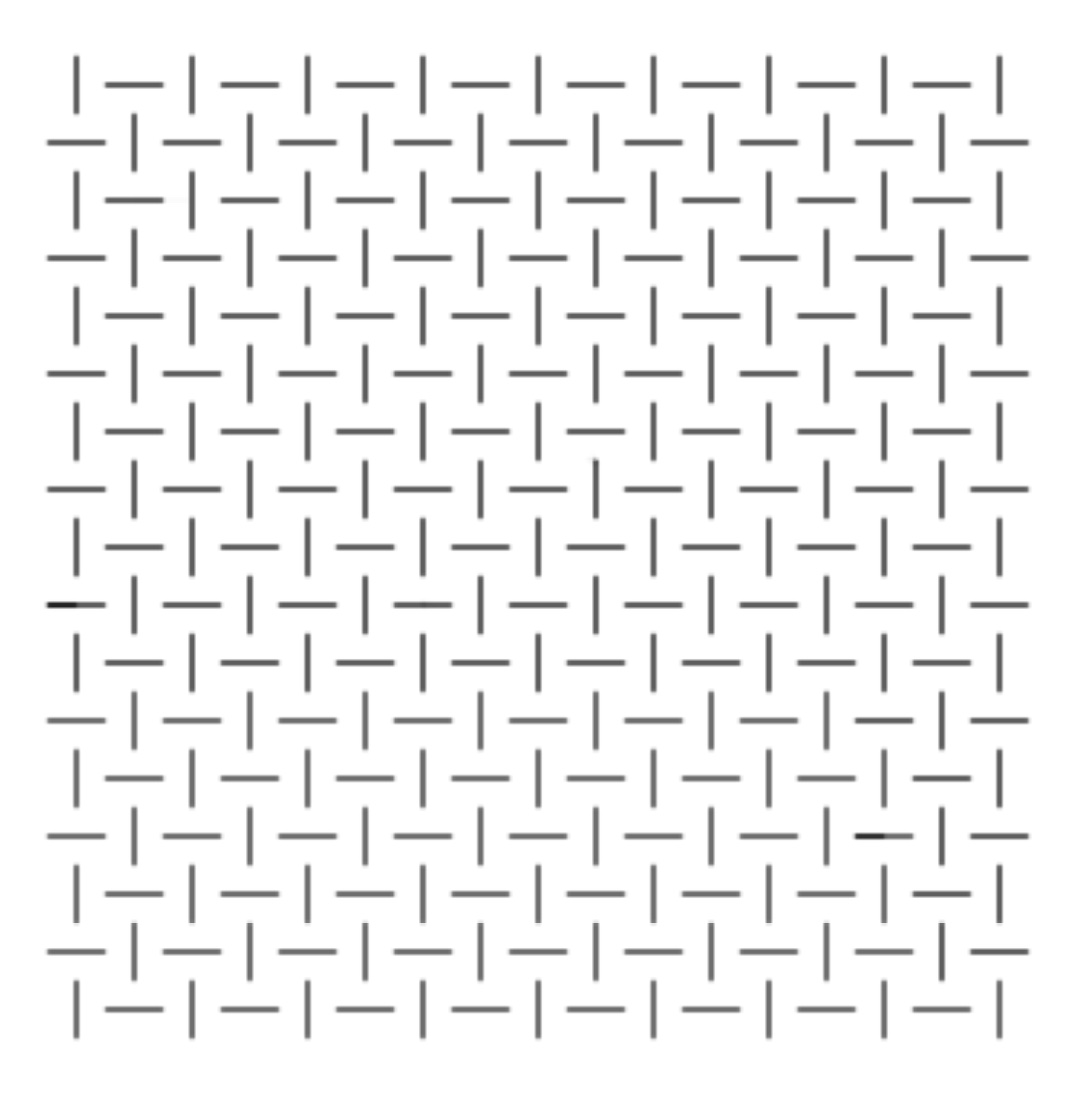} 
\includegraphics[width=.3 \textwidth]{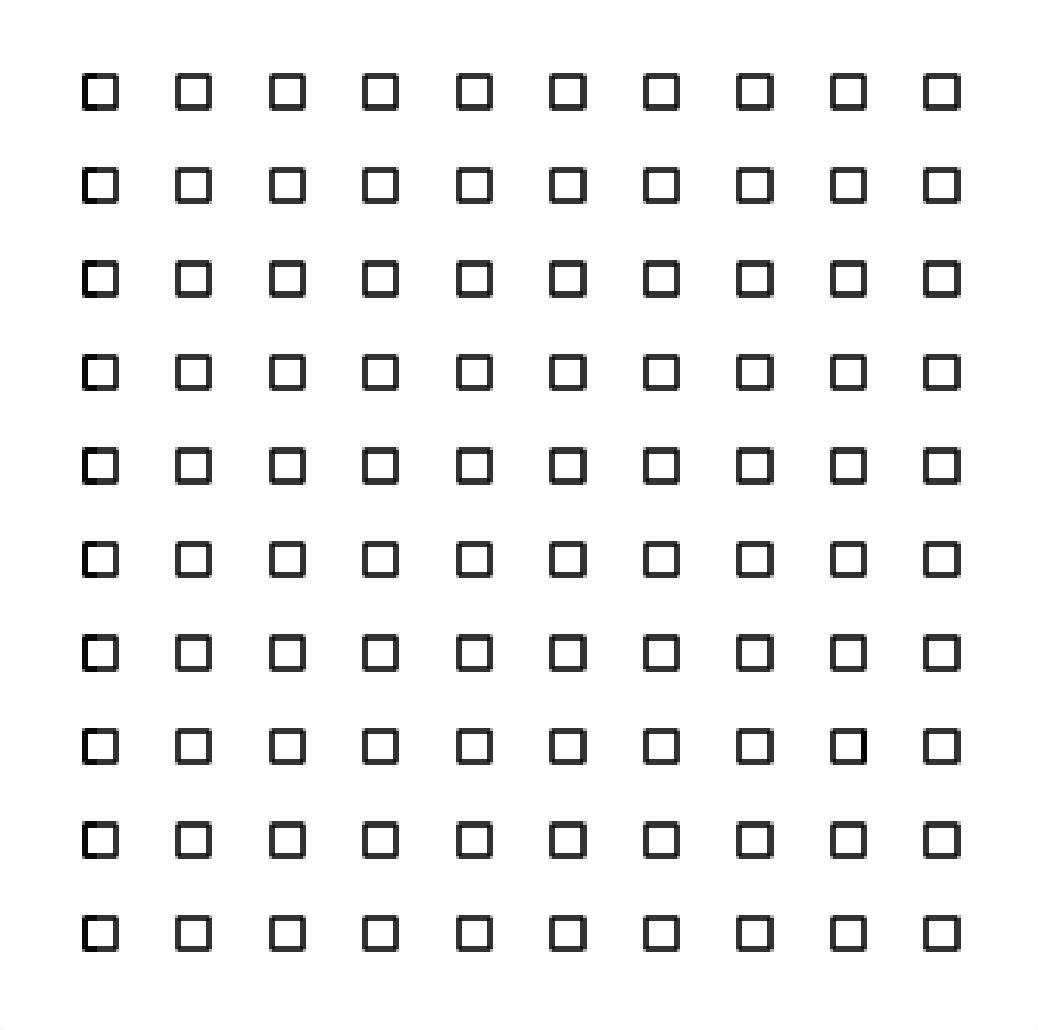} \\
\includegraphics[width=.3 \textwidth]{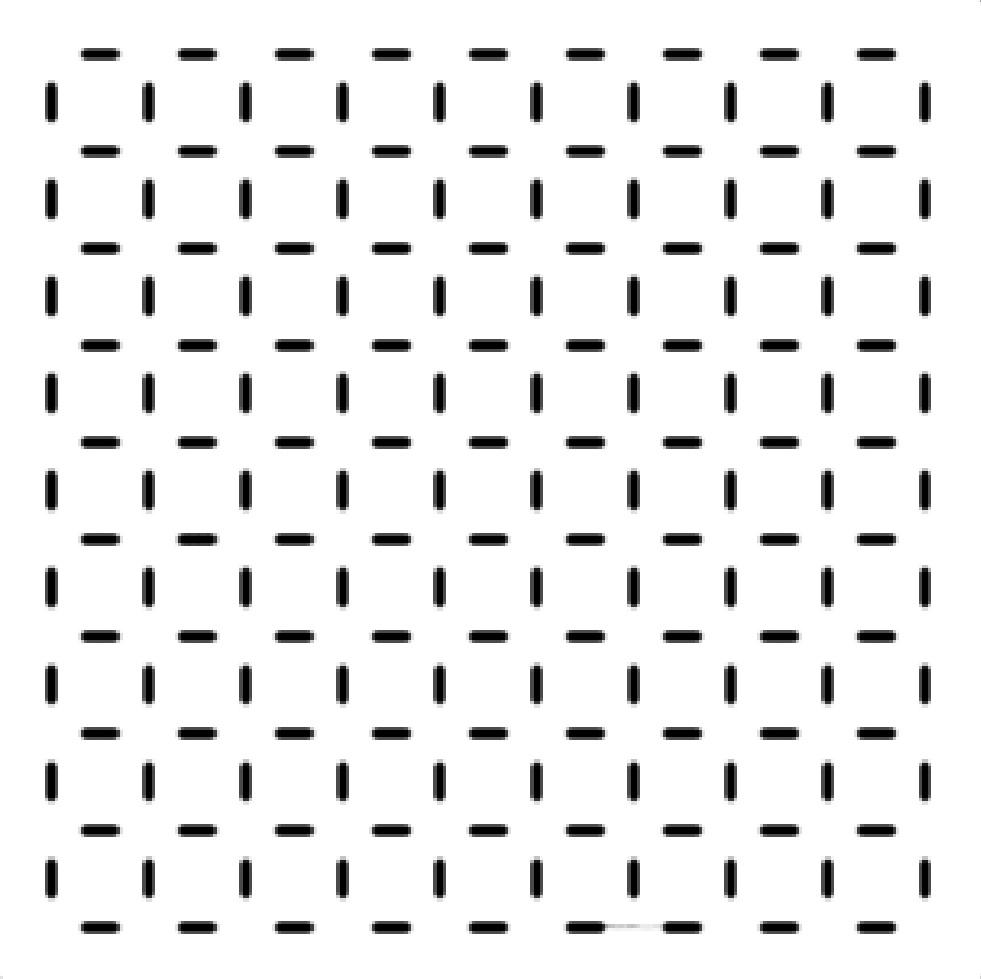} 
\includegraphics[width=.3 \textwidth]{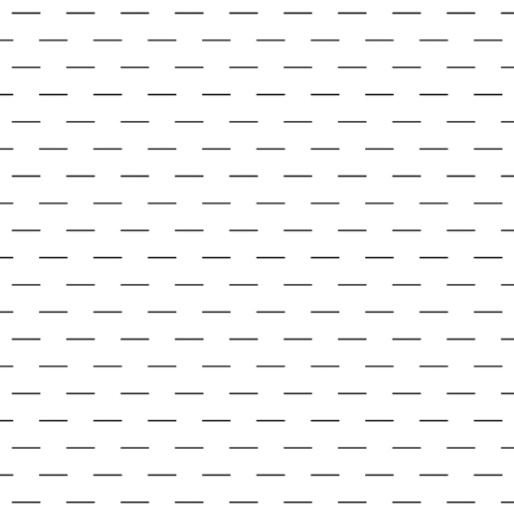} 
\includegraphics[width=.3 \textwidth]{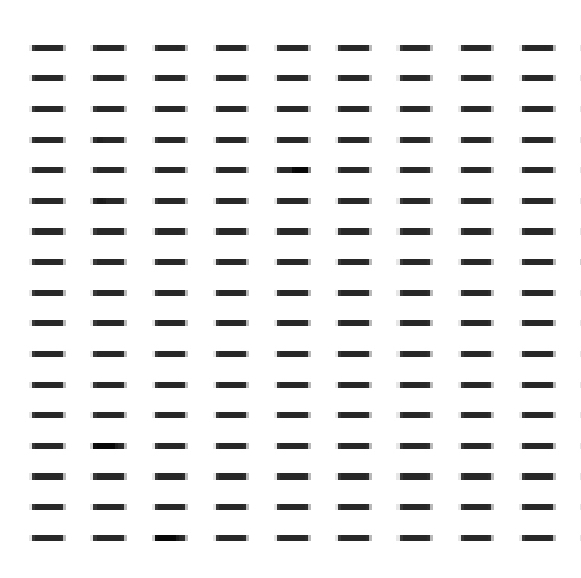} 
\caption{The smocked spaces: $X_+$, $X_T$, $X_\square$, and $X_H$, $X_\diamond$,  $X_=$.}
\label{fig:smocking-patterns}
\end{figure}

We present the rigorous definition of a smocked space appears in Definition~\ref{defn-smock} after reviewing the 
definition of pulling a thread to a point in Definition~\ref{defn-ps}.   We prove some lemmas
about balls in these spaces, which allow us to explore the balls of all six spaces on an intuitive level.
We then choose three smocked spaces to analyze more deeply, ultimately determining
their tangent cones at infinity in Theorem~\ref{thm-tan+}, Theorem~\ref{thm-tan-square},  
and Theorem~\ref{thm-tan-T}.  We've also proven
a general theorem one can apply to find the Gromov-Hausdorff limit of a large class of smocking
spaces in Theorem~\ref{thm-smocking-R}, while proving these tangent cones are unique and are normed spaces.   
In upcoming work by some of the authors \cite{SWIF-smocked},
we complete the analysis of the other three smocked spaces and examine three dimensional versions of them as well.
In upcoming work by Kazaras and Sormani \cite{Kazaras-Sormani-tori}, we will apply the results in these
two papers to a question of Gromov and Sormani \cite{Gromov-Sormani-IAS-report}.\footnote{Prof.~Sormani's grant
NSF DMS 1612049 funds this overarching project on scalar curvature but the students were unpaid volunteers
meeting in January 2019 at the CUNYGC and SCGP.}

This paper can be read by anyone who has studied basic metric geometry as we 
review every concept before we need to apply it.  Through the reading of this paper,
one can learn the definition of Gromov-Hausdorff convergence, a standard technique for proving this convergence
using correspondences, and the definition of the tangent cone at infinity of a metric space.  As each concept is
introduced, it is immediately applied to each of these patterns by various teams of students.  It may be fun for the
reader to find a smocking pattern online or in a smocking guide like \cite{smocking-book} and work with it while 
reading this paper to discover something new.
The reader might also slightly adapt the patterns explored within to see the consequences of altering the
sizes of stitches and the spacing between them. 

The paper closes with a section of open problems some of which are labeled as possible projects
for an undergraduate or masters thesis.   Other open problems are significantly more advanced.  Students
who would like to study metric geometry on a more advanced level are encouraged to read
Burago-Burago-Ivanov's award winning textbook \cite{BBI}.  

\tableofcontents

\section{Background}

Here we quickly review metric spaces and pulled thread spaces.  See also the award winning
textbook by Burago-Burago-Ivanov \cite{BBI}.  If you already know these topics, just briefly glance
through for notation.

\subsection{Metric spaces by Prof.~Sormani and Ajmain}

\begin{defn} \label{defn-metric}
A metric space $(X,d)$ is a set $X$ with a function $d: X\times X \to \R$
satisfying the following:
\begin{itemize}
\item Nonnegativity: $\quad d(x,y) \ge 0 \quad \forall x,y \in X.\quad $ 
\item Definiteness: $\quad d(x,y)=0 \iff x=y$.
\item Symmetry: $\quad d(x,y)=d(y,x) \quad \forall x,y \in X.$ 
\item The Triangle Inequality: $\quad d(x,y) \le d(x,z) + d(z,y) \quad \forall x,y,z \in X.$
\end{itemize}
\end{defn}

In this paper we will use the following notation:

\begin{defn}
Given a point $p$ in a metric space $X$ and $r>0$.
An open ball of radius $r$ about $p$:
\be
B_r(p)=\{x:\, d(x,p)<r\}.
\ee
A closed ball of radius $r$ about $p$:
\be
\bar{B}_r(p)=\{x:\, d(x,p)\le r\}.
\ee
A sphere of radius $r$ about $p$:
\be
\partial B_r(p)=\{x:\, d(x,p)=r\}.
\ee
\end{defn}

\begin{defn}
Given a set $K$ in a metric space $X$ and $r>0$,
The tubular neighborhood of radius $r$ about $K$:
\be
T_r(K) =\{x: \, \exists y\in K \, s.t.\, d(x,y)<r\}.
\ee
\end{defn}

\begin{lem} \label{tubular-ball}
Given a point $p$ in a metric space $X$ and $r,s>0$
we have
\be
T_s(B_r(p))=B_{r+s}(p).
\ee
\end{lem}

\begin{proof} 
Let $q\in T_s(B_r(p))$. By the definition of tubular neighborhood, 
$\exists y\in B_r(p)$ such that $d(q,y)<s$.  
By the triangle inequality $d(q,p) < r+s$, and therefore $q \in B_{r+s}(p)$
Now take $q\in B_{r+s}(p)$. So $d(q,p) < r+s$.
If  $B_s(q) \cap B_r(p)= \varnothing$ then $ r+s \leq d(q,p)$ which leads to a contradiction.  Hence 
$\exists y \in B_r(p)$ such that $d(y,q)<s$, and so $q\in T_s(B_r(p))$.
\end{proof}

\begin{lem} \label{tubular-subset}
If $K_1\subset K_2$ then $T_s(K_1) \subset T_s(K_2)$.
\end{lem}

\begin{proof} 
Let $p\in T_s(K_1)$.  Then $\exists y\in K_1$ such that $d(p,y)< s$.  Since $K_1\subset K_2$, we have $\exists y\in K_2$ such that $d(p,y)\in s$.  Thus $p\in T_s(K_2)$.
\end{proof}

\begin{lem} \label{tubular-cup}
$T_s(K_1\cup K_2) = T_s(K_1) \cup T_s(K_2)$.
\end{lem}

\begin{proof} 
This follows from the definition of union:
\begin{eqnarray*}
T_s(K_1\cup K_2) &=& \{ z: \, \exists w \in K_1\cup K_2 \,s.t.\, d(z,w)<s\}\\
&=&\{ z: \, \exists w \in K_1\,s.t.\, d(z,w)<s\,OR\, \exists w\in K_2 \,s.t.\, d(z,w)<s\}\\
&=&\{ z: \, \exists w \in K_1\,s.t.\, d(z,w)<s\} \cup \{ z: \, \exists w \in K_2\,s.t.\, d(z,w)<s\}\\
&=& T_s(K_1) \cup T_s( K_2). 
\end{eqnarray*}
\end{proof}

\begin{lem} \label{tubular-cap}
$T_s(K_1\cap K_2) \subset T_s(K_1) \cap T_s(K_2)$.
\end{lem}

\begin{proof} 
This follows from the definition of intersection:
\begin{eqnarray*}
T_s(K_1\cup K_2) &=& \{ z: \, \exists w \in K_1\cap K_2 \,s.t.\, d(z,w)<s\}\\
&=&\{ z: \, \exists w \in K_1\,AND\,  K_2 \,s.t.\, d(z,w)<s\}\\
&\subset &\{ z: \, \exists w \in K_1\,s.t.\, d(z,w)<s\} \cap \{ z: \, \exists w \in K_2\,s.t.\, d(z,w)<s\}\\
&=& T_s(K_1) \cap T_s( K_2).
\end{eqnarray*}
\end{proof}

\begin{lem}\label{tubular-interval}
In ${\mathbb{E}}^2$ we have the following
\be
T_s([a,b]\times\{y\})\subset [a-s,b+s]\times[y-s, y+s]
\ee
\be
T_s(\{x\}\times [a,b])\subset [x-s,x+s]\times [a-s,b+s].
\ee
\end{lem}

\begin{proof} 
Let $(p_1,p_2)\in T_s([a,b]\times\{y\})$.  
Then  $\exists(q_1,q_2)\in [a,b]\times\{y\}$ such that 
\be
d(p,q)= \sqrt{(p_1-q_1)^2 + (p_2-q_2)^2}<s.
\ee
  Hence $|p_1-q_1|<s$ and $|p_2-q_2|<s$. Therefore $p\in [a-s,b+s]\times[y-s,s+s]$, as required.  
A similar argument shows $T_s(\{x\}\times [a,b])\subset [x-s,x+s]\times [a-s,b+s]$.
\end{proof}

\subsection{Pulled Thread Spaces by Prof.~Sormani and Ajmain} \label{sect-ps}  

Before introducing a smocked metric space we recall the notion of a pulled thread metric space, 
particularly in the setting where one starts with a Euclidean space, $\mathbb{E}^N$.  The idea
is that if one views a Euclidean plane as a cloth, and marks an interval on that cloth, then sews 
along that interval, and then pulls the thread tight, one obtains a new metric space (called a pulled
strong space) in which the interval is now a point as in Figure~\ref{fig:pulled-thread}.

\begin{figure}[h]
\includegraphics[width=.15 \textwidth]{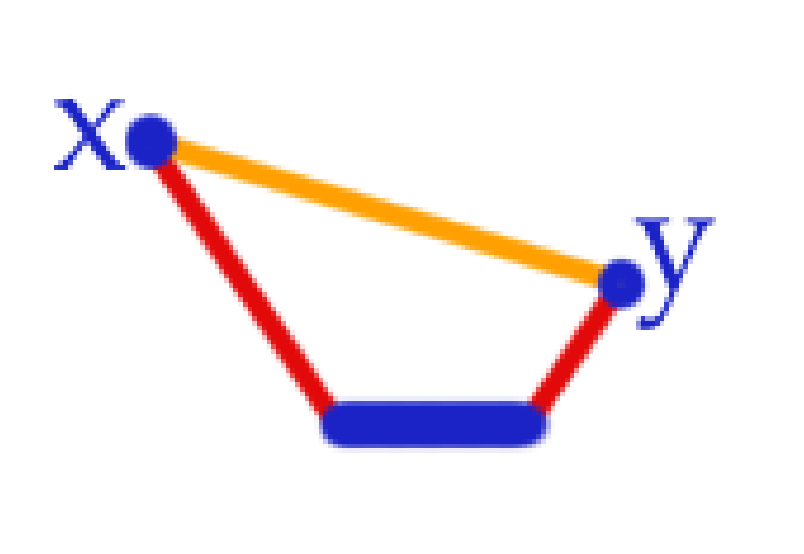} 
\caption{The distance between two points in a pulled thread space is the minimum of the length of the direct path and the sum of the lengths of a pair of segments touching the interval.}
\label{fig:pulled-thread}
\end{figure}

\begin{defn} \label{defn-ps}
Given a Euclidean space, $\mathbb{E}^N$, and an interval, $I$, one can define the {\bf pulled thread
metric space}, $(X,d)$, in which the interval is viewed as a single point.  
\be
X = \{ x: \, x \in {\mathbb{E}^N}\setminus I\} \cup \{I\}.
\ee
We have a {\bf pulled thread map} $\pi: \mathbb{E}^N \to X$ such that $\pi(x)=x$ for $x \in {\mathbb{E}^N}\setminus I$
and $\pi(x)=I$ for $x\in I$.
The distance function, $d: X\times X \to [0,\infty)$, is defined by
\begin{eqnarray}
d(x,y) &=& \min\{ |x-y|, \min\{ |x-z|+|z'-y|:\, z, z'\in I\} \quad \textrm{ for } x,y\in \mathbb{E}^N \\
d(x, I) &=& \min\{ |x-z|: \, z\in I \} \quad \textrm{ for } x\in \mathbb{E}^N.
\end{eqnarray}
We can then define the {\bf pulled thread pseudometric} $\bar{d}: {\mathbb{E}^N}\times {\mathbb{E}^N} \to [0, \infty)$
to be
\be
\bar{d}(x,y)= d(\pi(x), \pi(y))= \min\{ |x-y|, \min\{ |x-z|+|z'-y|:\, z, z'\in I\}. 
\ee
Finally we define the distance to the interval, $D: {\mathbb{E}^N}\to [0,\infty)$:
\be
 D(x)= \min\{ |x-z|: \, z\in I \} = d(\pi(x),I).
\ee
Notation: For any $p\in X$, $p\neq I$, write $z_p$ to denote the unique point in $I$ such that $|p-z_p| = d(p,I)$.
\end{defn}

\begin{lem} \label{lem-ps-metric}
A pulled thread space is a metric space.
\end{lem}

\begin{proof} This has three parts:  

\noindent
{\em \bf Positive definiteness: $d(p,q)\ge 0$ and $d(p,q)=0 \iff p=q$:}

The fact that $d(p,q) \geq 0$ for all $p,q\in X$ is clear from Definition~\ref{defn-ps}.  Suppose $d(p,q) = 0$
then $p=\pi(v)$ and $q=\pi(w)$ where
Then either 
\be
|v-w|=0 \implies v=w \implies p=q
\ee
or there exists $z, z'\in I$ such that
\be
|v-z|+|z'-w|=0 
\ee
which implies $v=z$ and $z'=w$ so $v,w \in I$ so 
\be
p=\pi(v)=\pi(w)=q.
\ee
{\em \bf Symmetry $d(p,q)=d(q,p)$: } 

This follows by taking $p=\pi(x)$ and $q=\pi(y)$ and noting that
\begin{eqnarray}
d(p,q)&=&\min\{ |x-y|, \min\{ |x-z|+|z'-y|:\, z, z'\in I\} \\
&=&\min\{ |y-x|, \min\{ |y-z|+|z'-x|:\, z, z'\in I\}=d(q,p).
\end{eqnarray}
{\em \bf Triangle inequality $d(a,b) \le d(a,q)+d(q,b)$:}  

We prove the triangle inequality in cases
using the notation that for any $p\in X$ there exists $v_p \in {\mathbb{E}}^N$ such that
$p=\pi(v_p)$ and if $p\neq I$ then there exists $z_p\in I$ such that 
\be
d(p,I)= |v_p-z_p|.
\ee
{\bf Case I: }We assume $a\in \pi(X\setminus \{I\})$ and $b= I$ (which also includes $a$ and $b$ switched by symmetry).
This breaks into two cases:

{\bf Case I.a:} We assume $q=I$, This  implies $d(b,q)=0$ and $d(a,q)=d(b,q)$ so
\be
d(a,b) = d(a,q) + d(q,b) 
\ee

{\bf Case I.b:} We assume $q\neq I$. This breaks into two deeper cases

{\bf Case I.b.i:}  We assume $d(a,q) = |a-q| $ which implies
\begin{align}
 d(a,b) = |v_a- z_a| < |v_a-z_q| \leq |v_a-v_q| - |v_q-z_q| = d(a,q) + d(q,b)
\end{align}

{\bf Case I.b.ii:} We assume $d(a,q) = |v_a-z_a| + |v_q-z_q|$ which implies
\begin{align}
 d(a,b) = |v_a - z_a| < |v_a-z_q| + 2|v_q-z_q| = d(a,q)  + d(q,b)
\end{align}
{\bf Case II :} We assume $a, b\in \pi(X\setminus \{I\})$.  This breaks into two cases:

{\bf Case II.a:} We assume $q=I $ which implies
\be
d(a,b) \leq |v_a-z_a| + |v_b-z_b| = d(a,q) + d(q,b)
\ee

{\bf Case II.b: } We assume $q\neq I$.  This breaks into four deeper cases:

\quad  {\bf Case II.b.i: }We assume $ d(a,q) = |v_a-z_a| + |v_q-z_q| $ and $ d(q,b) = |v_b-z_b| + |v_q-z_q|$. 
\begin{align}
\textrm{So }  d(a,b) &\leq  |v_a-z_a| + |v_b-z_b| \\ &\leq |v_a-z_a| + |v_b-z_b| + 2|v_q-z_q| \\&= d(a,q) + d(q,b).
\end{align}
\quad \quad {\bf Case II.b.ii:} We assume $d(a,q) = |v_a-v_q|$ and $ d(q,b) = |v_b-z_b| + |v_q-z_q| $. 
\begin{align}
\textrm{So }  d(a,b) &\leq |v_a-z_a| + |v_b-z_b| \leq |v_a-z_q| + |v_b-z_b| \\&\leq |v_a-v_q|  + |v_q-z_q|+ |v_b-z_b| \\ &= d(a,q) + d(q,b).
\end{align}
\quad \quad {\bf Case II.b.iii:} We assume $d(a,q) = |v_a-z_a| + |v_q-z_r|$ and $ d(q,b) = |v_b-v_q| $. 
\begin{align}
\textrm{So }  d(a,b) &\leq |v_a-z_a| + |v_b-z_b| \\&\leq |v_a-z_a| + |v_b-q| + |v_q-z_q|\\& = d(a,q). + d(q,b)
\end{align}
\quad \quad {\bf Case II.b.iv:} We assume $d(a,q) = |v_a-v_q|$ and  $d(v_q,v_b) = |v_b-v_q| $. 
\begin{align}
\textrm{So } d(a,b)\leq |v_a-v_b| \leq |v_a-v_q| + |v_b-v_q| =  d(a,q) + d(q,b).
\end{align}
\end{proof}

\begin{rmrk}
More generally pulled thread spaces can be defined starting with any geodesic metric space or length space.
See Burago-Burago-Ivanov's textbook \cite{BBI}.
\end{rmrk}

\begin{rmrk}
Note that there is no particular reason for $I$ to be a closed interval.  It might be any
compact set.  But the classical definition of a pulled thread is that the $I$ is an interval.
\end{rmrk}

\subsection{Balls in Pulled Thread Spaces by Prof.~Sormani, Ajmain and Julinda}

Here we prove in three lemmas
that the balls in pulled thread spaces have the form depicted in Figure~\ref{fig:pulled-thread-balls-center}.

\begin{figure}[h]
\includegraphics[width=.4 \textwidth]{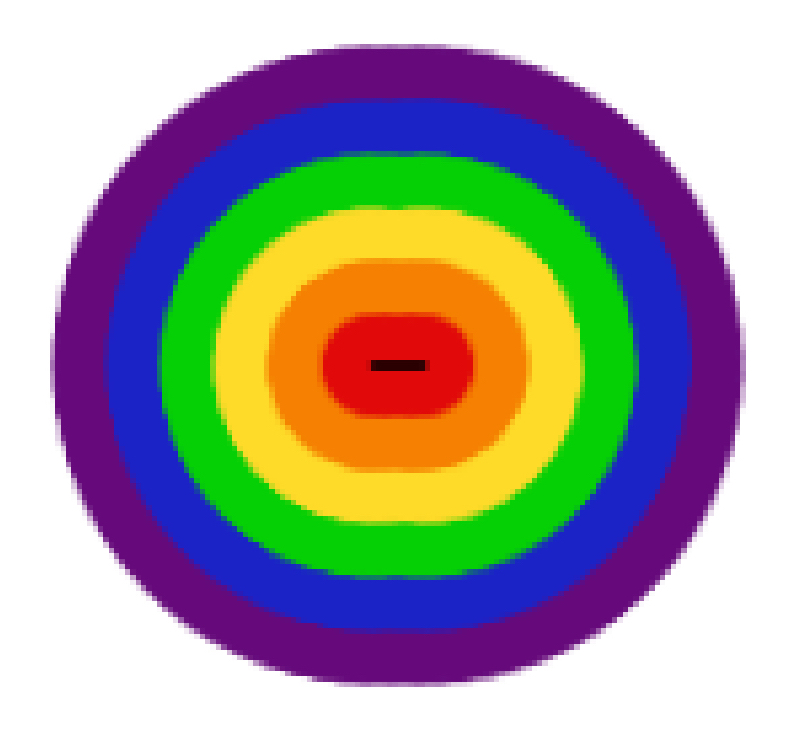} \includegraphics[width=.4 \textwidth]{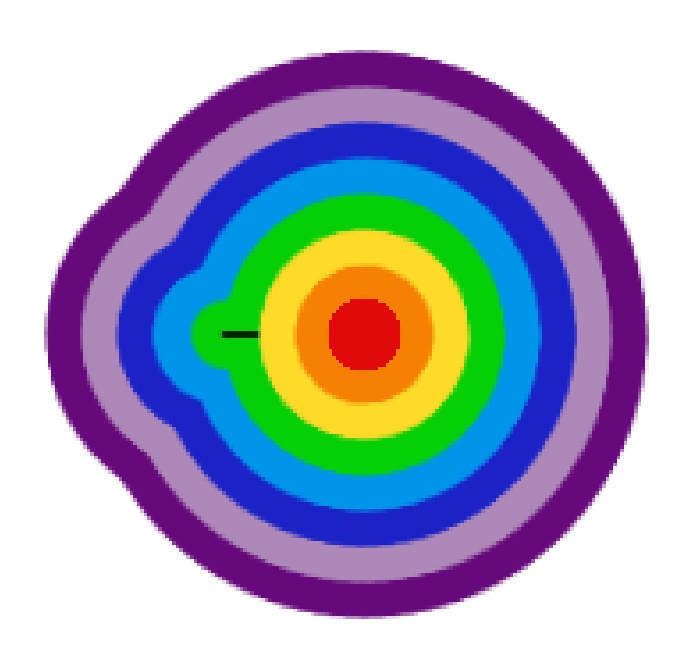} 
\caption{Concentric balls about the pulled interval as in Lemma~\ref{ball-ps-1} and balls
about point that is not the pulled interval as in Lemmas~\ref{ball-ps-1}
and~\ref{ball-ps-add}.}
\label{fig:pulled-thread-balls-center}\label{fig:pulled-thread-balls-off}
\end{figure}

\begin{lem} \label{ball-ps-1}
In a pulled thread space as in Definition~\ref{defn-ps}
\be
\pi^{-1}(B_r(I))= T_r(I).
\ee
See the left side of Figure~\ref{fig:pulled-thread-balls-center}.
\end{lem}

\begin{proof} 
Note that $v \in T_r(I)$ if and only if
\be
\exists\, z\in I \textrm{ such that } |v-z|<r.
\ee
By Definition~\ref{defn-ps}, this is true if and only if
\be
d(\pi(v), \pi(I))=\min\{ |v-z'|:\, z'\in I\} <r.
\ee
This is true if and only if $\pi(v) \in B_r(\pi(I))$ 
which is true if and only if
 $v \in \pi^{-1}(B_r(I))$.
\end{proof}

\begin{lem} \label{ball-ps-h}  
In a pulled thread space as in Definition~\ref{defn-ps}, 
\be
\pi^{-1}(B_r(\pi(x)))= B_r(x) \quad \forall  r\le D(x).
\ee
See the red, orange, and yellow balls on the right side of Figure~\ref{fig:pulled-thread-balls-off}.
\end{lem}

\begin{proof}  
If $v\in \pi^{-1}(B_r(\pi(x)))$ then 
$
\pi (v)\in B_r(\pi (x)).
$
Thus by Definition~\ref{defn-ps},
\be
\min\{|v-x|, |v-w|+|w'-x|: \, w, w'\in I\} =d(\pi(v), \pi(x)) < r < D(x).
\ee
Since $|w'-x|\ge D(x)$, we see that 
\be
|v-x|=d(\pi(v), \pi(x)) < r
\ee
which implies that $v\in B_r(x)$.

On the other hand, if $v\in B_r(x)$ then $|v-x|<r$.  So
\be
d(\pi(v), \pi(x))=\min\{|v-x|, |v-w|+|w'-x|: \, w, w'\in I\} \le |v-x|<r
\ee
which implies that $\pi(v) \in \pi^{-1}(B_r(\pi(x)))$.
\end{proof}

\begin{lem} \label{ball-ps-add}
In a pulled thread space as in Definition~\ref{defn-ps},
if $D(x) = h$ then 
\be
\pi^{-1}(B_{h+s}(x))= B_{h+s}(x)  \cup T_s(I) \quad \forall  s>0.
\ee
See the green, blue, and purple balls in Figure~\ref{fig:pulled-thread-balls-off}.
\end{lem}

\begin{proof} 
Observe, 
\begin{align*}
    v\in \pi ^{-1}(B_{h+s}(\pi (x))) \\
    \iff & \pi (v)\in B_{h+s}(\pi (x)) \\
    \iff & d(\pi (v), \pi (x)) < h+s \\
    \iff & \min \{|v-x| , \min\{|v-z_1|+|z_2-x|: z_1,z_2\in I\}\} <h+s \\
    \iff & |v-x|<h+s \,\,OR \,\, \exists \,z_1,z_2\in I \,\,s.t. \,\,|v-z_1|+|z_2-x|<h+s. 
    \end{align*}
Note: we can take $z_2$ closest to $x$ so $|z_2-x|=D(x)=h$.  Thus
\begin{align*}
  v\in \pi ^{-1}(B_{h+s}(\pi (x))) \\
    \iff & v\in B_{h+s}(x) \,\,OR\,\, \exists z_1 \in I \,\,s.t.\,\, |v_1-s| <s \\
    \iff & v\in B_{h+s}(x) \,\,OR\,\, v\in T_s(I) \\
    \iff & v\in B_{h+s}(x) \cup T_s(I).
\end{align*}
\end{proof}

\begin{prop}\label{ps-v}
In a pulled thread space as in Definition~\ref{defn-ps},
if the length of the interval is $L>0$ then
\be
B_r(x) \subset \pi^{-1}(B_r(\pi(x))) \subset B_{r+L}(x).
\ee
\end{prop}

\begin{proof}  
We see that $B_r(x) \subset \pi^{-1}(B_r(\pi(x)))$ because
\be
w\in B_r(x) \implies r>|w-x|  \implies r > d(\pi(w),\pi(x)) \implies \pi(w)\in (B_r(\pi(x))).
\ee
We see that $\pi^{-1}(B_r(\pi(x))) \subset B_{r+L}(x)$ because
\be
v\in \pi^{-1}(B_r(\pi(x)))
\implies \pi(v)\in B_r(\pi(x))  \implies d(\pi(v),\pi(x))<r.
\ee
Case 1:  $d(\pi(v),\pi(x)) = |v-x|$.   
$$\implies |v-x|<r\implies v\in B_r(x) \subseteq B_{r+L}(x).$$
Case 2:  $d(\pi(v),\pi(x)) =|v-z_1|+|z_2-x|$ for some  $z_1,z_2\in I$
$$ \implies |v-x|\leq |v-z_1| + |z_1 - z_2| + |z_2 - x|\le r+L \implies v\in B_{r+L}(x).$$  
\end{proof}

\section{Introducing Smocked Spaces}

We now introduce a new notion called a smocked space.
In sewing there is a technique called smocking which is used
to add texture to a cloth.  See for example \cite{smocking-book} for a few patterns
and search ``Canadian Smocking'' in {\em Pinterest} for many more.  
To create such a smocked cloth, the seamstress follows
a pattern.   In Figure~\ref{fig:smocking-patterns} we presented some such patterns.
Each interval (or stitch) marked in black is sewn by a 
thread and pulled to a point.    Sometimes the stitches are 
squares or plus signs.    There are many other standard smocking patterns
which are also periodic and many more which are not periodic.

In this section we rigorously define the metric space created from a plane
by pulling every stitch in a smocking pattern to a point.  We then describe the
six patterns and their smocked spaces each on their own subsection along
with graphics.   It is recommended that the reader glance into these subsections
while reading the definition.

\subsection{ The Definition of a Smocked Space by Prof.~Sormani}

\begin{defn} \label{defn-smock}
Given a Euclidean space, $\mathbb{E}^N$, and a finite or countable collection of
disjoint connected compact sets called {\bf smocking intervals} or {\bf smocking stitches}, 
\be
\mathcal{I}=\{I_j: \, j \in J\},
\ee
with a positive {\bf smocking separation factor},
\be \label{s-factor}
\delta=\min\{|z-z'|: \, z\in I_j, \, z'\in I_{j'},\, j\neq j' \in J\} >0,
\ee
 one can define the {\bf smocked
metric space}, $(X,d)$, in which each stitch is viewed as a single point.  
\be
X = \left\{ x: \, x \in {\mathbb{E}^N}\setminus S\right\} \cup \mathcal{I}
\ee
where $S$ is the {\bf smocking set} or {\bf smocking pattern}:
\be
S= \bigcup_{j \in J} I_j .
\ee
We have a {\bf smocking map} 
$
\pi: \mathbb{E}^N \to X $ defined by
\be
 \pi(x) = \begin{cases} 
          x & \textrm{ for }x \in {\mathbb{E}^N}\setminus S \\
          I_j&  \textrm{ for } x\in I_j \textrm{ and } j\in J
                 \end{cases}
\ee
The {\bf smocked distance function}, $d: X\times X \to [0,\infty)$, is defined for $y, x\notin \pi(S)$, and stitches
$I_m$ and $I_k$ as follows:
\begin{eqnarray*}
d(\,x,\,y\,) &=& \min \left\{d_0(x,y), d_1(x,y), d_2(x,y), d_3(x,y), ...\right\} \\
d(\,x,\, I_k) &=& \min \{ d_0(x,z),  d_1(x,z), d_2(x,z), d_3(x,z), ...:\, z\in I_k\} \\
d(I_m,I_k) &=& \min \{ d_0(z',z), d_1(z',z), d_2(z',z), d_3(z',z), ... \,:z'\in I_m,\, z\in I_k \},
\end{eqnarray*}
where $d_k$ are the sums of lengths of segments that jump to and between $k$ stitches:
\begin{eqnarray*}
d_0(v,w) &=& |v-w|\\
d_1(v,w) &=&  \min\{ |v-z_1|+|z'_1-w|:\, z_1, z_1'\in I_{j_1}, \, j_1 \in J\}\\
d_2(v,w) &=& \min\{ |v-z_1|+|z'_1-z_2|+|z'_2-w|:\, z_i, z'_i\in I_{j_i}, \, j_1\neq j_2 \in J\}\\
d_k(v,w) &=& \min \{  |v-z_1|+\sum_{i=1}^{k-1} |z'_i-z_{i+1}|+|z'_k-w|:\, z_i, z'_i\in I_{j_i}, \, j_1\neq \cdots \neq j_k \in J\}.
\end{eqnarray*}
We define the {\bf smocking pseudometric} $\bar{d}: {\mathbb{E}^N}\times {\mathbb{E}^N} \to [0, \infty)$
to be
$$
\bar{d}(v,w)= d(\pi(v), \pi(w))=\min \{d_k(v',w'): \,\pi(v)=\pi(v'),\, \pi(w)=\pi(w'),\, k\in {\mathbb{N}}\}.
$$
We will say the smocked space is {\bf parametrized by points in the stitches}, if 
\be\label{param-by-points}
J \subset {\mathbb{E}}^N \textrm{ and } \forall j \in J \,\, j \in I_j.
\ee
\end{defn}

In Theorem~\ref{thm-smock-metric} below, we will prove the minima are achieved and that the smocking
spaces are metric spaces.   In fact, we will see that the
distances on these spaces can be determined by explicitly finding the shortest
path of segments $z_i'$ to $z_i$ between a given pair of points.   In Figure~\ref{fig:geodesics} we see
such paths in one smocked space.   If such a path has a single segment and jumps
through no intervals then $d(x,y)=d_0(x,y)$.  If the path jumps through one
interval then $d(x,y)=d_1(x,y)$, and if it jumps through $k$ intervals with $k+1$ segments then
$d(x,y)=d_k(x,y)$.

\begin{figure}[h]
\includegraphics[scale=.3]{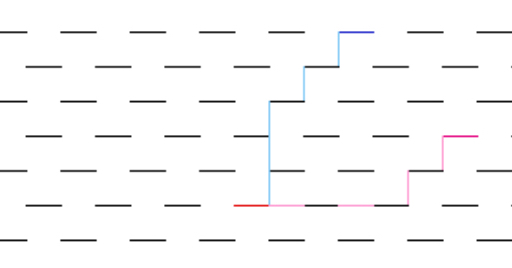}
\caption{Here we see a path of segments in blue from a dark blue interval to a
red interval and another path of segments in pink from a dark pink interval to the
red interval.}
\label{fig:geodesics}
\end{figure}

\begin{rmrk}
Note that often the $I_k$ are closed intervals.   As one can
see in the introduction, in some smocking patterns the stitches are 
replaced with squares or plus signs
or wedges. 
\end{rmrk}

\subsection{ A Smocked Space is a Metric Space by Prof.~Sormani}

\begin{thm} \label{thm-smock-metric}
The smocked metric space is a well defined metric space and in fact
for any $v,w  \in \mathbb{E}^N$ 
\be \label{exists-N}
\exists N(v,w) \le \lceil{|v-w|/\delta}\rceil \,\, s.t.\,\, d_{N(v,w)}(v,w) \le d_{k}(v,w) \quad \forall k \ge N,
\ee
where $\delta$ is the separation factor.    
So the minimum in the definition of the smocking distance and pseudometric is achieved 
\be
\bar{d}(v,w)=d_N(v,w) \textrm{ and } d(\pi(v), \pi(w))=d_N(v,w).
\ee
\end{thm}

\begin{proof}
Initially we consider all minima in the definition of a smocked space to be infima.  

First, we fix $v \neq w \in {\mathbb{E}^N}$, and prove (\ref{exists-N}).
Observe that
\be \label{k-delta}
d_k(v,w) \ge k \delta
\ee
where $\delta$ is the separation factor defined in (\ref{s-factor})
because each $ |z'_i-z_i| \ge \delta$ in the definition of $d_k$.
Thus
\be
d_0(v,w)=|v-w| \le d_k(v,w) \qquad \forall k \ge N'=\lceil{|v-w|/\delta}\rceil.
\ee
So we need only choose $N_{v,w}$ such that
\be
d_{N(v,w)}(v,w)= \min\{ d_0(v,w),..., d_{N'}(v,w)\}.
\ee

We now examine
\be
d_{k}(v,w)=\inf \{  |v-z_1|+\sum_{i=1}^k |z'_i-z_i|+|z'_k-w|:\, z_i, z'_i\in I_{j_i}, \, j_1\neq \cdots \neq j_k \in J\}.
\ee
Since $d_{k}(v,w) \le |v-w|$ one need not consider stitches such that $I_j \cap B(v, |v-w|) =\emptyset$.
Since the smocking stitches are a definite distance $\delta>0$ apart, there are only finitely many smocking stitches
such that $|_j\cap B(v, |v-w|) \neq \emptyset$.  Since each smocking stitch is compact, the infima over choices
of $z_i$ and $z'_i$ is also achieved.  Thus this infimum is a minimum as well.

Now we prove $d$ is symmetric.  First observe that for any $v,w\in {\mathbb{E}}^N$
\be
d_k(v,w)=d_k(w,v) 
\ee
because we can reverse the order of the segments:
\be
 |v-z_1|+\sum_{i=1}^k |z'_i-z_i|+|z'_k-w| =  |w-z'_k|+\sum_{i=1}^k |z'_i-z_i|+|z_1-v|.
 \ee
 Taking the minimum of symmetric $d_k$ we have:
 $
 \bar{d}(v,w)=\bar{d}(w,v),
 $
 and thus 
 \be
 d(\pi(v), \pi(w))= d(\pi(w), \pi(v)).
 \ee
 This suffices to prove symmetry since  the smocking map $\pi$ is surjective.
 
 To see that $d$ is definite, we again take advantage of the fact that the smocking map
is surjective, and refer to the points in $X$ as $\pi(v)$ and $\pi(w)$.
  By (\ref{k-delta}),
\be 
d(\pi(v),\pi(w))=0 \iff |v-w|=d_0(v,w)=d(v,w)=0 \iff v=w .
\ee
Note that $\bar{d}(v,w)=0$ when $v\neq w\in I_j$ which is why 
$\bar{d}$ is only a pseudometric.

To see that $d$ and $\bar{d}$ satisfy the triangle inequality, consider $y, v, w \in {\mathbb{E}}^N$.
As the minima are achieved, we know there exists stitches, $I_{j_1},..., I_{j_k}$
where $k=N(y,v)$,
and $z_i, z_i' \in I_{j_i}$ such that
\be
\bar{d}(y,v) = |y-z_1| +\sum_{i=1}^{k} |z'_i-z_i|+ |z'_{k}-v|.
\ee
There exists more stitches, $I_{j_{k+1}}, ... I_{j_{k+k'}}$, where $k'=N(v,w)$
and $z_i, z_i' \in I_{j_i}$ such that
\be
\bar{d}(v,w) = |v-z_{k+1}| +\sum_{i=k+1}^{k+k'} |z'_i-z_i|+ |z_{{k+k'}}-w|.
\ee
Adding these together and using the fact that 
\be
 |z'_{{k}}-v| + |v-z_{k+1}| \ge  |z'_{{k}} -z_{k+1}|, 
 \ee
 we have
\be 
\bar{d}(y,v) + \bar{d}(v,w) \ge   |y-z_1| + \sum_{i=1}^{k+k'} |z'_i-z_i| + |z_{{k+k'}}-w| \ge \bar{d}(y,w).
\ee
Applying the surjective smocking map we see that $d$ satisfies the triangle inequality.
\end{proof}

\subsection{ The Diamond Smocked Space $X_\diamond$ by Prof.~Sormani and Dr.~Kazaras}

In this section we study one of the most classic smocking patterns: diamond smocking (also called honeycomb smocking although there are no hexagons).  The pattern used to create diamond smocking is depicted in Figure~\ref{fig:pattern-diamond}.  It only appears to have diamonds after sewing the threads tight as seen in the same figure.  

\begin{figure}[h]
\includegraphics[width=.2 \textwidth]{pattern-diamond} \hspace{2cm}
\includegraphics[width=.2 \textwidth]{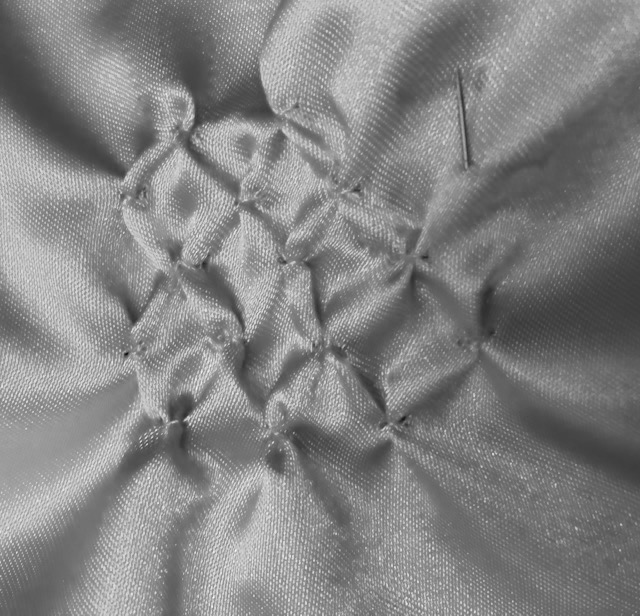}
\caption{The classical diamond smocking pattern will be used to define the smocking stitches of $(X_\diamond, d_\diamond)$.}
\label{fig:pattern-diamond}
\end{figure}

We need to define the smocking stitches to define this space rigorously: 

\begin{defn}\label{defn-diamond}
Our metric space $(X_\diamond, d_\diamond)$ is a smocked plane defined as in Definition~\ref{defn-smock}. 
We start with the Euclidean plane ${\mathbb{E}}^2$.
We define our index set (which will also be the center points of our stitches):
$$
J_\diamond\,\,\,=\,\,\,\{(j_1,j_2):  \, j_1=2n_1-n_2,\, j_2=n_2): \, n_1, n_2 \in {\mathbb{Z}} \}\,\, =
$$
$$
=\,\{(0,0), (\pm 2,0), (\pm 1, \pm 1), (0, \pm 2), (\pm 4,0), (\pm 3, \pm 1), ...\}.
$$
We define our stitches (which are all horizontal of unit length):
$$
I_{(j_1,j_2)}= [j_1-1/2, j_1+1/2] \times \{j_2\},
$$
as in Figure~\ref{fig:pattern-diamond}.
\end{defn}


\subsection{ The Ribbed Smocked Space, $X_=$ by Leslie and Shanell}

In this section we study one of the most classic smocking patterns: ribbed smocking (which is used to create fabric with a ribbed texture).  The pattern used to create ribbed smocking is depicted in Figure~\ref{fig:pattern=}.  It only appears to be ribbed after sewing the threads tight as seen in the same figure.  

\begin{figure}[h]
\includegraphics[width=.2 \textwidth]{pattern=} \hspace{2cm}  \includegraphics[width=.2 \textwidth]{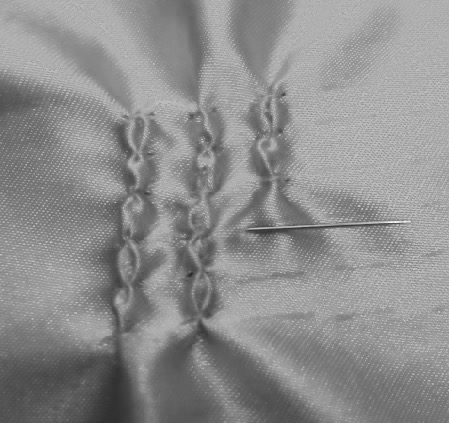}
\caption{The classical ribbed smocking pattern will be used to define the smocking stitches of $(X_=, d_=)$.}
\label{fig:pattern=}
\end{figure}

Let us begin by describing the stitches which will be used to create the smocked plane.
 
\begin{defn}\label{defn-=}
Our metric space $(X_=, d_=)$ is a smocked plane defined as in Definition~\ref{defn-smock}. 
We start with the Euclidean plane $\mathbb{E}^2$.
We define our index set:
$$
J_= = 2\mathbb{Z}\times\mathbb{Z}.
$$

We define our stitches: if $(j_1,j_2) \in J_=$, then
$$
I_{(j_1,j_2)}= [j_1-0.5,j_1+0.5]\times \{j_2\} 
$$
as in Figure~\ref{fig:pattern=}.
\end{defn}

\subsection{ The Woven Smocked Space $X_T$ by Julinda, Aleah, and Victoria:}

In this section we study one of the most classic smocking patterns: weaved smocking (which is used to create fabric with a basket weave texture).  The pattern used to create weaved smocking is depicted in Figure~\ref{fig:patternT}.  In the same figure we have sewn the threads to a point in a cloth.   In traditional smocking only the endpoints of these stitches
are joined and then the pattern looks like like a garden lattice.  However, here we have sewn each entire stitch to a point.  

\begin{figure}[h]
\includegraphics[width=.2 \textwidth]{patternT} \hspace{2cm} \includegraphics[width=.2 \textwidth]{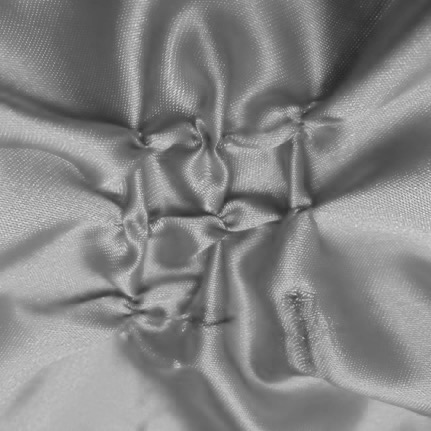}
\caption{This classic smocking pattern will be used to define the smocking stitches of $(X_T, d_T)$.}
\label{fig:patternT}
\end{figure}
 
Let us begin by describing the stitches which will be used to create the smocked plane.

\begin{defn}\label{defn-T}
Our metric space $(X_T, d_T)$ is a smocked plane defined as in Definition~\ref{defn-smock}. 
We start with the Euclidean plane ${\mathbb{E}}^2$.
We define our index set (which will also be the center points of our stitches):
$$
J_T = \{(j_1,j_2):\, j_1=2n_1,\, j_2=2n_2,\,\, n_1, n_2 \in {\mathbb{Z}}\} =2  {\mathbb{Z}}\times 2  {\mathbb{Z}}.
$$
We define our horizontal stitches (of length 2)
$$
I_{(j_1,j_2)}= [j_1-1,j_1+1]\times \{j_2\} \textrm{ when } (j_1+j_2)/2 \textrm{ is even,}
$$
and our vertical stitches (of length 2):
$$
I_{(j_1,j_2)}= \{j_1\}\times [j_2-1,j_2+1] \textrm{ when } (j_1+j_2)/2 \textrm{ is odd,}
$$
as in Figure~\ref{fig:patternT}.
\end{defn}

\subsection{ The Flower Smocked Space $X_+$ by Emilio, Moshe, and Ajmain:} 

In this section we study one of the most classic smocking patterns: flower smocking (which is used to create fabric with
flowers).  The pattern used to create flower smocking is depicted in Figure~\ref{fig:pattern-+}.  Each stitch appears to be a 
flower once it is sewn as seen in the same figure.  Note that the 
smocking stitches here are compact sets formed by overlapping pairs of stitches.

\begin{figure}[h]
\includegraphics[width=.2 \textwidth]{pattern+}\hspace{2cm} \includegraphics[width=.2 \textwidth]{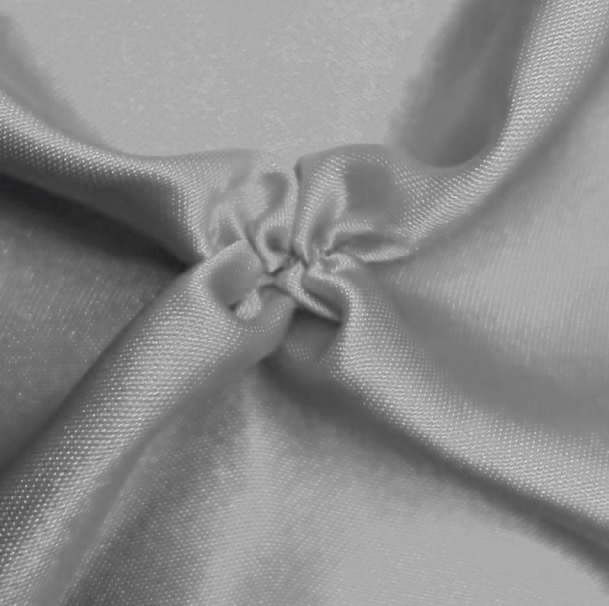}
\caption{This smocking pattern will be used to define the smocking stitches of $(X_+, d_+)$.}
\label{fig:pattern-+}
\end{figure}

Let us begin by describing the stitches which will be used to create the smocked plane.

\begin{defn}\label{defn-+}
Our metric space $(X_+, d_+)$ is a smocked plane defined as in Definition~\ref{defn-smock}. 
We start with the Euclidean plane ${\mathbb{E}}^2$.
We define our index set (which will also be the center points of our stitches):
$$
J_+= \{(j_1,j_2):\, j_1=3n_1,\, j_2=3n_2,\,\, n_1, n_2 \in {\mathbb{Z}}\}.
$$
We define our stitches (which are $+$ shapes):
$$
I_{(j_1,j_2)}= \left( [j_1-1, j_1+1]\times \{j_2\} \right)\,\,\, \cup \,\,\,\left( \{j_1\} \times [j_2-1, j_2+1] \right)
$$
as in Figure~\ref{fig:pattern-+}.
\end{defn}

\subsection{ The Checkered Smocked Space, $X_H$, by David and Vishnu} 

In this section we study checkered smocking (which is used to create fabric with alternating diamonds puffed up and down).  The pattern used to create checkered smocking is depicted in Figure~\ref{fig:patternH}.  It only appears to be checkered after sewing the threads tight as seen in the same figure.  

\begin{figure}[h]
\includegraphics[width=.2 \textwidth]{patternH} \hspace{2cm} \includegraphics[width=.2 \textwidth]{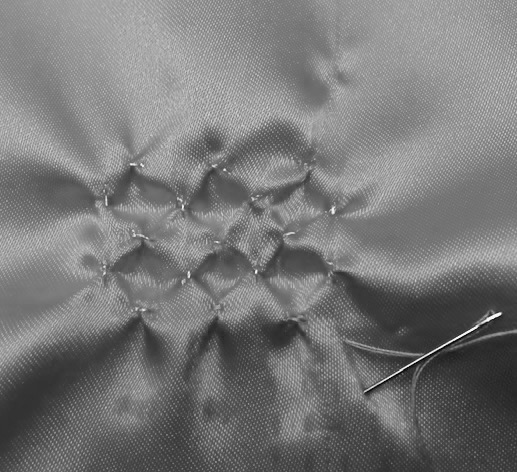}
\caption{This smocking pattern will be used to define the smocking stitches of $(X_H, d_H)$.}
\label{fig:patternH}
\end{figure}

Let us begin by describing the stitches which will be used to create the smocked plane $X_H$.
 
\begin{defn}\label{defn-H}
Our metric space $(X_H, d_H)$ is a smocked plane defined as in Definition~\ref{defn-smock}. 
We start with the Euclidean plane ${\mathbb{E}}^2$.
We define our index set in two parts:
\be
J_H = J_H^- \cup J_H^\vert
\ee
where
\be
J_H^-= 3\mathbb{Z} \times 3\mathbb{Z}
\ee
and
\be
J_H^\vert=(3\mathbb{Z} +1.5) \times (3\mathbb{Z} + 1.5).
\ee
We define our stitches: if $(j_1,j_2) \in J_H^-$, $I_{(j_1,j_2)}$ is the horizontal segment
\be
I_{(j_1,j_2)}= [j_1-0.5,j_1+0.5]\times \{j_2\} 
\ee
and if $(j_1,j_2) \in J_H^\vert$, $I_{(j_1,j_2)}$ is the vertical segment
\be
I_{(j_1,j_2)}= \{j_1\}\times [j_2-0.5,j_2+0.5]
\ee
as in Figure~\ref{fig:patternH}.
\end{defn}

\subsection{ The Bumpy Smocked Space $X_\square$ by Maziar and Hindy:}

In this section we study bumpy smocking (which is used to create fabric with a bumpy texture).  The pattern used to create bumpy smocking is depicted in Figure~\ref{fig:pattern-square}.  It only appears to be bumpy after sewing the threads tight as seen in the same figure. 

\begin{figure}[h]
\includegraphics[width=.2 \textwidth]{pattern-square} \hspace{2cm} \includegraphics[width=.2 \textwidth]{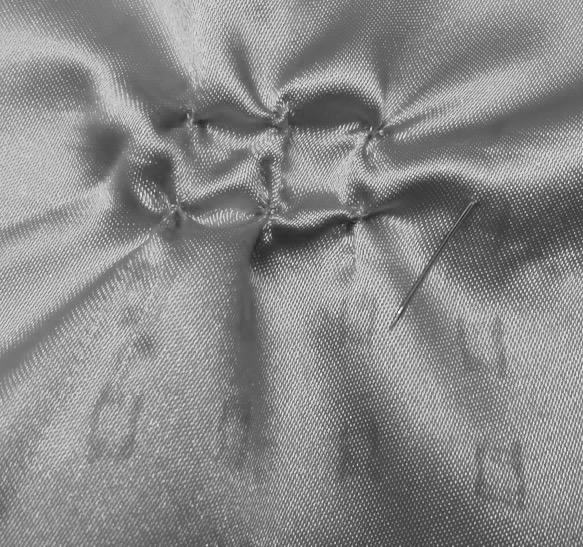}
\caption{This smocking pattern will be used to define the smocking stitches of $(X_\square, d_\square)$.}
\label{fig:pattern-square}
\end{figure}

Let us begin by describing the stitches which will be used to create the smocked plane.

\begin{defn}\label{defn-square}
Our metric space $(X_\square, d_\square)$ is a smocked plane defined as in Definition~\ref{defn-smock}. 
We start with the Euclidean plane ${\mathbb{E}}^2$.
We define our index set (which will also be lower left corners of our squares) by:
\be
J_\square= \{ (j_1, j_2):\,\, j_1=3n_1,\, j_2=3n_2,\,\, n_1, n_2 \in {\mathbb{Z}}\}.
\ee
We define our stitches (which are squares of unit side lengths) by:
\begin{eqnarray*}
I_{(j_1,j_2)}&=&  \left([j_1, j_1+1] \times \{j_2\} \right) \,\, \cup \,\,  \left([j_1, j_1+1] \times \{j_2+1\} \right)\\
&& \quad \cup \left( \{j_1\}\times [j_2, j_2+1] \right) \,\,\cup \,\, \left( \{j_1+1\}\times [j_2, j_2+1] \right)
\end{eqnarray*}
as in Figure~\ref{fig:pattern-square}.
\end{defn}

\section{The Smocking Constants: Depth, Lengths, and Separation Factor} \label{subsect-consts}

In this section we describe some essential parameters that we can later use to
estimate the balls in smocked spaces and their distance functions.  After providing the definitions,
we have subsections in which we find their values for the six smocked spaces that we have just defined above.

\subsection{ Defining Depth and Lengths by Prof.~Sormani and Maziar}

Given a smocked space, $(X,d)$, as in Definition~\ref{defn-smock}, with smocking
stitches $\{I_j:\, j \in J\}$ and smocking set $S=\bigcup_{j\in J} I_j$ we make the following
definitions:

\begin{defn}\label{defn-distance-to-smocking set}
$D: {\mathbb{E}^N}\to [0,\infty)$, the distance of a point $x\in \mathbb{E}^N$ to the smocking set is defined to be
\be
 D(x)= \min\{ |x-z|: \, z\in I_j, \, j \in J \}.
\ee
\end{defn}

\begin{lem}\label{lem-distance-to-smocking set}
The minimum in Definition~\ref{defn-distance-to-smocking set} is achieved.
\end{lem}

\begin{proof}  
For any $x\in S$, $D(x)=0$ and hence is achieved. Now, fix $x\in \mathbb{E}^N\setminus S$. For any $r>0$, the number of stitches inside $B(x,r)$ is finite. This is because the smocking separation factor $\delta$ is always positive, i.e.
\be \label{s-factor}
\delta=\min\{|z-z'|: \, z\in I_j, \, z'\in I_{j'},\, j\neq j' \in J\} >0.
\ee
and only finitely many balls of radius $\delta$ cover $B(x,r)$, because the closure of $B(x,r)$ is compact. 

Let $K$ be the index set of stitches that intersect $B(x,r)$. For each $k\in K$, let 
\be
D_k(x)= \inf\{|x-z|: \, z\in I_k\}.
\ee 
There exists ${z_k}_n \in I_k$ approaching this infimum, such that $|x-{z_k}_{n+1}|<|x-{z_k}_n|$ for all $n$. Since each stitch is compact, $({z_k}_n)_n$ has a convergent subsequence, and hence there is a point $z_k\in I_k$ that achieves the infimum for $D_k(x)$. Then, $D(x)= \min\{D_k(x): \, k\in K\}$ is achieved since $K$ is finite.
\end{proof}

\begin{defn} \label{defn-smocking-depth}
The smocking depth, $h$, is defined to be
\be
h=\inf\{ r:\, \mathbb{E}^N \subset T_r(S)\} \in [0,\infty],
\ee
which by definition of tubular neighborhood is 
\be
h=\inf\{ r:\, \forall x \in X \,\,\exists j \in J \, \exists z \in I_j \,\,s.t.\,\, |x-z|<r\}.
\ee
\end{defn}

\begin{lem} \label{lem-smocking-depth}
If the smocking depth, $h$, is finite, then 
\be
h= \sup \{ D(x): \, x\in \mathbb{E}^N\}.
\ee
\end{lem}

\begin{proof}  
Since the smocking depth is finite, there exists
 $r\in(0,\infty)$ such that $\mathbb{E}^N\subset T_r(S)$.  We claim
\be 
D(x)\leq r \quad \forall x\in \mathbb{E}^N.
\ee 
Otherwise, $D(x)>r$ for some $x\in \mathbb{E}^N$.  So the minimum distance of $x$ to a stitch is strictly more than $r$. This implies $x\notin T_r(S)$ by definition of tubular neighborhood, which is a contradiction. 

Taking the supremum of the LHS and infimum of the RHS of the inequality, we have
\be 
\sup\{D(x): \, x\in \mathbb{E}^N\}\leq \inf\{r:\, \mathbb{E}^N \subset T_r(S)\}.
\ee
Let $r'=\sup\{D(x): \, x\in \mathbb{E}^N\}$. 
Then, $\mathbb{E}^N \subset T_{r'}(S)$, since 
\be
\forall x\in \mathbb{E}^N, \exists j\in J, \exists z \in I_j, \mbox{ such that } |x-z|=D(x)\leq r'.
\ee
Therefore, it is not possible that $r'<r$, $r$ being the infimum. 
Thus, $r'=h$ as defined in Definition~\ref{defn-smocking-depth}.   
\end{proof}

\begin{defn}\label{defn-smocking-lengths}
The smocking lengths are defined
either using the lengths of intervals
\begin{eqnarray}
L_{min}&=& \inf \{ L(I_j): \, j \in J\} \in [0,\infty)\\
L_{max}&=& \sup \{ L(I_j): \, j \in J\} \in (0,\infty]
\end{eqnarray}
and if $L_{min}=L_{max}$ we call this the smocking length.
If the smocking stitches are not intervals we replace 
length with diameter in the above.
\end{defn}

\begin{defn} \label{defn-sep}
The smocking separation factor, $\delta=\delta_X$, is defined to be
\be
\delta_X =  \min\left\{ |z -w|: \,\, z\in  I_j, \, w\in I_k, \, j\neq k \in J\right\}.
\ee
\end{defn}

\begin{lem}\label{lem-hplusL}
If a smocked spaces is parametrized by points in stitches as in (\ref{param-by-points}),
then
\be
{\mathbb{E}}^N \subset T_{h+L}(S),
\ee
where $S = \bigcup_{j\in J} I_j$ is the smocking set and $h$ is the smocking depth
and $L=L_{max}$ is the maximum smocking length.
\end{lem}

\begin{proof}
Given any $x\in {\mathbb{E}}^N$, by the definition of smocking depth,
we have a closest point in a closest stitch, $z\in I_j$, such that
\be
d(x, I_j)=\bar{d}(x, z)=|x-z| \le h. 
\ee
Since our smocked space is parametrized by points in stitches
we have
\be
|z-j|\le L .
\ee
So
\be
|x-j|\le L+h.
\ee
Thus $x \in T_{L+h}(S)$.
\end{proof}

\subsection{ The Smocking Constants of $X_\diamond$  by Prof.~Sormani and Maziar}   

Here we find the smocking constants for the Diamond Smocking Space:

\begin{lem} \label{lem-depth-diamond} 
The smocking depth
\be
h_\diamond=\inf\{ r:\, \mathbb{E}^N \subset T_r(S)\} \in [0,\infty] =\frac{5}{8}.
\ee
\end{lem}

\begin{proof} 
Note that 
\be 
\mathbb{E}^2 =\{[j_1-\tfrac{1}{2},j_1+\tfrac{1}{2}] \times[j_2-2,j_2]: j\in J\},
\ee
i.e. rectangles 
\be
R_j=[j_1-\tfrac{1}{2},j_1+\tfrac{1}{2}]\times [j_2-2,j_2]
\ee
 tile the Euclidean plane. Moreover, the rectangles $R_j$ have a reflective symmetry along the line connecting $(j_1-1,j_2-1)$ and $(j_1+1,j_2-1)$. Therefore, it is sufficient to consider the square 
 \be
 Q_j= [j_1-\tfrac{1}{2},j_1+\tfrac{1}{2}]\times [j_2-1,j_2].
 \ee
  The rectangle $R_j$ and the square $Q_j$ are shown in Figure \ref{Diamond_tile}. 

\begin{figure}[h!]
\centering
\includegraphics[width=.4 \textwidth]{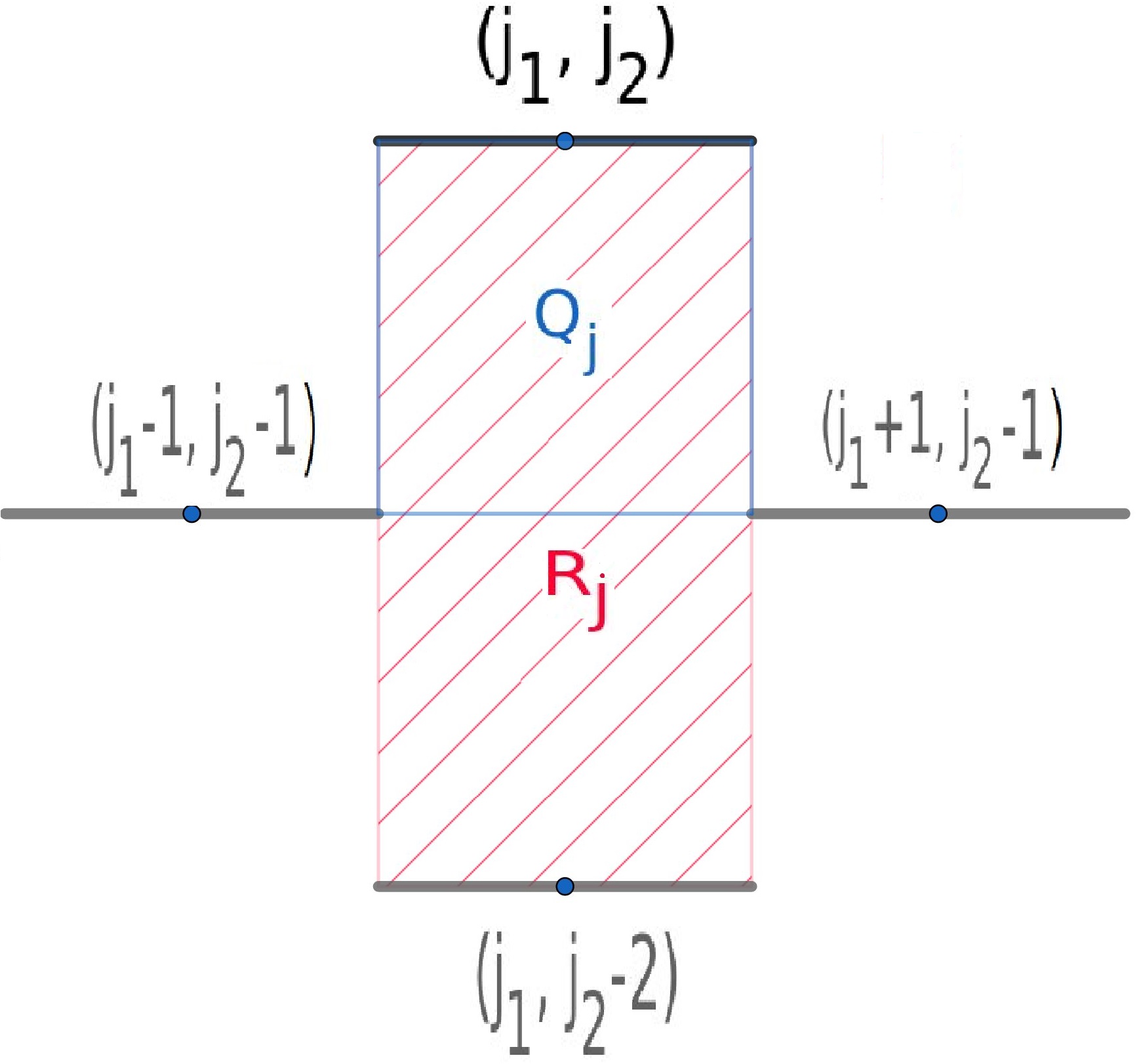}
\caption{$R_j$ and $Q_j$}
\label{Diamond_tile} 
\end{figure}

The tubular neighborhoods of radius $r$ of the diamond smocked space for four of the stitches is shown in Figure \ref{Diamond_tubular}.

\begin{figure}[h!]
\centering
\includegraphics[width=.4 \textwidth]{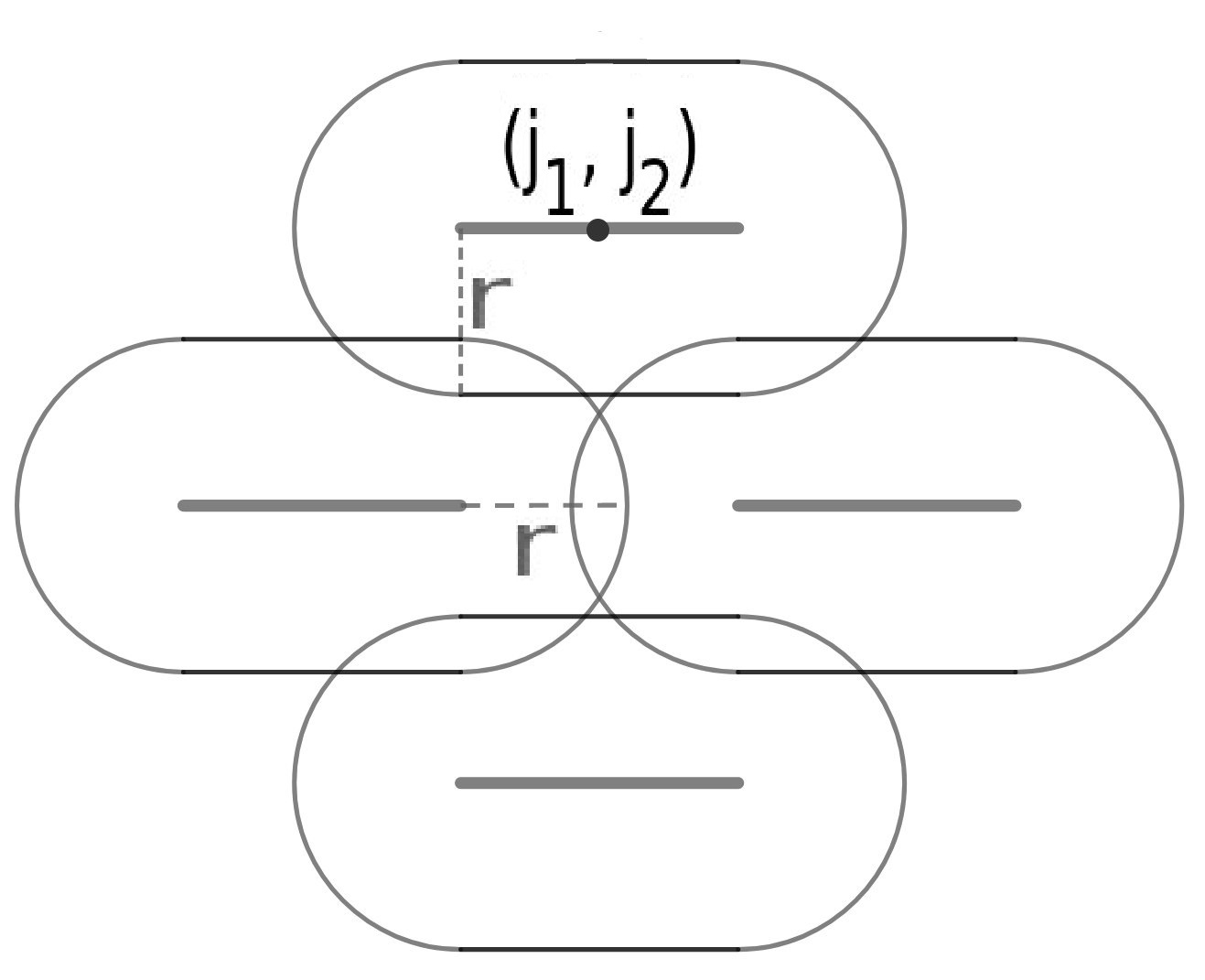}
\caption{Tubular neighborhoods in $X_\diamond$. }
\label{Diamond_tubular} 
\end{figure}

We see that the boundaries of the tubular neighborhoods that lie in $Q_j$ are 
\begin{enumerate}
    \item a length one horizontal segment from the tubular neighborhood of $I_{(j_1,j_2)}$, distance $r$ below it,
    \item a quarter of circle right of $I_{(j_1-1, j_2-1)}$ of radius $r$,
    \item a quarter of circle left of $I_{(j_1+1, j_2-1)}$ of radius $r$. 
\end{enumerate} 
The smallest $r$ for which the tubular neighborhoods of the stitches cover $Q_j$ is achieved when these boundary pieces intersect at a point $B$ as shown in Figure \ref{Diamond_smk_depth}. This is because otherwise the point, $B$, distance $r=\frac{5}{8}$ from the $(j_1,j_2)$ stitch and from the end point of the $(j_1+1,j_2-1)$ stitch would not be covered.

\begin{figure}[h!]
\centering
\includegraphics[width=.4 \textwidth]{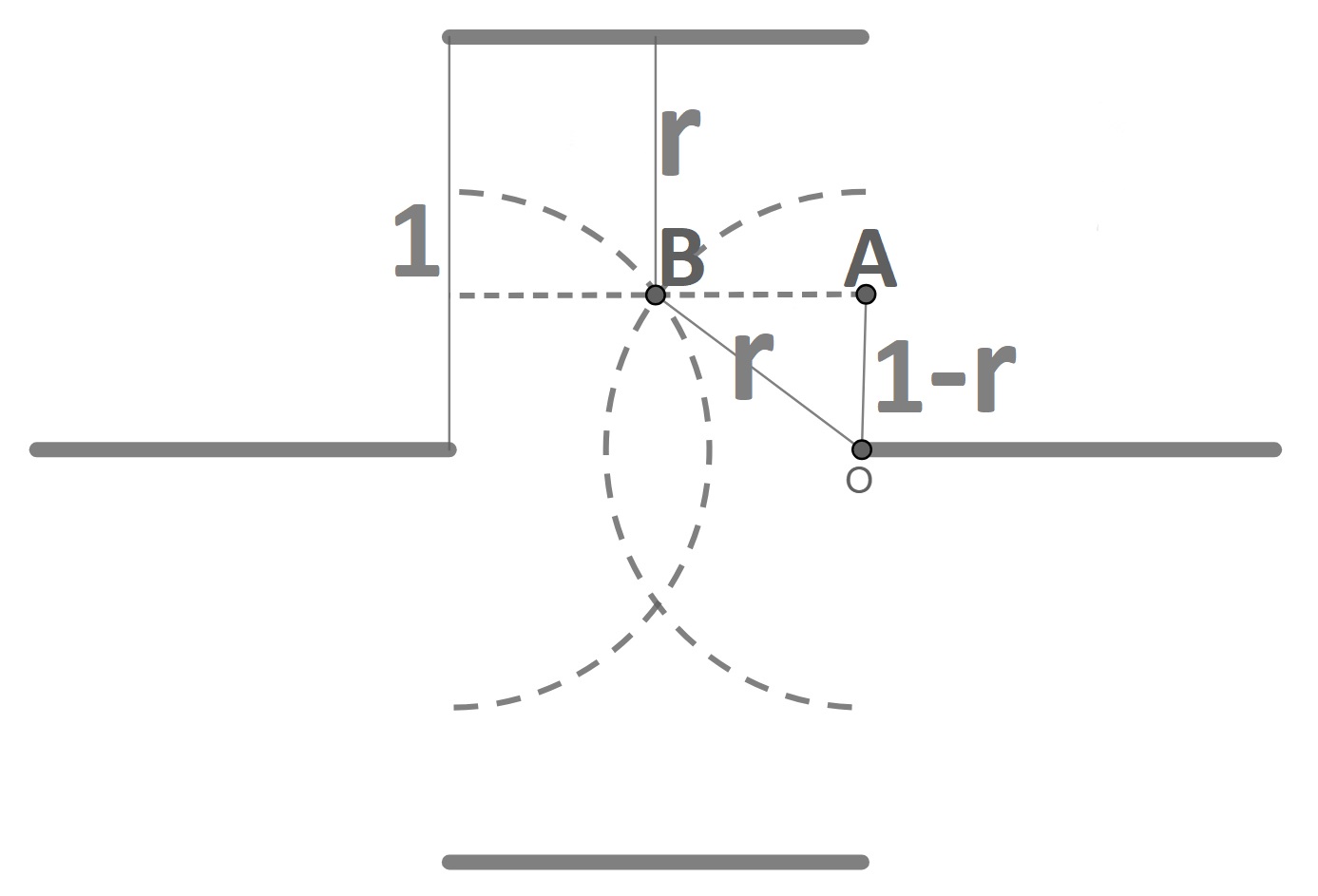}
\caption{Smallest covering tubular neighborhoods.}
\label{Diamond_smk_depth} 
\end{figure}

Consider the triangle $OAB$ in figure \ref{Diamond_smk_depth}, for which $AB=\frac{1}{2}L= \frac{1}{2}$, $OA= 1-r$, and $OB= r$. By the Pythagorean Theorem,
\be
0= \left(\tfrac{1}{ 2}\right)^2+ (1-r)^2 -r^2=\frac{1}{4} + 1-2r+r^2 -r^2=\frac{5}{4}-2r.
\ee
Therefore, $r=\frac{5}{8}$ is the smallest $r$ such that 
\be
R_j  \,\,\subset \,\,T_r(I_{(j_1,j_2)})\cup T_r(I_{(j_1+1, j_2-1)})\cup T_r(I_{(j_1, j_2-2)})\cup T_r(I_{(j_1-1, j_2-1)}).
\ee
\end{proof}

\begin{lem} \label{lem-length-diamond}
The smocking lengths are 
\begin{eqnarray}
L^\diamond_{min}&=& \inf \{ L(I_j): \, j \in J_\diamond\} =1\\
L^\diamond_{max}&=& \sup \{ L(I_j): \, j \in J\diamond\} =1
\end{eqnarray}
so the smocking length is $1$.
\end{lem}

\begin{proof}
All our stitches $I_j=\{j_1-1/2, j_1+1/2\}\times\{j_2\}$ have length 
\be
(j_1+1/2)-(j_1-1/2)=1.
\ee
\end{proof}

\begin{lem} \label{lem-delta-diamond}
The smocking separation factor is
\be
\delta_\diamond =1
\ee
\end{lem}

\begin{proof} 
Recall
\be
\delta_\diamond =\min\left\{ |z -w|: \,\, z\in  I_j, \, w\in I_k, \, j\neq k \in J\right\}.
\ee
If we let
\be
z_0 = (1/2, 0) \in I_{(0,0)} \in J_\diamond
\ee
and
\be
w_0=(1/2,1) \in I_{(1,1)} \in J_\diamond
\ee
 then
we see that
\be
\delta_\diamond \le |z_0-w_0|= 1.
\ee

On the other hand, taking any $j \neq k$ we consider the following cases:

Case I: $j_2\neq k_2$.   Then by the definition of $J_\diamond$,  $|j_2-k_2|\ge 1$.
For any $z\in  I_j, \, w\in I_k$,
\be
z_2=j_2 \textrm{ and }w_2=k_2
\ee
by the definitions of our stitches.  Thus
\be
|z-w| \ge |z_2-w_2|=|j_2-k_2| \ge 1.
\ee

Case II: $j_2= k_2$ and so $j_1\neq k_1$.  By the definition of $J_\diamond$,  $|j_2-k_2|\ge 2$.
For any $z\in  I_j, \, w\in I_k$,
\be
z_1\in [j_1-1/2, j_1+1/2] \textrm{ and } w_1\in [k_1-1/2, k_1+1/2]
\ee
by the definitions of our stitches.  Thus
\be
|z-w|\ge |z_1-w_1| \ge 2 -(1/2)-(1/2) =1.
\ee
Combining our cases, we see the minimum is $\ge 1$.
\end{proof}

\subsection{ The Smocking Constants of $X_=$ by Julinda and Hindy}

Here we find the smocking constants for the $X_=$:

\begin{lem} \label{lem-depth-=}
The smocking depth
\be
h_=  =\inf\{ r:\, \mathbb{E}^N \subset T_r(S)\} \in [0,\infty] = \frac{\sqrt{2}}{2}.
\ee
\end{lem}

\begin{proof}
Let $A = [-\frac{1}{2}, \frac{5}{2}] \times [0,1] \subset \mathbb{E}^2.$ $A$ contains four smocking stitches: $I_0, I_{(0,1)},I_{(2,0)}, I_{(2,1)}.$ Denote the union of these smocking stitches by $S' \subset S$. Consider the point $a = (1, \frac{1}{2}) \in A.$ 
 We have that 
 \be
 d_{\mathbb{E}^2}(a, I_j) = \frac{\sqrt{2}}{2}
 \ee
  for all $I_j$ in $A$. Furthermore, $B_{\frac{\sqrt{2}}{2}}(a)$ contains just these four smocking stitches, so $a \notin T_r(S)$ for any $r < \frac{\sqrt{2}}{2}$. Therefore, $$h_= \geq \frac{\sqrt{2}}{2}.$$
 
 Note that $a$ is the center of the square $[\frac{1}{2}, \frac{3}{2}] \times [0,1]$ whose corners are the inner endpoints of the $I_j$. Since the center of a square is the farthest point from the corners, all other points in the square are included in $T_{\frac{\sqrt{2}}{2}}(S)$. Furthermore, it is clear that 
 \be
 \left(\left[-\frac{1}{2}, \frac{1}{2}\right] \times\left [0,1\right] \right) \cup \left(\left[\frac{3}{2}, \frac{5}{2}\right] \times \left[0,1\right]\right) \subset  T_{\frac{1}{2}}(S').
 \ee
 
 Therefore, $A \subset T_{\frac{\sqrt{2}}{2}}(S).$
 
Since our space is periodic in copies of $A$, the proof is complete.
\end{proof}

The next two lemmas are very easy to see:

\begin{lem} \label{lem-length-=}
The smocking lengths are 
\be
L=L^=_{min} = L^=_{max} =1.
\ee
\end{lem}

\begin{lem} \label{lem-delta-=}
The smocking separation factor is
\be
\delta_= = 1
\ee
\end{lem}

\subsection{ The Smocking Constants of $X_T$ by Dr.~Kazaras, Moshe and David :}

Recall Definition~\ref{defn-T} and Subsection~\ref{subsect-consts}.   Here we find the smocking constants for $X_T$. 

\begin{lem} \label{lem-depth-T}
The smocking depth of $X_T$ is
\be
h_T = 1
\ee
\end{lem}

\begin{figure}[h]
\includegraphics[width=.4 \textwidth]{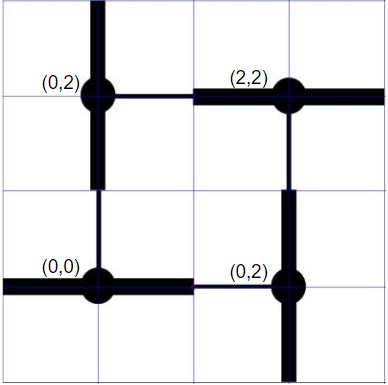}
\caption{The four smocking stitches that intersect with $[0,2]\times [0,2]$.}
\label{fig:depth-T}
\end{figure}

\begin{proof}   
Our space is invariant under translation by elements in the lattice $2\mathbb{Z}\times 2\mathbb{Z}$.  So we need only consider the point with the largest distance from any smocking stitches in $[0,2]\times [0,2]$.  
As one can see in Figure~\ref{fig:depth-T}, there
are four smocking stitches that intersect with this square:
\begin{eqnarray}
I_{(0,2)} = \{0\} \times [1,3] & \textrm{ and }& I_{(2,2)}=[1,3] \times \{2\} \\
 I_{(0,0)}=[-1,1] \times \{0\}  & \textrm{ and } &  I_{(2,0)} = \{2\}\times [-1,1]
 \end{eqnarray}
The point $(1, 1)$ in the center of this square has distance $1$ from these stitches, 
while the other points in the square are each closer to one of these four smocking stitches. 
In other words, the farthest distance possible is $1=h_T$.
\end{proof}

\begin{lem} \label{lem-length-T}
The smocking lengths of $X_T$ are 
\begin{eqnarray}
L^T_{min}&=& \inf \{ L(I_j): \, j \in J_T\} =2\\
L^T_{max}&=& \sup \{ L(I_j): \, j \in J_T\} =2
\end{eqnarray}
so the smocking length is $L_T = 2$.
\end{lem}

\begin{proof}
This follows from the fact that each interval $I_{(j_1,j_2)}$ has length $2$.
\end{proof}

\begin{lem} \label{lem-delta-T}
The smocking separation factor of $X_T$ is
\be
\delta_T =1
\ee
\end{lem}

\begin{proof}
As above, we can restrict our study to the square in Figure~\ref{fig:depth-T}.
The shortest distance between distinct stitches is realized by neighboring 
horizontal and vertical intervals, which are all a distance $1$ apart achieved
by a segment running from the center of one smocking interval to the end point
the other smocking interval.
\end{proof}

\subsection{ The Smocking Constants of $X_+$ by Shanell and Vishnu}   

Now we will compute the smocking constants for $(X_+, d_+)$.   

\begin{lem} \label{lem-depth-+}
The smocking depth of $(X_+, d_+)$ is
\be
h_+  =\inf\{ r:\, \mathbb{E}^N \subset T_r(S)\} =\sqrt {\tfrac{5}{2}}
\ee
\end{lem}

\begin{proof}
The point $(0,1) + (\frac{1}{2},\frac{3}{2})$ is distance
\be
d = \sqrt{\left(\tfrac{1}{2}\right)^2 + \left(\tfrac{3}{2}\right)^2} = \sqrt {\tfrac{5}{2}}
\ee 
away from each of the four stitches surrounding it.  Thus $h_+ \geq d$.  Notice for any other point in the square 
\be
\square = [0,3]\times[0,3],
\ee 
the distance to a stitch is less than $d$.  This is because we can partition $\square$ into $8$ right triangles, all sharing a vertex at $(0,1) + (\frac{1}{2},\frac{3}{2})$ and having some leg intersecting one of the $4$ surrounding stitches, so that the point $(0,1) + (\frac{1}{2},\frac{3}{2})$ is the furthest away from the side of the triangle which intersects a stitch, as in Figure \ref{fig:Partition+}.
\end{proof}

\begin{figure}[h]
\includegraphics[width=.2 \textwidth]{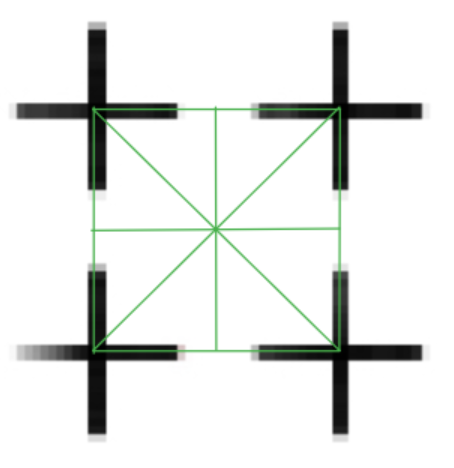}
\caption{The partition of the square $[0,3]\times[0,3]$ into $8$ triangles.}
\label{fig:Partition+}
\end{figure}

\begin{lem} \label{lem-length-+}
The smocking lengths are 
\begin{eqnarray}
L^+_{min}&=& \inf \{ L(I_j): \, j \in J_+\} =2\\
L^+_{max}&=& \sup \{ L(I_j): \, j \in J_+\} =2
\end{eqnarray}
\end{lem}

\begin{proof}
This follows from the fact that the diameter of a $+$ shaped stitch $I_j$ is $2$.
\end{proof}

\begin{lem} \label{lem-delta-+}
The smocking separation factor of $(X_+,d_+)$ is
\be
\delta_+ =1
\ee
\end{lem}

\begin{proof}
Notice the stitches $I_{(0,0)}$ and $I_{(3,0)}$ are a distance $1$ apart and the stitches $I_{(0,0)}$ and $I_{(3,3)}$ are a distance $\sqrt{2^2 + 2^2} = \sqrt 8$ apart. By the symmetry of $X_+$, $\delta_+ =1$.
\end{proof}

\subsection{ The Smocking Constants of $X_H$ by Victoria, Maziar, and Ajmain:}

Here we find the smocking constants for $X_H$:

\begin{lem} \label{lem-depth-=}
The smocking depth
\be
h_H =\inf\{ r:\, \mathbb{E}^2 \subset T_r(S)\} \in [0,\infty] = 1.5
\ee
\end{lem}

\begin{proof}
Note that the point $(0,1.5)$ is exactly distance $1.5$ from the four stitches around it: $I_{(0,0)}, I_{(0,3)}, I_{(-1.5,1.5)},$ and $I_{(1.5, 1.5)}$).  So $h_H \ge 1.5$.

Now we must show ${\mathbb{E}}^2 \subset T_{1.5}(J_H)$.   

For each $j\in J_-$ let 
\be
\square_j = [j_1 - 1.5, j_1. + 1.5]\times[j_2, j_2 + 3],
\ee
which is a 3x3 square above  $I_j$ that touches 
$I_{(j_1,J_2+3)}$, $I_{(j_1-1.5,j_2+1.5)}$ or $I_{(j_1+1.5, j_2+1.5)}$.
Let
\begin{eqnarray*}
\boxplus_j &=& [j_1-1.5,j_1-0.5]\times[j_2,j_2+1]\\
&&\sqcup\,\,\, [j_1+0.5,j_1+1.5]\times[j_2,j_2+1] \\
 &&\sqcup\,\,\, [j_1-1.5,j_1-0.5]\times[j_2+2,j_2+3]\\
 &&\sqcup\,\,\, [j_1+0.5,j_1+1.5]\times[j_2+2,j_2+3]. 
\end{eqnarray*} 
be four smaller unit squares in the corners of $\square_j$.

Notice 
\be
{\mathbb{E}}^2 = \bigcup_{j\in J_H^-} \square_j = \bigcup _{j\in J_H^-}  \boxplus_j \sqcup \left( \square_j \setminus \boxplus_j\right).
\ee
 Let $(x,y)\in {\mathbb{E}}^2$. Then there exists $j\in J_H^-$ such that 
 \be
 (x,y) \in  \boxplus_j \sqcup \left( \square_j - \boxplus_j\right).
 \ee
  Suppose $(x,y) \in  \boxplus_j$.  Since $\boxplus_j$ is a disjoint union of four  sets each with diameter $\sqrt 2$ and each intersecting $S$, we know 
  \be
  (x,y) \in T_{\sqrt2} (S) \subset T_{1.5}(S).
  \ee
Suppose $(x,y)\in \square_j - \boxplus_j$. Note that 
\begin{eqnarray*}
\square_j - \boxplus_j &=& [j_1-0.5, j_1+0.5]\times [j_2,j_2+1.5]\\
&\cup&[j_1-0.5, j_1+0,5]\times [j_2+1.5,j_2+3] \\
&\cup&[j_1,j_1+1.5]\times [j_2+1,j_2+2] \\
&\cup&[j_1-1.5, j_1]\times [j_2+1,j_2+2]. 
\end{eqnarray*}
Therefore, $(x,y)$ is in at most a $1.5$ distance from one of the stitches $I_{(j_1,j_2)}$, $I_{(j_1,J_2+3)}$, $I_{(j_1-1.5,j_2+1.5)}$ or $I_{(j_1+1.5, j_2+1.5)}$.  Thus
 \be
  (x,y) \in  T_{1.5}(S).
  \ee
\end{proof}

\begin{lem} \label{lem-length-H}
The smocking lengths are 
\be
L=L^H_{min}=L^H_{max}=1.
\ee
\end{lem}

\begin{proof}
For each $j\in J$ we have $L(I_j) = 1$.
\end{proof}

\begin{lem} \label{lem-delta-H}
The smocking separation factor is
\be
\delta_H =\sqrt 2
\ee
\end{lem}

\begin{proof}
By the translational and reflective symmetry of the H-smocking pattern, 
\be
\delta_H = \min \{|z-w| : z \in I_0, w\in I_j, j\in \left(J - \{0\}\right)  \cap \mathbb{R}_{\geq 0}^2 \}.
\ee
Suppose $j\in J_{H}^- \cap \mathbb{R}_+^2$.  Then 
$$
\min \{ |v-w| : v\in I_0 , w \in I_j \} = |(0.5, 0) - (j_1 - 0.5, j_2)| \geq |(0.5,0)-(3-0.5,3)| = \sqrt{13}.
$$
Suppose $j\in J_{H}^\vert \cap \mathbb{R}_+^2$.  Then 
$$
\min \{ |v-w| : v\in I_0 , w \in I_j \} = |(0.5, 0) - (j_1, j_2-0.5)| \geq |(0.5,0)-(1.5,1.5-0.5)| = \sqrt{2}
$$
Suppose $j\in J_H^- \cap \left( \{0\}\times \mathbb{R}_+ \right)$.  Then 
$$
\min \{ |v-w| : v\in I_0 , w \in I_j \} = |(0,0)-(0,j_2)| \geq |(0,0)-(0,3)| = 3.
$$
Suppose $j\in J_H^- \cap \left(\mathbb{R}_+ \times \{0\}\right)$.  Then 
$$
\min \{ |v-w| : v\in I_0 , w \in I_j \} = |(0.5,0)-(j_1-0.5,0)| \geq |(0.5,0)-(3-0.5,0)| = 2.
$$
\end{proof}

\subsection{ The Smocking Constants of $X_\square$ by Leslie, Emilio, and Aleah}

Here we find the smocking constants for $X_\square$.

\begin{lem} \label{lem-depth-square}
The smocking depth of $X_\square$ is
\be
h_\square = \sqrt{2}
\ee
\end{lem}

\begin{proof}
By the symmetry of the square lattice we need only examine $I_{(0,0)}, I_{(0,3)}, I_{(3,0)}, I_{(3,3)}$ it is clear that the point simultaneously farthest from each square is $(2,2)$. It is $\sqrt{2}$ distance from every square.
\end{proof}

\begin{lem} \label{lem-length-square}
\begin{eqnarray}
L^\square_{min}&=& \inf \{ L(I_j): \, j \in J_\square\} = \sqrt{2}\\
L^\square_{max}&=& \sup \{ L(I_j): \, j \in J_\square\} = \sqrt{2}
\end{eqnarray}
so the smocking length is $L_\square = \sqrt{2}$.
\end{lem}

\begin{proof}
For $X_\square$, $L(I_j) = \diam I_j$. In this case of unit squares, it follows that $L(I_j) = \sqrt{2}$. 
\end{proof}

\begin{lem} \label{lem-delta-square}
The smocking separation factor is
\be
\delta_\square = 2
\ee
\end{lem}

\begin{proof}
The minimum distance between any two squares is the plain Euclidean distance, which is $2$. To prove this, we can use the symmetry of the pattern and compute any example. Consider the minimum distance between $I_{(0,0)}$ and $I_{(3,0)}$, the distance would be the distance between $(x_1,y_1) = (1,0)$ and $(x_2,y_2) = (3,0)$ which is $2$.
\end{proof}

\section{Balls in Smocked Spaces}

In this section we examine how balls in smocked spaces look by describing their
preimages under the smocking map.  We begin with a few useful lemmas and 
propositions about balls and then draw some of the balls in our smocked spaces
in subsequent subsections.  

\subsection{Useful Facts about Balls in Smocked Spaces by Hindy and Moshe}

Recall $\pi: \mathbb{E}^N \to X$ of Definition~\ref{defn-smock}.   In order to describe the
balls in a smocked space precisely, we instead examine 
\begin{eqnarray}
\pi^{-1}\left(B_r(\pi(p))\right)&=&\left\{ x\in \mathbb{E}^N:\, \pi(x) \in  B_r(\pi(p)) \right\}\\
&=&\left\{ x\in \mathbb{E}^N:\, d(\pi(x),\pi(p))<r\right\} \\
&=& \left\{ x\in \mathbb{E}^N:\, \bar{d}(x, p) < r\right\}.
\end{eqnarray}

Our first lemma will be applied repeatedly within this section:

\begin{lem} \label{smock-h-1}
For all $p\in {\mathbb E}^N$ and all $r,s,>0$ we have
\be\label{eq-smock-h-1}
B_s(x) \subset \pi^{-1}(B_s(\pi(x))
\ee
and
\be
T_s(\pi^{-1}(B_r(\pi(x))))  \subset \pi^{-1}(B_{r+s}(\pi(x))).
\ee
\end{lem}

\begin{proof}
If $v \in B_s(x)$ then $|v-x|<s$ so $d(\pi(v),\pi(x))<s$ which implies $\pi(v) \in B_s(\pi(x))$
and we get (\ref{eq-smock-h-1}).
If $v \in T_s(\pi^{-1}(B_r(\pi(x))))$, there exists $z \in \pi^{-1}(B_r(\pi(x)))$ such that 
\be
d(\pi(v), \pi(z))  < s \textrm{ and } d(\pi(z), \pi(x))<r.
\ee
Then by the triangle inequality,
\be
d(\pi(v), \pi(x)) \leq d(\pi(v), \pi(z)) + d(\pi(z), \pi(x)) < s + r.
\ee
It follows that $v \in \pi^{-1}(B_{r+s}(\pi(x))).$ 
\end{proof}

In our first proposition we examine a small ball about a stitch point whose radius is
less than the separation factor of the smocked space:

\begin{prop} \label{ball-delta}
Suppose that $I_j$ is a smocking stitch and $r<\delta_X$ as defined in Definition~\ref{defn-sep}, 
then
\be
\pi^{-1}(B_{r}(\pi(I_j)))=T_r(I_j)
\ee
\end{prop}

\begin{proof}
For each $v \in T_r(I_j),$ there exists some $z \in I_j$ such that $|v - z| < r.$ Therefore,
\be \label{tube-inclusion}
d(\pi(v), \pi(I_j)) \leq d_0(v, z) = |v-z| < r
\ee
so $v \in \pi^{-1}(B_{r}(\pi(I_j))).$ 

To prove the converse, consider $v \in \pi^{-1}(B_{r}(\pi(I_j))).$ 
Setting $p=\pi(v)$ we have $d(p, \pi(I_j)) < r.$ 
This implies that
\be
d(p, \pi(I_j)) = \min \{d_0(y,y'),  d_1(y,y') : \, y\in \pi^{-1}(p),\,\, y'\in I_j \} 
\ee
since 
\be
d_n(y, y') \geq \delta_X > r \textrm{ for } n \ge 2.
\ee
By definition, 
\be
d_1(y, y') =\min \{|y-z_1|+|z'_1-y'|:\, z_1, z_1'\in I_{k}\}
\ee
for some $I_{k}.$    If $d_1(y, y') < r < \delta_X$, then $I_k = I_j$, so it follows that 
\be
d_1(y,y')  \geq  \min\{|y - z_1| \ : z_1 \in I_j \} =  d_0(y,y').
\ee
But also
\be
d_1(y,y')  \leq \min\{|y - z_1| + |y' - y'| \ : z_1 \in I_j \} =  d_0(y,y') 
\ee

Therefore
\begin{eqnarray}
d(p, \pi(I_j)) &=& \min \{d_0(y, y'):\, y \in \pi^{-1}(p), \,y'\in I_j\} \\
&=& \min \{|y - y'|:\, y \in \pi^{-1}(p), \, y'\in I_j\} 
\end{eqnarray}
So
\be
\exists y'\in I_j \,\exists y \in \pi^{-1}(p) \textrm{ such that } |y - y'| < r.
\ee
  Since $r < \delta_X$, it follows that $y \notin S$.    Thus 
  \be
  \pi(v)=p=\pi(y) \notin \pi(S).
  \ee
Therefore, $v=y$, and 
\be
\exists y'\in I_j \textrm{ such that } |v - y'| < r.
\ee
So $v\in T_r(I_j)$.
\end{proof}

\begin{prop} \label{ball-D(x)}
Suppose that $x\in {\mathbb{E}}^N\setminus S$ and $r<D(x)$ is defined as in Definition~\ref{defn-sep}, 
then
\be
\pi^{-1}(B_{r}(\pi(x)))=B_r(x)=\{y:\, |x-y|<r\}.
\ee
\end{prop}

\begin{proof} 
Note that $v\in \pi^{-1}(B_r(\pi(x)))$ iff
$
\pi (v)\in B_r(\pi (x)).
$
By the definition of the smocking distance, this is true iff
\be \label{ball-D(x)-1}
\min\{d_0(\pi(v),\pi(x)), d_1(\pi(v),\pi(x)), d_2(\pi(v),\pi(x)),...\} =d(\pi(v), \pi(x)) < r < D(x).
\ee
By the hypothesis $|z-x|\ge D(x) \quad \forall z\in S$, so by the definition
of smocking distance for all $j \ge 1$
\be
d_j(\pi(v),\pi(x)) \ge \min \{ |z-x|: \, z\in S\} \ge D(x).
\ee
Thus the minimum in (\ref{ball-D(x)-1}) is achieved by 
\be
d_0(\pi(v),\pi(x))=d(\pi(v), \pi(x)) < r < D(x).
\ee
So (\ref{ball-D(x)-1}) holds iff $|v-x| < r$ which is true iff $x\in B_r(p)$.
\end{proof}

In our next proposition we explore how a ball grows when there are no smocking stitches
too close to the original ball:

\begin{prop} \label{smock-h}
Suppose that the points in the smocking set are a definite distance,
\be
\delta_r= \min\left\{ |z -w|: \,\, z\in S\, \setminus \pi^{-1}(B_r(\pi(p))), \, \, w \in \pi^{-1}(B_r(\pi(p))) \right \}>0,
\ee
away from the points within the ball.
Then for all $r>0$ and for all $s \in (0,\delta_r]$
\be
\pi^{-1}\bigg(B_{r+s}(\pi(x))\bigg)=T_s\bigg(\pi^{-1}(B_r(\pi(x)))\bigg)
\ee
\end{prop}

\begin{proof} 
Suppose $v \in \pi^{-1}(B_{r+s}(\pi(x))).$ If $x \in  \pi^{-1}(B_{r}(\pi(x)))$ then clearly  
\be
v \in T_s\left(\pi^{-1}(B_{r}(\pi(x)))\right),
\ee
and we are done.  If not, note that the distance 
$
d(\pi(v), \pi(x)) < r+s,
$
is achieved by a collection of segments starting at 
$y \in \pi^{-1}(\pi(v))$ and ending at $y' \in \pi^{-1}(\pi(x))$: 
\be
d(\pi(v), \pi(x)) = |y-z_1|+\sum_{i=1}^{n} |z'_i-z_{i+1}|+|z'_k-y'| 
\ee
for some $n \geq 0$, where 
$\pi(z_i) = \pi(z'_i)$ for each $i$. 
Notice that for any $z_i'$ in the minimizing sum, we must have 
\be
d(\pi(v),\pi(x)) = d(\pi(v), \pi(z_i')) + d(\pi(z_i'), \pi(x)),
\ee
Since $y \notin \pi^{-1}(B_r(\pi(x)))$ and $y' \in \pi^{-1}(B_r(\pi(x))),$ at least one term in the sum must be of the form $|a- b|$ where $a \notin \pi^{-1}(B_r(\pi(x)))$ and $b \in  \pi^{-1}(B_r(\pi(x))).$ Since $\bar{d}(a, x) > r$ and $\bar{d}(b,x) < r,$ the line segment from $a$ to $b$ must hit a point $c$ such that $\bar{d}(c, p ) = d(\pi(c), \pi(x)) = r,$ by the continuity of the pseudometric $\bar{d}$.
It follows that 
\begin{eqnarray}
r + s &>& d_n(y, y') \\
&\geq& d(\pi(a), \pi(x)) \\
&=& |a - b| + d(\pi(b), \pi(x))\\
&=& |a - c| + |c - b| + d(\pi(b), \pi(x))\\
&=& |a -c| + d(\pi(c), \pi(x)) \\
&=& |a-c| + r.
\end{eqnarray}
This implies that $|a - c| < s.$ If $a \in S$ then $|a -c| \geq \delta_r \geq s$, which is a contradiction. Therefore, $a \notin S$ which implies that $a = v.$ Hence, $|v - c| < s$ which implies that 
\be
v \in T_s\left(\,\overline{\pi^{-1}(B_r(\pi(x)))}\,\right) = T_s\left(\pi^{-1}(B_r(\pi(x)))\right) .
\ee
The other direction holds by Lemma~\ref{smock-h-1}.
\end{proof}

\begin{lem} \label{smock-add-1}
In a smocked metric space as in Definition~\ref{defn-smock}.
If there is a stitch, $I_j$ such that
\be
\pi^{-1}(\bar{B}_r(\pi(x))) \cap I_j \neq \emptyset
\ee
then for all $s>0$
\be
T_s\bigg(\pi^{-1}(B_r(\pi(x)))\bigg) \cup T_s(I) \subset \pi^{-1}\bigg(B_{r+s}(\pi(x))\bigg)
\ee
\end{lem}

\begin{proof}
For each $v \in T_s(I_j),$ there exists some $z \in I_j$ such that $|v - z| < s.$ Therefore,
\be
d(\pi(v), \pi(I_j)) \leq d_0(v, z) = |v-z| < s
\ee
so
 \be \label{smock-add-1-line-1} 
 T_s(I_j)\subset \pi^{-1}(B_{s}(\pi(I_j))).
 \ee 
 By assumption we have 
\be
\pi(I_j)\in \bar B_r(\pi(x))
\ee 
which gives us
\be 
B_{s}(\pi(I_j)) \subset B_{r+s}(\pi(x)).
\ee  
By looking at the preimages we get
\be 
 \pi^{-1}(B_{s}(\pi(I_j))) \subset \pi^{-1}(B_{r+s}(\pi(x))).
 \ee 
 Combining this with line \ref{smock-add-1-line-1} we get 
\be  
T_s(I_j) \subset \pi^{-1}(B_{s}(\pi(I_j))) \subset \pi^{-1}(B_{r+s}(\pi(x))).
\ee
In order to show that 
\be
T_s(\pi^{-1}(B_r(\pi(x))))\subset \pi^{-1}(B_{r+s}(\pi(p)))
\ee
 we apply Lemma~\ref{smock-h-1}.
\end{proof}

\begin{prop} \label{smock-add}
If  we consider all the smocking stitches that just touch a
given ball:
\be
J_r= \left\{ j \in J: \,\, I_j \cap \pi^{-1}(\bar{B}_r(x)) \neq \emptyset \textrm{ and } I_j \cap \pi^{-1}(B_r(x)) = \emptyset\right\}
\ee
and the distance to the nearest smocking interval that does not touch this ball
\be
\bar{\delta}_r= \min\left\{ |z -w|: \,\, z\in S\, \setminus \pi^{-1}(\bar{B}_r(\pi(v))), \, \, w \in \pi^{-1}(\bar{B}_r(\pi(v))) \right \}>0,
\ee
then for all $s \le \bar{\delta}_r$ we have
\be
\pi^{-1}(B_{r+s}(\pi(v))) =
T_s\left(\pi^{-1}(B_r(\pi(v)))\right) \cup 
\bigcup_{j\in J_r} T_s(I_j).
\ee
\end{prop}

\begin{proof}
By applying Lemma \ref{smock-add-1} to each $j\in J_r$ we know that 
\be 
T_s\left(\pi^{-1}(B_r(\pi(v)))\right) \cup 
\bigcup_{j\in J_r} T_s(I_j)\subset \pi^{-1}(B_{r+s}(\pi(v))).
\ee
So we need only show that 
\be 
\pi^{-1}(B_{r+s}(\pi(v)))\subset T_s\left(\pi^{-1}(B_r(\pi(v)))\right) \cup 
\bigcup_{j\in J_r} T_s(I_j).
\ee
We want to show that for any $q\in B_{r+s}(\pi(v))$, that one of the following holds:
\begin{eqnarray*}
(A) \qquad&\exists    x'\in \pi^{-1}(B_r(\pi(v)))& \textrm{ such that } |x'-\pi^{-1}(q)|<s
\\
(B)\qquad  &\exists  j \in J_r, \,\, z\in I_j & \textrm{such that } |z-\pi^{-1}(q)| <s.
\end{eqnarray*}
First, consider $q\in (B_{r+s}(\pi(v)))\cap \pi(S)$.
Since $s< \bar \delta_r$ all of the smocking intervals intersecting with $\pi^{-1}(B_{r+s}(\pi(v)))$ must either be contained inside $\pi^{-1}(\bar B_{r}(\pi(v)))$ in which case (A) holds, or be one of the $I_j$ for some $j\in J_r$ in which case (B) holds.

Now consider $q\in(B_{r+s}(\pi(v)))\setminus \pi(S)$, which implies there exists a unique point $y\in {\mathbb  E}^N$
such that 
\be
\pi(y)=q
\textrm{ and }
d(p,\pi(v))=d(\pi(y), \pi(v))< r+s.
\ee
Since $\delta_X>0$, by Theorem \ref{thm-smock-metric} we know that  one of the following cases holds:
\begin{eqnarray*}
\textrm{Case I: }&& d(p,\pi(v))=d(\pi(y), \pi(v))=|v-y|\\
\textrm{Case II:}&&\exists\, I_j\in S \,\, \exists\, w,w'\in I_j \,s.t.\,\,\, d(q,\pi(v))=d(\pi(y), \pi(v))= |y-w| + |w'-v|.
\end{eqnarray*}
In  Case I  we see that 
  \be
 y \in T_s(B_r(v)) \subset T_s(\pi^{-1}(B_r(\pi(v))))
  \ee
  by Lemma~\ref{smock-h-1}.  So we have (A) because  
  \be
  \exists x'\in \pi^{-1}(B_r(\pi(v)))\textrm{ such that }
 |x'-\pi^{-1}(q)|=|x'-y|<s.
 \ee
 In Case II, 
 \be
 \exists I_j\in S \,\, w,w'\in I_j \, s.t. \, d(q,\pi(v))=d(\pi(y), \pi(v))= |y-w| + |w'-v|<r+s
 \ee
and the interval $I_j$ intersects with $B_r(\pi(v))$, so  $I_j\in B_r(\pi(v))$. Consider the point $z\in \bar B_r(\pi(p))$ which minimizes $d(q, z)$. Since the closest smocking interval lies inside $B_r(\pi(v))$ we know that $z$ is either inside $B_r(\pi(v))$ (in which case $y$ is as well) or $z$ is on the boundary. In the second case $d(p, z)=r$ so 
\be
d(q, z)=|y-\pi^{-1}(z)|<s.
\ee 
 If the interval $I_j$ does not intersect with $B_r(\pi(p))$ then it must be in $J_r$. Then since 
 \be 
 d(q, \pi(v))=d(q, I_j)+d(I_j, \pi(v))<r+s
 \ee 
 and $d(I_j, \pi(v))=r$ we have $ d(q, I_j)<s$ which gives us that $\pi^{-1}(q)\in T_s(I_j)$and so we have completed the proof of (B).
\end{proof}


\subsection{ Exploring the Balls in $X_\diamond$ by Prof.~Sormani and Dr.~Kazaras:}

Here we consider balls in $X_\diamond$ centered on the point,
\be
p_0=I_{(0,0)}=[-1/2, 1/2]\times \{0\},
\ee  
by drawing their preimages $\pi^{-1}(B_{R}(p_0))\subset {\mathbb{E}}^N$.  See Figure~\ref{fig:diamond-balls}.

\begin{figure}[h]
\label{fig:diamond-balls}
\includegraphics[width=.4 \textwidth]{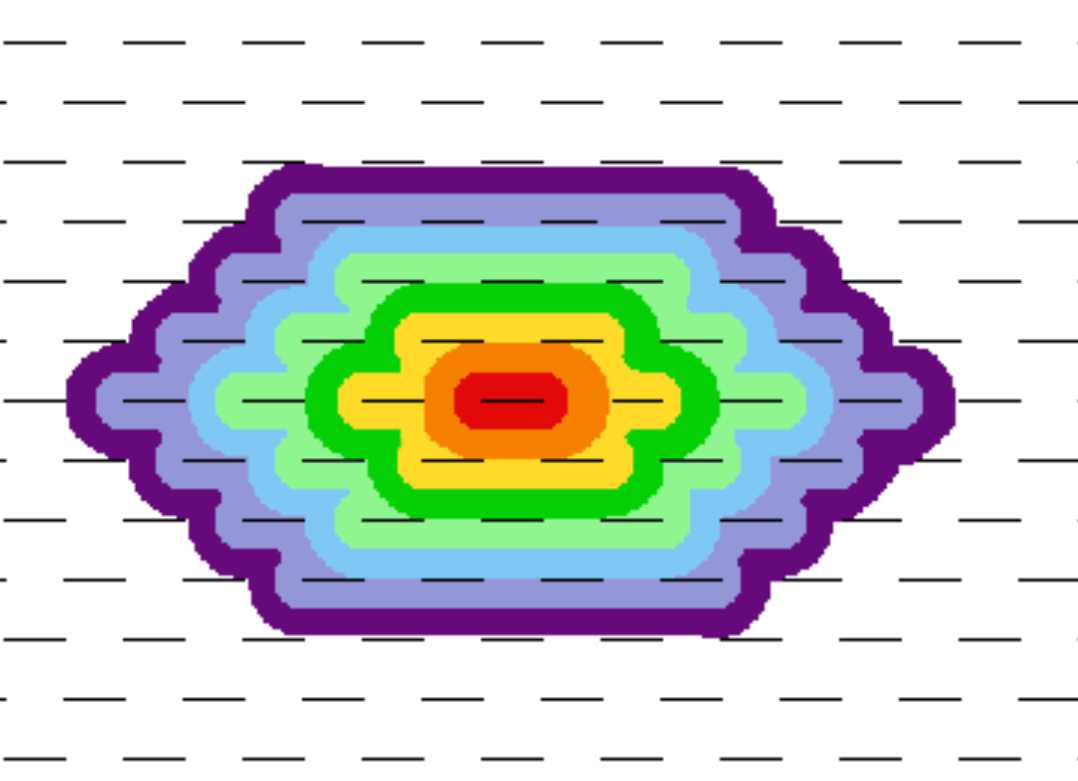}
\caption{Lifts of balls, $B_{R}(p_0)$, where $p_0=I_{(0,0)}$ in the smocked space $(X_\diamond, d_\diamond)$.}
\end{figure}

Observe that by  Lemma~\ref{ball-delta} we know that
\be
\pi^{-1}(B_{R}(p_0))=T_R(I_{(0,0)}) \quad \forall R\in (0,1],
\ee
because 
the smocking separation factor was proven in Lemma~\ref{lem-delta-diamond} to be
$
\delta_\diamond =1.
$
Thus the ball of radius $R=1/2$ is depicted in red  and of radius $R=1$ is depicted in orange in Figure~\ref{fig:diamond-balls}.

We next apply Proposition~\ref{smock-add}, keeping in mind that $\delta_\diamond=1$ and observing that
when $r=1$ there are six stitches touching the ball of radius $r=1$.  This gives us the ball of radius $3/2$
depicted in yellow in Figure~\ref{fig:diamond-balls}:
\be
\pi^{-1}(B_{3/2}(p)) =
T_{1/2} \left(\pi^{-1}(B_1(p))\right) \cup 
\bigcup_{j\in J_1} T_{1/2}(I_j)
\ee
where 
\be
J_1=\{(1,1), (2,0), (1,-1), (-1,-1), (-2,0), (-1,1)\}.
\ee
We then apply Proposition~\ref{smock-h} with $r=3/2$ and $s=1/2$ to find the ball of radius $2$
depicted in green in  Figure~\ref{fig:diamond-balls}:
\be
\pi^{-1}(B_{2}(p))=T_{1/2}(\pi^{-1}(B_{3/2}(p))).
\ee
We next apply Proposition~\ref{smock-add},  observing that
when $r=2$ there are twelve stitches touching the ball of radius $r=2$:
\be
J_2=\{(0,\pm 2), (\pm 2,\pm 2), (\pm 3, \pm 1), (\pm 4, \pm 0) \}.
\ee
  This gives us the ball of radius $5/2$
depicted in pale green in Figure~\ref{fig:diamond-balls}:
\be
\pi^{-1}(B_{5/2}(p)) =
T_{1/2} \left(\pi^{-1}(B_2(p))\right) \cup 
\bigcup_{j\in J_2} T_{1/2}(I_j).
\ee
We continue in this matter to complete the drawing in Figure~\ref{fig:diamond-balls} by eye.

While it can be rather complicated to provide formulas for these sets, we can
nevertheless approximately describe their shapes. For large radii,
these balls appear to have a hexagonal shape to them.  
 In fact it appears approximately to be the intersection of the strip, $y^{-1}(-r,r)$ with the diamond:
\be \label{diamond-set}
 \{ (x,y):\, \, |x|+|y|<R\} \textrm{ which has vertices at } (\pm R, 0) \textrm{ and } (0, \pm R)
\ee
where $R=2r$.  There is an error in this approximation of about the length of a smocking interval.  

We conjecture  that
$
J_r= \bar{J}_r \cup \hat{J}_r
$
where
\begin{eqnarray*}
\bar{J}_r&=&\{ (j_1, j_2) \in J_\diamond:\,\, |j_2|=r \textrm{ and } |j_1|\le |j_2|\}\\
\hat{J}_r&=& \{ (j_1, j_2) \in J_\diamond:\,\, |j_1+j_2| \in [2r-1, 2r] \}
\end{eqnarray*}

Due to the lengthiness of the proof required to rigorously prove this guess is true, 
we postpone studying this particular smocked space further in this paper.
See \cite{SWIF-smocked} for a rigorous study of this space.

\subsection{ Exploring the Balls in $X_T$ by Julinda, Aleah, and Victoria}

Here we consider balls in $X_T$ centered on the point,
\be
p_0=I_{(0,0)}=[-1, 1]\times \{0\},
\ee  
by drawing their preimages $\pi^{-1}(B_{R}(p_0))\subset {\mathbb{E}}^N$.  See Figure~\ref{fig:ballsT}.

\begin{figure}[h]
\includegraphics[width=.4\textwidth]{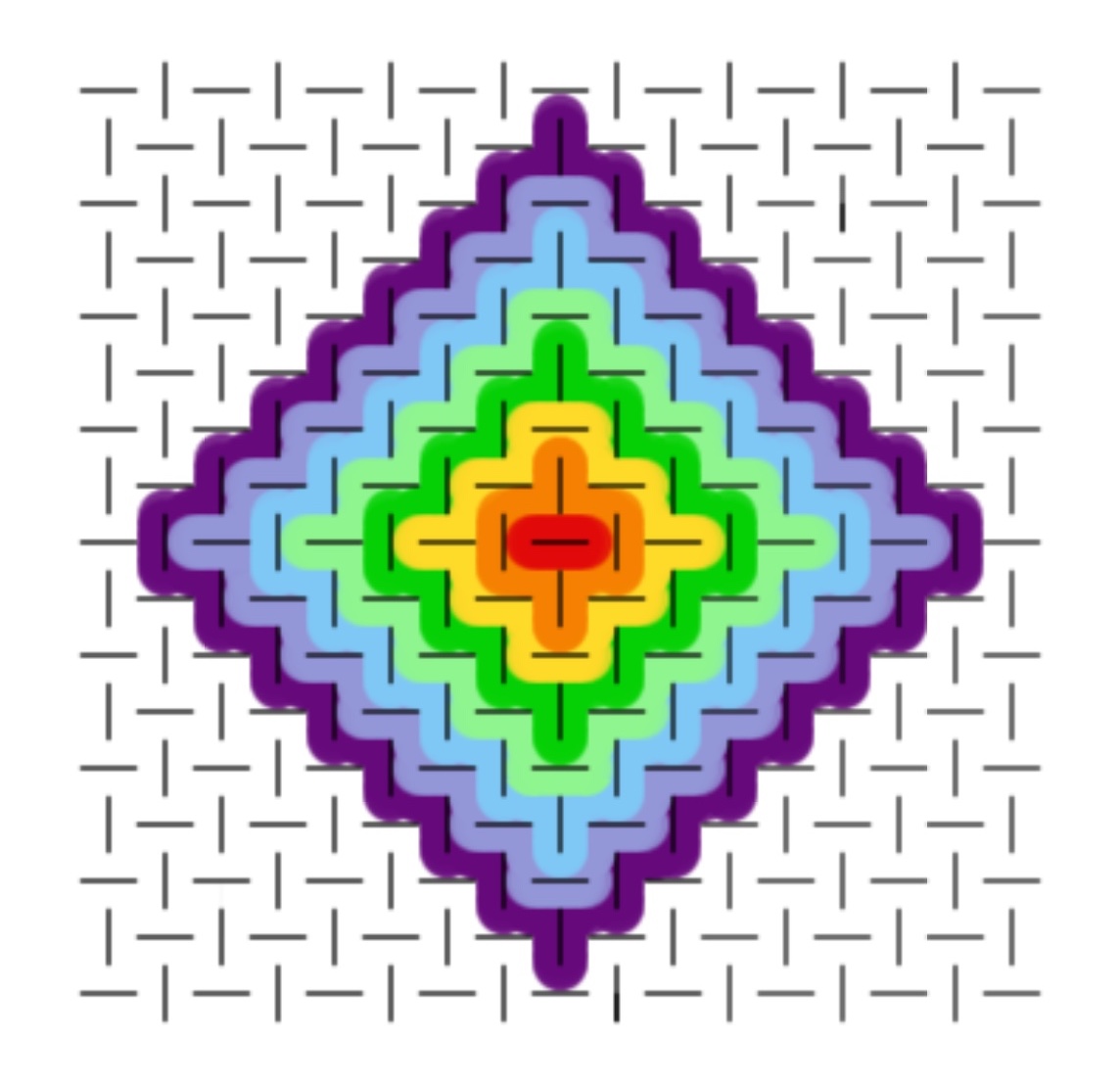}
\caption{Lifts of balls, $B_{R}(p_0)$, where $p_0=I_{(0,0)}$ in the smocked space $(X_T, d_T)$.}
\label{fig:ballsT}
\end{figure}

Observe that by  Lemma~\ref{ball-delta} we know that
\be
\pi^{-1}(B_{R}(p_0))=T_R(I_{(0,0)}) \quad \forall R\in (0,1],
\ee
because 
the smocking separation factor was proven in Lemma~\ref{lem-delta-T} to be
$
\delta_\diamond =1.
$
Thus the ball of radius $R=1$ is depicted in red   in Figure~\ref{fig:ballsT}.

We next apply Proposition~\ref{smock-add}, keeping in mind that $\delta_T=1$ and observing that
when $r=1$ there are four stitches touching the ball of radius $r=1$.  This gives us the ball of radius $2$
depicted in orange in Figure~\ref{fig:ballsT}:
\be
\pi^{-1}(B_{2}(p)) =
T_{1/2} \left(\pi^{-1}(B_1(p))\right) \cup 
\bigcup_{j\in J_1} T_{1}(I_j)
\ee
where 
\be
J_1=\{(2,0), (0,2), (-2,0), (0,-2)\}.
\ee
We next apply Proposition~\ref{smock-add},  observing that
when $r=3$ there are eight stitches touching the ball of radius $r=2$:
\be
J_2=\{(0,\pm 4), (\pm 2,\pm 2), (\pm 4, \pm 0) \}.
\ee
  This gives us the ball of radius $3$
depicted in  yellow in Figure~\ref{fig:ballsT}:
\be
\pi^{-1}(B_{3}(p)) =
T_{1} \left(\pi^{-1}(B_2(p))\right) \cup 
\bigcup_{j\in J_2} T_{1}(I_j).
\ee
We continue in this matter to easily complete the drawing in Figure~\ref{fig:ballsT}.

A rigorous analysis of this space will be continued within this paper.

\subsection{ Exploring the Balls in $X_=$ by Prof.~Sormani, Leslie, and Shanell}

Here we consider balls in $X_=$ centered on the point,
\be
p_0=I_{(0,0)}=[-.5, .5]\times \{0\},
\ee  
by drawing their preimages $\pi^{-1}(B_{R}(p_0))\subset {\mathbb{E}}^N$.  See Figure~\ref{fig:balls=}.

\begin{figure}[h]
\includegraphics[width=.4\textwidth]{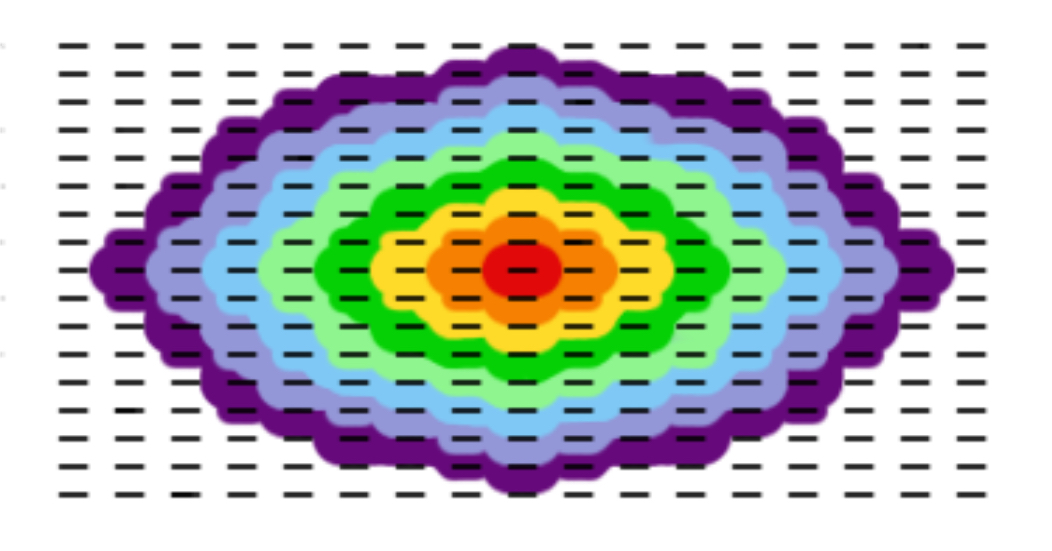}
\caption{Lifts of balls, $B_{R}(p_0)$, where $p_0=I_{(0,0)}$ in the smocked space $(X_=, d_=)$..}
\label{fig:balls=}
\end{figure}

By  Proposition~\ref{ball-delta} we know that
\be
\pi^{-1}(B_{R}(p_0))=T_R(I_{(0,0)}) \quad \forall R\in (0,1],
\ee
because 
the smocking separation factor was proven in Lemma~\ref{lem-delta-T} to be
$
\delta_\diamond =1.
$
Thus the ball of radius $R=1$ is depicted in red   in Figure~\ref{fig:balls=}.

To find the ball of radius $2$
depicted in orange in Figure~\ref{fig:T-balls} observe that we need to do things
very carefully.   Proposition~\ref{smock-add} implies that
\be
\pi^{-1}(B_{1+s}(\pi(p))) =
T_s\left(\pi^{-1}(B_1(\pi(p)))\right) \cup 
\bigcup_{j\in J_1} T_s(I_j)
\ee
where
\be
J_1=\{(2,0), (0,2), (-2,0), (0,-2)\} 
\ee
and
\be
s\le \bar{\delta_r}=\sqrt{2}-1
\ee
because at $\sqrt{2}$ we hit the intervals in
\be
J_{\sqrt{2}}=\{\pm 2, \pm 2\}.
\ee
So
\be
\pi^{-1}(B_{2}(p)) =
T_{1} \left(\pi^{-1}(B_1(p))\right) \cup 
\bigcup_{j\in J_1} T_{1}(I_j) \cup \bigcup_{j\in J_{\sqrt{2}}} T_{2-\sqrt{2}}(I_j).
\ee

To find the ball of radius $3$ we need
\be
J_3=\{(\pm 3,0), (0, \pm 3)\} \textrm{ and } J_{\sqrt{5}}=\{\pm 1, \pm 2\}
\textrm{ and } J_{1+\sqrt{2}} = \{\pm 2, \pm 1\}
\ee
so we see that things become very complicated rapidly.  Nevertheless
we roughly draw the balls by eye using an art program to produce Figure~\ref{fig:balls=}.

Although these balls appear to be converging to an ellipse, this has been shown to
be false.   Due to the lengthy nature of the estimates involved we will not
explore $X_=$ further within this paper.
Further work on this space will appear in \cite{SWIF-smocked}.

\subsection{ Exploring the Balls in $X_+$ by Emilio, Moshe, and Ajmain}

Let us consider balls in $X_+$ around 
\be
p_0=I_{(0,0)} = ([-1,1] \times \lbrace (0,0) \rbrace) \cup (\lbrace (0,0) \rbrace \times [-1,1])
\ee  
by drawing their preimages $\pi^{-1}(B_{R}(p_0))\subset {\mathbb{E}}^N$.  See Figure~\ref{fig:balls+}.

\begin{figure}[h]
\includegraphics[width=.4 \textwidth]{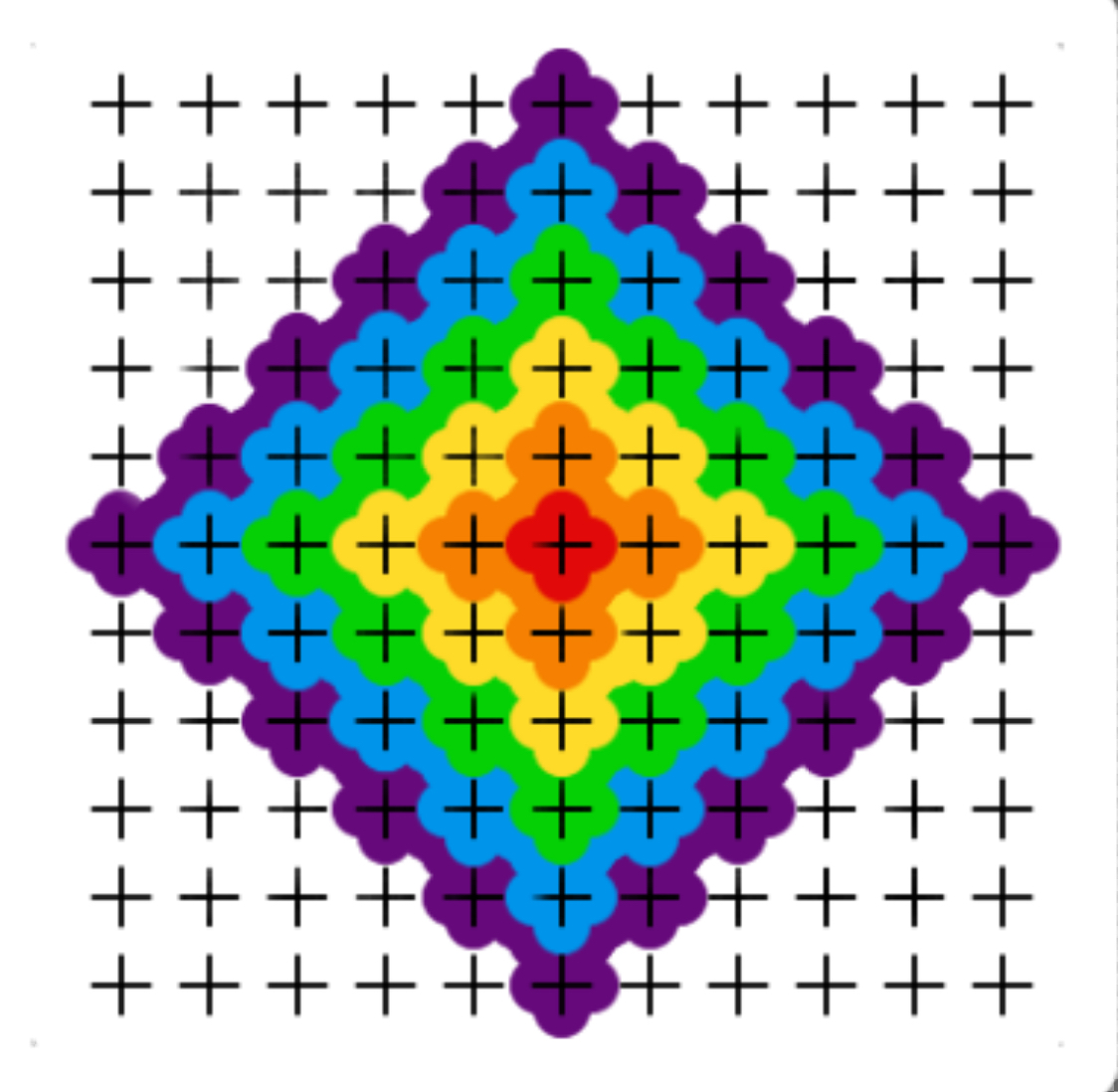}
\caption{Lifts of balls, $B_{R}(p_0)$, where $p_0=I_{(0,0)}$ in the smocked space $(X_+, d_+)$.}
\label{fig:balls+}
\end{figure}

Observe that by  Lemma~\ref{ball-delta} we know that
\be
\pi^{-1}(B_{R}(p_0))=T_R(I_{(0,0)}) \quad \forall R\in (0,1],
\ee
because 
the smocking separation factor was proven in Lemma~\ref{lem-delta-+} to be
$
\delta_+ =1.
$
Thus the ball of radius $R=1$ is depicted in red in Figure~\ref{fig:balls+}.

We next apply Proposition~\ref{smock-add}, keeping in mind that $\delta_T=1$ and observing that
when $r=1$ there are four stitches touching the ball of radius $r=1$.  This gives us the ball of radius $2$
depicted in orange in Figure~\ref{fig:balls+}:
\be
\pi^{-1}(B_{2}(p)) =
T_{1/2} \left(\pi^{-1}(B_1(p))\right) \cup 
\bigcup_{j\in J_1} T_{1}(I_j)
\ee
where 
\be
J_1=\{(\pm 3,0), (0,\pm3)\}.
\ee
We next apply Proposition~\ref{smock-add} again,  observing that
when $r=3$ there are eight stitches touching the ball of radius $r=2$:
\be
J_r=\{(0,\pm 6), (\pm 3,\pm 3), (\pm 6, \pm 0) \}.
\ee
  This gives us the ball of radius $3$
depicted in  yellow in Figure~\ref{fig:balls+}:
\be
\pi^{-1}(B_{3}(p)) =
T_{1} \left(\pi^{-1}(B_2(p))\right) \cup 
\bigcup_{j\in J_2} T_{1}(I_j).
\ee
We continue in this matter to complete the drawing in Figure~\ref{fig:balls+}.

\begin{rmrk} \label{rmrk-shape-+}
It appears that the balls are bumpy diamond shapes.
That is if $k = \lceil{r}\rceil$, with $r > 1$ then 
\be
S'_{3(k-2)} \subset \pi^{-1}(B_{r}(p)) \subset S'_{3k},
\ee
where
\be
S'_r= \lbrace x \in X_+ \, : \, d_T(x,p) < r \rbrace
\ee
with $d_T$ being the taxicab metric:
\be
d_T(x,y) = |x_1 - y_1| + |x_2 - y_2|
\ee
where $x = (x_1,x_2)$, $y = (y_1, y_2)$.
\end{rmrk}

\begin{rmrk} \label{rmrk-J_r-+}
If $k = \lceil{r}\rceil$, then we define $J^+_r$ as such. 
\be
J^+_r = J^+_1 \cup J^+_2 \cup \dots \cup J^+_{k-1} \cup J^+_k
\ee
where
\be
J^+_k = \lbrace \pm(0,3(k-1)), \pm(3\cdot1, 3(k-2)), \dots, \pm(3\cdot (k-1),0) \rbrace
\ee
\end{rmrk}

A rigorous proof of these intuitive estimates will be provided in later sections.

\subsection{ Exploring the Balls in $X_H$ by Prof.~Sormani, David, and Vishnu}

Here we consider balls in $X_H$ centered on the point,
\be
p_0=I_{(0,0)}= [-.5, +.5] \times \{0\}
\ee  
by drawing their preimages $\pi^{-1}(B_{R}(p_0))\subset {\mathbb{E}}^N$.  
We also consider balls about the deepest point $q_0=(1.5, 1.5)$ because they
have more symmetry.   For very large balls the center point should not matter
because
\be
B_R(p_0) \subset B_{R+C}(q_0) \subset B_{R+2C}(p_0)
\ee
where $C=d_H(p_0,q_0)$.   See Figure~\ref{fig:ballsH}.

\begin{figure}[h]
\includegraphics[width=.4\textwidth]{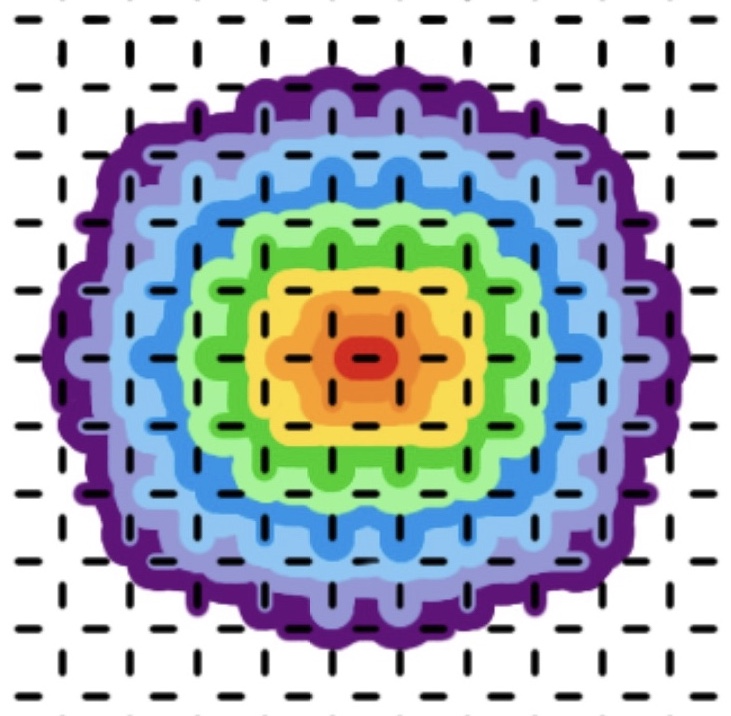} \hspace{1cm} \includegraphics[width=.4\textwidth]{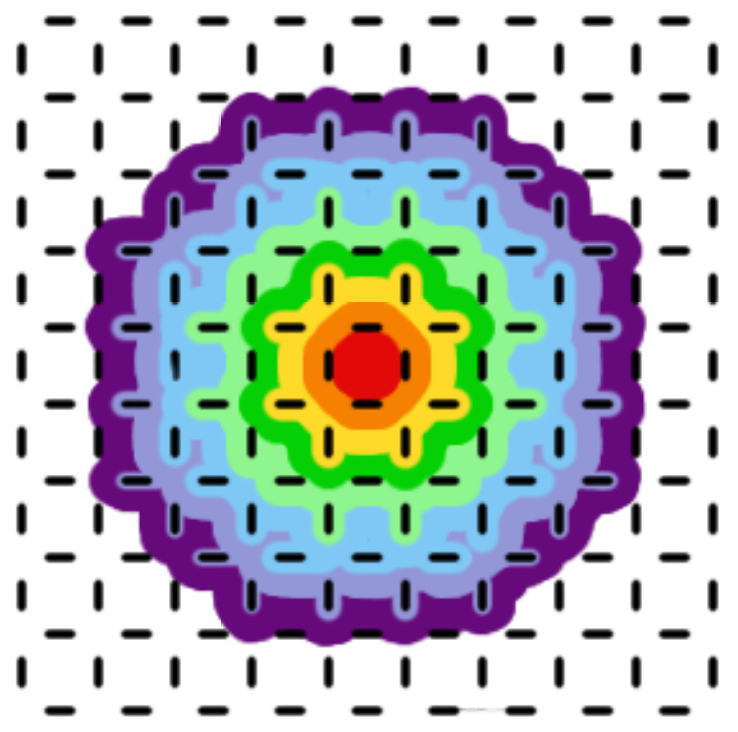}
\caption{Balls centered on $p_0=I_{(0,0)}$ in $(X_H,d_H)$ have vertical and horizontal symmetry
but balls centered at the deepest point $q_0=(1.5, 1.5)$ display octagonal symmetry.}
\label{fig:ballsH}
\end{figure}

Observe that by  Lemma~\ref{ball-delta} we know that
\be
\pi^{-1}(B_{R}(p_0))=T_R(I_{(0,0)}) \quad \forall R\in (0,\sqrt{2}].
\ee
because 
the smocking separation factor was proven in Lemma~\ref{lem-delta-H} to be
$
\delta_\diamond =\sqrt{2}.
$
Thus the ball of radius $R=1$ is depicted in red  in Figure~\ref{fig:balls-square}.

To find the balls of radius $2$ 
depicted in dark orange in Figure~\ref{fig:ballsH} we apply  
Proposition~\ref{smock-add} implies that
\be
\pi^{-1}(B_{2}(\pi(p))) =
T_1\left(\pi^{-1}(B_1(\pi(p)))\right) \cup 
\bigcup_{j\in J_{\sqrt{2}}} T_{2-\sqrt{2}}(I_{\sqrt{2}})
\ee
where
\be
J_{\sqrt{2}}=\{(\pm 1.5,\pm1.5)\}. 
\ee
For radius $3$ we need to include two more intervals
\be
J_2=\{\pm 2.5, 0)\}
\ee
For radius $4$ depicted in yellow we need to include
\be
J_3=\{(0, \pm 3)\} \textrm{ and } J_{2\sqrt{2}}=\{(\pm 3, \pm 3)\} \textrm{ and }
J_{2+\sqrt{2}}=\{(\pm 4.5, \pm 1.5)\}.   
\ee
We continuing drawing larger and larger balls by eye 
on the left side of Figure~\ref{fig:ballsH}
and we cannot see any shape developing.   

However, on the right side of Figure~\ref{fig:ballsH} drawn by eye, we see
a nice octagonal shape forming.  
Due to the complicated nature of the balls in this space, further analysis of this
space is postponed to \cite{SWIF-smocked}.

\subsection{ Exploring the Balls in $X_\square$ by Prof.~Sormani, Maziar, and Hindy:}

Here we consider balls in $X_\square$ centered on the point,
\be
p_0=I_{(0,0)}=\left([0, 1] \times \{0\} \right) \,\, \cup \,\,  \left([0, 1] \times \{1\} \right)
 \quad \cup \left( \{0\}\times [0, 1] \right) \,\,\cup \,\, \left( \{1\}\times [0, 1] \right)
\ee  
by drawing their preimages $\pi^{-1}(B_{R}(p_0))\subset {\mathbb{E}}^N$.  See Figure~\ref{fig:balls-square}.

\begin{figure}[h]
\includegraphics[width=.4\textwidth]{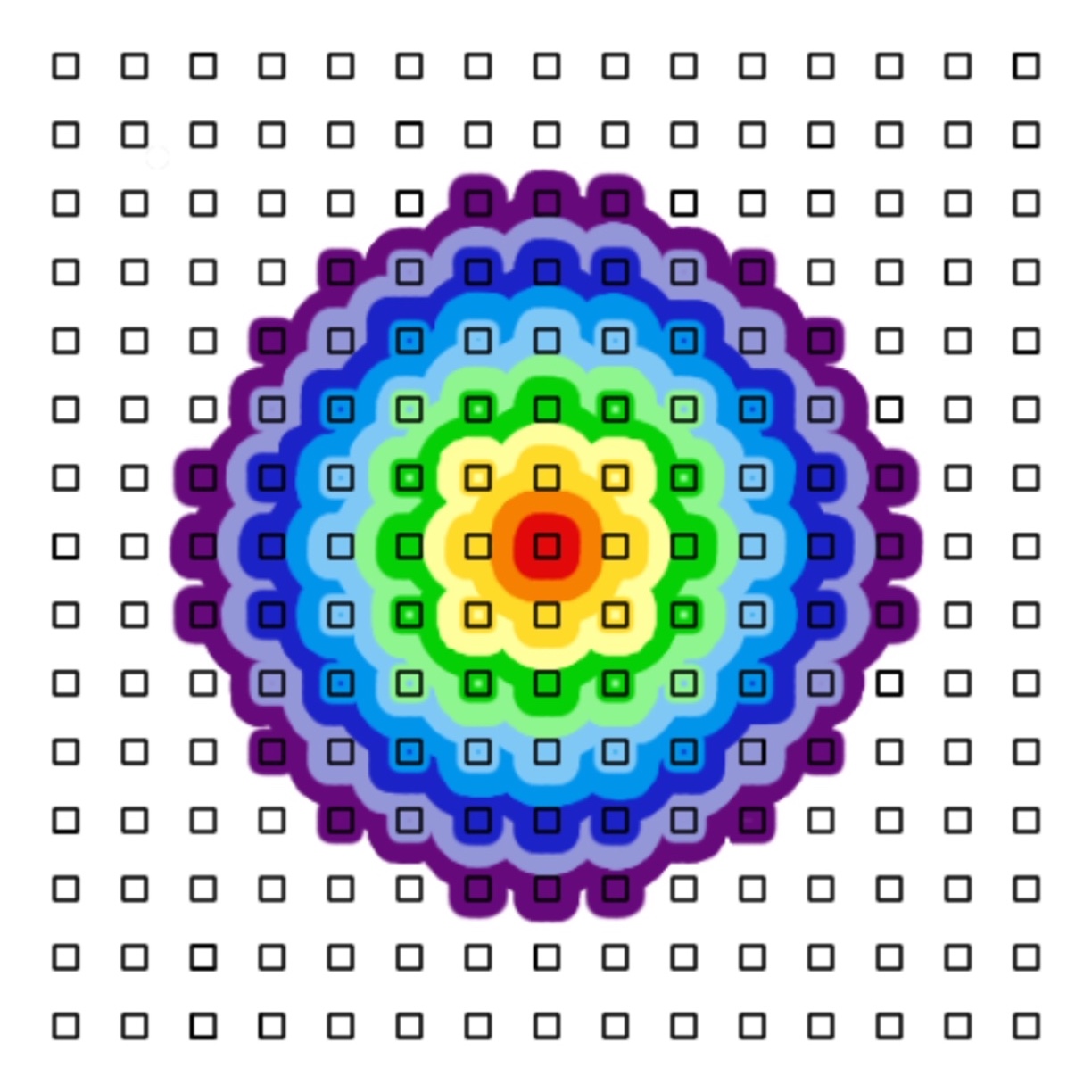}
\caption{Lifts of balls, $B_{R}(p_0)$, where $p_0=I_{(0,0)}$ in the smocked space $(X_\square, d_\square)$..}
\label{fig:balls-square}
\end{figure}

Observe that by  Lemma~\ref{ball-delta} we know that
\be
\pi^{-1}(B_{R}(p_0))=T_R(I_{(0,0)}) \quad \forall R\in (0,2].
\ee
because 
the smocking separation factor was proven in Lemma~\ref{lem-delta-square} to be
$
\delta_\diamond =2.
$
Thus the ball of radius $R=1$ is depicted in red  and
the ball of radius $R=2$ is depicted in orange in Figure~\ref{fig:balls-square}.

To find the balls of radius $3$ and $4$
depicted in shades of yellow in Figure~\ref{fig:balls-square} observe that we need to do things
very carefully.   Proposition~\ref{smock-add} implies that
\be
\pi^{-1}(B_{2+s}(\pi(p))) =
T_s\left(\pi^{-1}(B_2(\pi(p)))\right) \cup 
\bigcup_{j\in J_2} T_s(I_j)
\ee
where
\be
J_2=\{(\pm3,0), (0,\pm3)\} 
\ee
and
\be
s\le \bar{\delta_r}=\sqrt{8}-2
\ee
because at $\sqrt{8}$ we hit the four stitches of
\be
J_{\sqrt{8}}=\{\pm 3, \pm 3\}.
\ee
So
\be
\pi^{-1}(B_{3}(p)) =
T_{1} \left(\pi^{-1}(B_2(p))\right) \cup 
\bigcup_{j\in J_2} T_{1}(I_j) \cup \bigcup_{j\in J_{\sqrt{8}}} T_{3-\sqrt{8}}(I_j).
\ee
and
\be
\pi^{-1}(B_{4}(p)) =
T_{2} \left(\pi^{-1}(B_2(p))\right) \cup 
\bigcup_{j\in J_2} T_{2}(I_j) \cup \bigcup_{j\in J_{\sqrt{8}}} T_{4-\sqrt{8}}(I_j).
\ee
We draw the rest of the balls by eye using an art program to produce Figure~\ref{fig:balls-square}.

To really better understand the shape of these balls on a large scale we 
computed 
\be
\bar{J}^\square_R = \bigcup_{r<R} J^\square_R  \cup  \{(0,0)\}
\ee 
for increasingly large values of $r$ using the computer program: {\em Processing 3}.
See Figure~\ref{fig:square-hex}.  It appears that the balls are becoming octagonal in shape. 
This space will be studied further within this paper.

\begin{figure}[h]
\includegraphics[width=.18\textwidth]{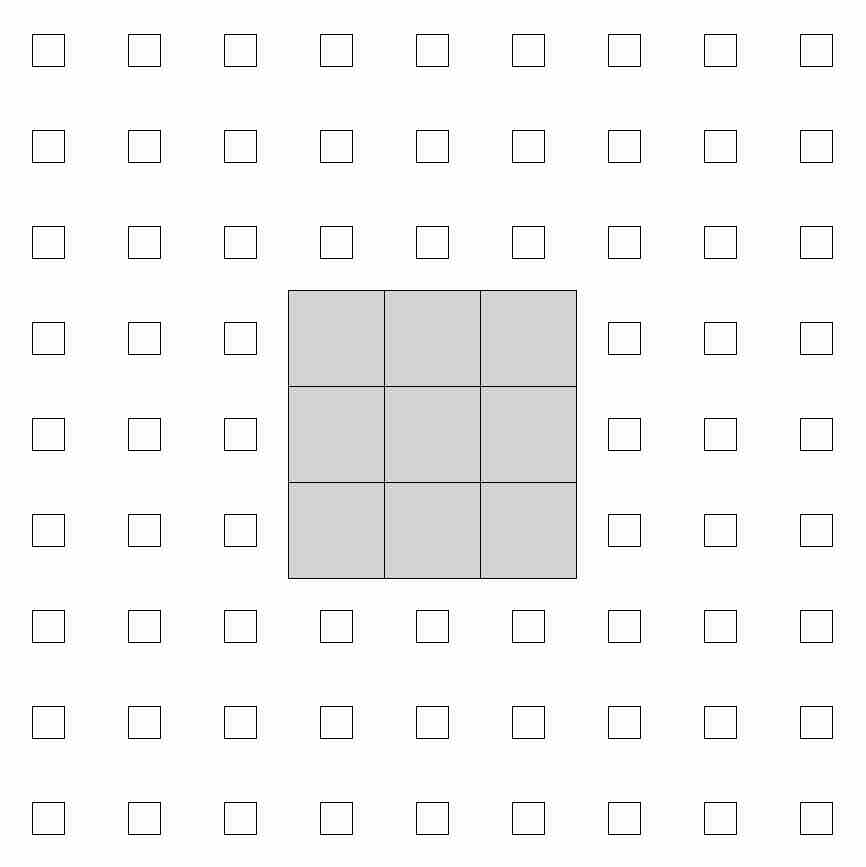}
\includegraphics[width=.18\textwidth]{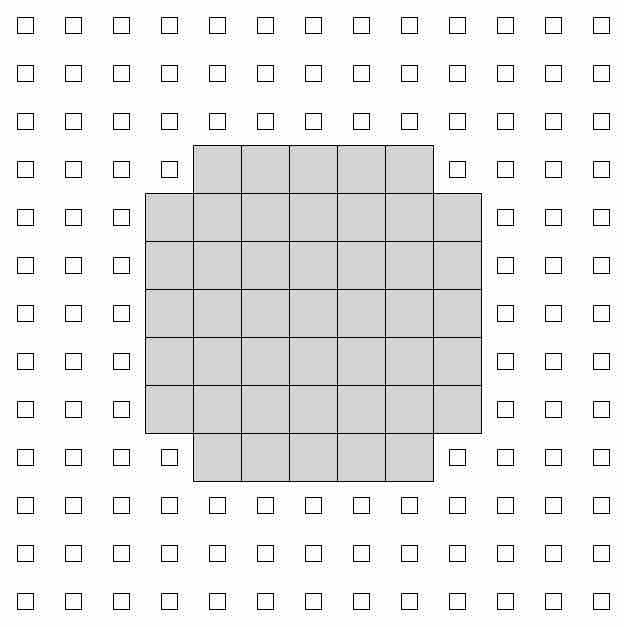}
\includegraphics[width=.18\textwidth]{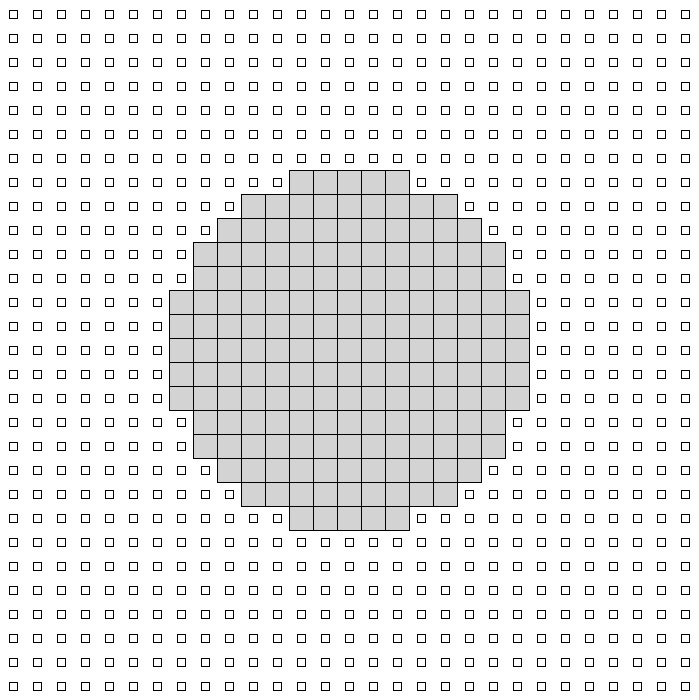}
\includegraphics[width=.18\textwidth]{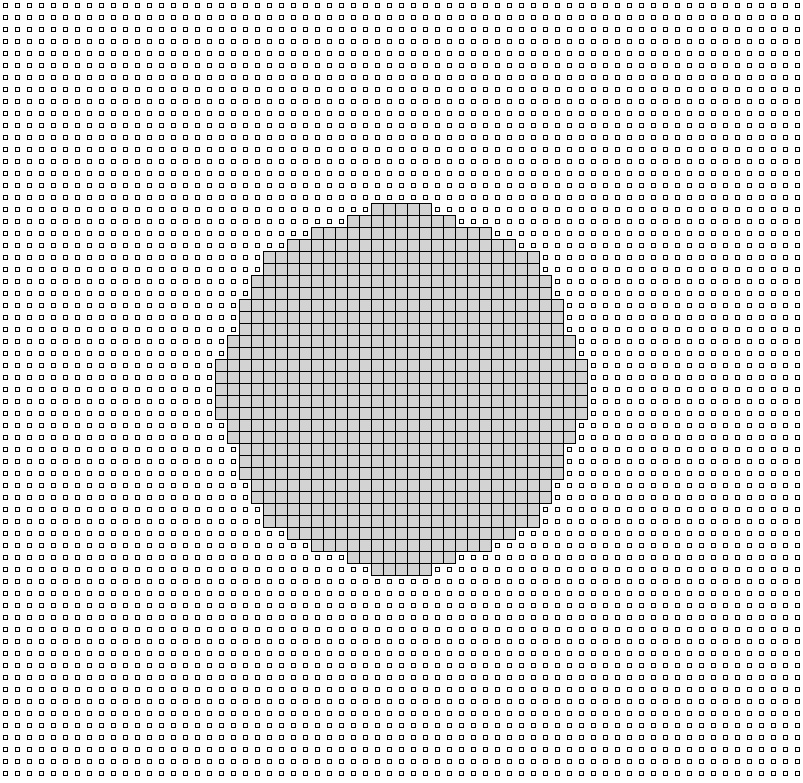}
\includegraphics[width=.18\textwidth]{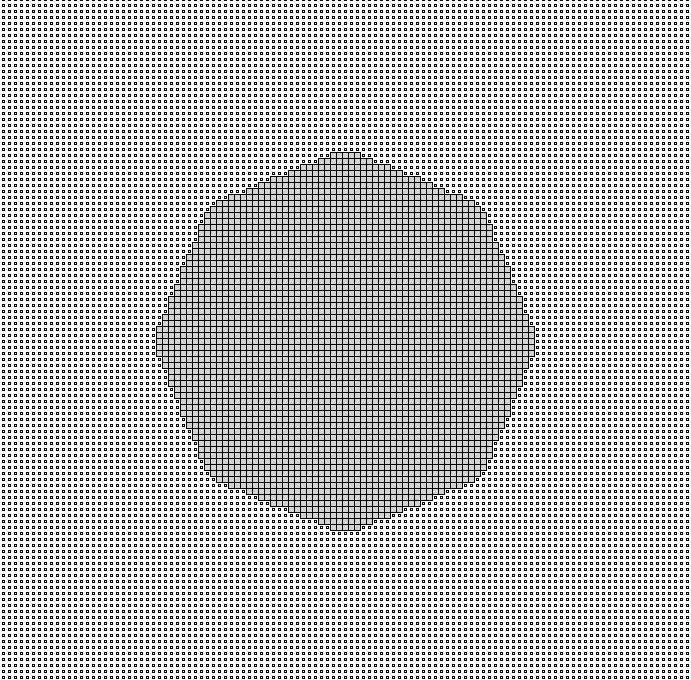}
\caption{Computer generated rescaled images capturing $\bar{J}^\square_R$
for increasingly large values of $R$ using {\em Processing 3}.}   
\label{fig:square-hex}
\end{figure}

\section{ Smocked Metric Spaces are Complete and Noncompact by Prof.~Sormani and Maziar} 

Recall Definition~\ref{defn-smock} of a smocked metric space.  In this section we
prove the following theorem:

\begin{thm} \label{thm-complete}
Any smocked space $(X,d)$ is complete and noncompact.
\end{thm}

Before we prove this theorem, we prove two lemmas.

\begin{lem} \label{lem-one-stitch}
Suppose $(X,d)$ is a smocked metric space with separation factor
$\delta_X>0$ and smocked length $L_{max}<\infty$.  Then
\be \label{one-stitch-a}
\bar{d}(v,w) < \delta \implies \bar{d}(v,w)=\min\{d_0(v,w), d_1(v,w)\} .
\ee
In other words the sum defining the smocking
distance between $\pi(v)$ and $\pi(w)$
is of only one or two segments.  Thus
\be \label{one-stitch-b}
|v-w| \le \bar{d}(v,w) + L_{max}< \delta + L.
\ee
\end{lem}

\begin{proof}
By Theorem~\ref{thm-smock-metric}, we know there exists $N=N(v,w)$ such that
\be 
\bar{d}(v,w) = d_N(v,w).
\ee
If $N \ge 2$ then there are more than three segments and so there
exists $z, z'$ in different smocking stitches such that
\be
\bar{d}(v,w) \ge |z-z'| \ge \delta.
\ee
This is a contradiction.   So we have (\ref{one-stitch-a}).
If $N=0$ then 
\be
\bar{d}(v,w)=|v-w|\ge |v-w|-L_{max}
\ee
and if $N=1$ then there exists $z,z'$ in the same stitch such that
\be
\bar{d}(v,w)=|v-z|+|z'-w|\ge |v-w|-|z-z'| \ge |v-w|-L_{max}
\ee
by the triangle inequality,  Thus in either case we have (\ref{one-stitch-b}).
\end{proof}

\begin{lem}\label{lem-disjoint}
Let $I_j$ and $I_{j'}$ be distinct smocking stitches $j\neq j'$ and
let $p_j=\pi(I_j)$ and $p_{j'} = \pi(I_{j'})$ then
\be\label{tube-cap}
T_r(I_j) \cap T_r(I_{j'}) = \emptyset \quad \forall r < \delta/2
\ee
and
\be \label{ball-cap}
B_r(p_j) \cap B_r(p_{j'}) = \emptyset \quad \forall r < \delta/2.
\ee
where $\delta$ is the separation factor.
\end{lem}

\begin{proof}
First recall that by Lemma~\ref{ball-delta}
\be
B_r(p_j)=\pi(T_r(I_j)).
\ee
So (\ref{ball-cap}) implies (\ref{tube-cap}).

Suppose $q \in B_r(p_j) \cap B_r(p_{j'}) $.  Then
\be
d(p_j, p_{j'}) \le d(p_j,q) + d(q, p_{j'}) < \delta/2 + \delta/2= \delta.
\ee
However if  $z \in I_j$ and $z'\in I_{j'}$ we have
\be
d(p_j, p_{j'})= \bar{d}(z, z') \le |z-z'| < \delta
\ee
by the definition of the separation factor.  Thus we have a contradiction.
\end{proof}

We can now prove Theorem~\ref{thm-complete}:

\begin{proof}
To see that $(X,d)$ is noncompact we show it contains an infinite collection of pairwise
disjoint balls of radius $\delta/2$.   When the smocking set is infinite, we can center
these balls on the images of the smocking stitches and apply Lemma~\ref{lem-disjoint}.
When the smocking set is finite, we center the balls on the images of a sequence of
points diverging to infinity in Euclidean space that are greater than $\delta$ apart and
are far from all the smocking stitches.

To prove completeness we must show that any Cauchy sequence $x_n\in X$ converges to a limit in $X$.
Let $v_n \in \pi^{-1}(x_n) \subset {\mathbb{E}}^N$,
so
\be
\forall \varepsilon>0 \,\, \exists N_\varepsilon \,\, s.t.\,\, \forall n,m \ge N_\varepsilon \,\,\,
d(x_n, x_m)=\bar{d}(v_n, v_m) < \varepsilon.
\ee
Take $\varepsilon=\delta_X$, and $N'=N_{\delta_X}$ then
\be
d(x_n, x_m) =\bar{d}(v_n, v_m)< \delta_X \quad \forall n,m \ge N'.
\ee
By Lemma~\ref{lem-one-stitch}, 
\be
|v_n-v_m| < \delta_X + L_{max} \quad \forall n,m \ge N'.
\ee
So $v_n$ is a bounded sequence in $\mathbb{E}^N$.

By the Bolzano-Weierstrass theorem, any bounded sequence in $\mathbb{E}^N$ has a convergent subsequence, 
So $\exists y \in \mathbb{E}^N$ and a subsequence $v_{n_k} \to y \in \mathbb{E}^N$. 
Since
\be
d(\pi(v_{n_k}), \pi(y)) = \bar{d}(v_{n_k}, y) \le |v_{n_k}-y| \to 0
\ee
we see that the subsequence $x_{n_k}=\pi(v_{n_k}) \to \pi(y)$. 

If a Cauchy sequence has a converging subsequence, then the sequence itself converges to the same limit.
So $x_n \to \pi(y) \in X$.  Thus $(X,d)$ is complete.

\end{proof}


\section{GH Convergence and Tangent Cones at Infinity}

\subsection{ Review of the Definitions by Prof.~Sormani}

Gromov-Hausdorff convergence was first defined by David Edwards in \cite{Edwards} and 
rediscovered by Gromov in \cite{Gromov-metric}.  See the text by Burago-Burago-Ivanov
\cite{BBI} for an
excellent introduction to this topic.  

\begin{defn} \label{defn-GH}
We say a sequence of compact metric spaces
\be
(X_j, d_j) \GHto (X_\infty, d_\infty)
\ee
iff
\be
d_{GH}((X_j,d_j), (X_\infty, d_\infty)) \to 0.
\ee
Where the Gromov-Hausdorff distance is defined
\be
d_{GH}(X_j, X_\infty) = \inf \{d^Z_H(\varphi_j(X_j), \varphi_\infty(X_\infty)): \,\, Z,\,\, \varphi_j: X_j \to Z\}
\ee
where the infimum is over all compact metric spaces, $Z$, and over all
distance preserving maps $\varphi_j: X_j\to Z$:
\be
d_Z(\varphi_j(a), \varphi_j(b))=d_j(a,b) \,\,\,\forall a,b \in X_j.
\ee
The Hausdorff distance is defined 
\be
d_H(A_1, A_2) = \inf\{r:\,\, A_1\subset T_r(A_2) \textrm{ and } A_2\subset T_r(A_1) \}.
\ee
\end{defn}

Notice in the above definition that if a compact metric space $X_j$ were replaced by another 
compact metric space $Y$ which is isometric to it then
\be
d_{GH}(X_j, X_\infty)=d_{GH}(Y_j, X_\infty).
\ee
So this Gromov-Hausdorff distance between metric spaces is really a distance between isometry
classes of metric spaces.   Furthermore, Gromov proved in that given two compact metric spaces, $X$ and $Y$,
\be
d_{GH}(X,Y)=0 \iff X \textrm{ is isometric to } Y.
\ee

\begin{ex} \label{taxi-same}
The taxi space
\be
(X_1,d_1)=({\mathbb{E}}^2, d_{taxi}) \textrm{ where }d_{taxi}((x_1,x_2),(y_1,y_2)) =|x_1-y_1|+|x_2-y_2|
\ee
and the rescaled taxi space:
\be
(X_2,d_2)=({\mathbb{E}}^2, d_{taxi/2}) \textrm{ where }d_{taxi/2}((x_1,x_2),(y_1,y_2)) =(|x_1-y_1|+|x_2-y_2|)/2
\ee
are isometric via the isometry: 
\be
F: X_1\to X_2 \textrm{ where } F(x_1,x_2)= (2x_1, 2x_2)
\ee
because for all $x,y\in {\mathbb{E}}^2$ we have
\be
d_2(F(x),F(y))= (|2x_1-2y_1|+|2x_2-2y_2|)/2=|x_1-y_1|+|x_2-y_2|=d_1(x,y) .
\ee
The Gromov-Hausdorff distance between these spaces is $0$.  This can easily be
proven using the following theorem by defining the correspondence 
\be
\mathcal{C}=\{(x,F(x)): \, x\in X_1\}
\ee
\end{ex}

\begin{thm}\label{thm-correspondence}
If there exists a correspondence $\mathcal{C} \subset X \times Y$:
\be
\forall x \in X \,\,\exists y \in Y \,\,s.t.\,\, (x,y) \in \mathcal{C} \textrm{ and }
\forall y \in Y \,\,\exists x \in X \,\,s.t.\,\, (x,y) \in \mathcal{C}
\ee
which is $\epsilon$ almost distance preserving:
\be
|d_X(x_1,x_2)-d_Y(y_1,y_2)|<\epsilon \quad \forall (x_i,y_i) \in \mathcal{C}
\ee
then
\be
d_{GH}\left((X,d_X), (Y,d_Y)\right) < 2\epsilon.
\ee
\end{thm}

In this paper we are considering unbounded metric spaces, so
we must consider the following definition by Gromov:

\begin{defn}\label{defn-ptGH}
If one has a sequence of complete noncompact metric spaces, $(X_j, d_j)$, and points
$x_j \in X_j$, one can define pointed GH convergence:
\be
(X_j, d_j,x_j) \ptGHto (X_\infty, d_\infty, x_\infty)
\ee
iff
for every radius $r>0$ the closed balls of radius $R$ in $X_j$ converge in the GH sense as metric spaces
with
the restricted distance to closed balls in $X_\infty$:
\be
d_{GH}((\bar{B}_r(x_j)\subset X_j,d_j), (\bar{B}_r(x_\infty)\subset X_\infty, d_\infty)) \to 0.
\ee
\end{defn}

We will be rescaling our metric spaces to see how they look from a distance.  Unlike the taxi
space in Example~\ref{taxi-same}, most metric spaces are not isometric to their rescalings.  So when
we rescale a metric space repeatedly we obtain a sequence of metric spaces.  If the sequence
or a subsequence converges in the Gromov-Hausdorff sense then we obtain a space that is called the
tangent cone at infinity:

\begin{defn} \label{defn-tan-cone}
A complete noncompact metric space with infinite diameter, $(X, d_X)$,
has a tangent cone at infinity, $(Y, d_Y)$, if there is a sequence of
rescalings, $R_j \to \infty$, and points, $x_0\in X$ and $y_0\in Y$, such that
\be
(X, d/R_j, x_0) \ptGHto (Y, d_Y, y_0)
\ee
\end{defn}

There are a variety of theorems in the literature concerning the existence and
uniqueness of such tangent spaces at infinity.  We will not be applying those
theorems in this paper.  We can prove our theorems directly using only what
we've stated in this section.

\subsection{Distances in Pulled thread Spaces by Prof.~Sormani}

\begin{lem} \label{ps-dist}
In a pulled thread space with an interval $I$ of length $L$ we have
\be
| \bar{d}(a,b) - |a-b| | \le L \quad \forall a,b \in {\mathbb{E}}^N.
\ee
\end{lem}

\begin{proof} 
By Definition~\ref{defn-ps}, 
\be
\bar{d}(a,b)=\min\{|a-b|, |a-z|+|z'-b|:\, z,z'\in I\}.                 
\ee
If the minimum is $|a-b|$, then the lemma follows trivially.
If the minimum occurs at $z, z'\in I$ then
\begin{eqnarray}
|a-b|  &\ge& \bar{d}(a,b)=|a-z|+|z'-b| \\
&=&|a-z|+|z-z'|+|z'-b|  -|z-z'|\\
&\ge& |a-b| -|z-z'| \ge |a-b|-L
\end{eqnarray}
because $z,z'\in I \implies |z-z'|\le L$.  This implies
the lemma.
\end{proof}

Because of this uniform comparison, we have the following 
surprising fact even though the pair of spaces
we consider are unbounded:

\begin{lem} \label{GH-dist}
If $X$ is an $N$ dimensional pulled thread space with an interval of length $L$,
then
\be
d_{GH}((X,d_X), ({\mathbb{E}}^N, d_E)) \le 2L
\ee
where $d_E(v,w)=|v-w|$.
\end{lem}

\begin{proof}
We set up the correspondence
\be
\mathcal{C} = \{ (\pi(w) ,w)\in X \times {\mathbb{E}}^N: \,\, w \in {\mathbb{E}}^N \},
\ee
which is a correspondence because $\pi: {\mathbb{E}}^N \to X$
is surjective.   It is $L$ distance preserving because
\be
| d(\pi(v), \pi(w)) -|v-w| | = | \bar{d}(v,w) -|v-w| | \le L.
\ee
The lemma then follows from Theorem~\ref{thm-correspondence}.
\end{proof}

\subsection{Rescaling Pulled Thread Spaces by Prof.~Sormani}

\begin{lem} \label{ps-dist-R}
In a pulled thread space with an interval $I$ of length $L$ we have
\be
\lim_{R\to \infty} \frac{\bar{d}(Rx,Rx)}{R} = |x-y| 
\ee
where the convergence is uniform on ${\mathbb{E}}^N$:
\be
| \,  \bar{d}(Rx,Ry)/R -|x-y|\, |\, \le\, L/R \quad \forall a,b \in {\mathbb{E}}^N.
\ee
\end{lem}

\begin{proof} 
By Lemma~\ref{ps-dist} applied with $a=Rx$ and $b=Ry$
we have
\be
\left| \frac{\bar{d}(Rx,Ry)}{R} - |x-y| \right| = \frac{| \bar{d}(Rx,Ry)- |Rx-Ry| |}{R} \le \frac{L}{R}
\ee
so $\lim_{R\to0} | \,  \bar{d}(Rx,Ry)/R -|x-y|\, |=0$ uniformly on ${\mathbb{E}}^N$.
\end{proof}

\begin{thm}\label{thm-ps-R}
If $X$ is an $N$ dimensional pulled thread space with an interval $I$ of length $L$,
then it has a unique tangent cone at infinity which is ${\mathbb{E}}^N$ endowed with
the standard Euclidean metric $d_E(v,w)=|v-w|$.
\end{thm}

\begin{proof}
Take any $x_0\in X$.  By shifting
the location of the interval, $I$, we may assume that $\pi(0)=x_0$ where
$\pi: {\mathbb{E}}^N\to X$ is the pulled thread map.   

We need to show that for all $r>0$
\be
\lim_{R\to \infty}d_{GH}((B_{Rr}(x_0), d_X/R), (B_r(0), d_E)) =0
\ee
and we will do this by finding a correspondence for each $R,r>0$.
Let 
\be 
U_{Rr}(x_0)= \pi^{-1}\left(\bar{B}_{Rr}(x_0)\right)\subset {\mathbb{E}}^N.
\ee   
Note that by the fact that 
\be
d(x,x_0)=\bar{d}(u,0) \le |u-0| \textrm{ when } \pi(u)=x
\ee
 we have
\be 
\bar{B}_{Rr}(0) \subset U_{Rr}(x_0).
\ee 
We set up a correspondence
\be
\mathcal{C}_R = \{ (\pi(w), f(w)) : \,\, w\in U_{Rr}(x_0)\} \subset B_{Rr}(x_0) \times B_r(0),
\ee
where $\pi$ is the pulled thread map and $f: U_{Rr}(x_0) \to \bar{B}_{r}(0) $
is defined:
\begin{equation}
   f(w)=%
   \begin{cases}
     w/R &\text{if }|w| < rR \\
       r(w/|w|) &\text{if $w$} \ge rR.
   \end{cases}
\end{equation}
This is a correspondence because $\pi:  U_{Rr}(x_0) \to B_{Rr}(x_0)$ and
$f: U_{Rr}(x_0) \to \bar{B}_r(0)$ are surjective.   

We claim that $\mathcal{C}$ is $(3L)/R$ almost distance preserving:
$$
|\,d_X(\pi(v),\pi(w))/R - | f(v)-f(w)| \,| \le (3L/R).
$$
Observe that 
\be
|f(w)-w/R| =   \begin{cases}
   0 &\text{if $|w| < rR$} \\
    (|w|/R) \,-\,r  &\text{if $w \ge rR$}
   \end{cases}
\ee
and since $|w| \le Rr + L$ on $U_{Rr}(x_0)$, 
we have
\be
|f(w)-w/R|\le L/R \qquad \forall w \in U_{Rr}(x_0).
\ee
So
\be
| \, |f(v)-f(w)| - |v/R - w/R| \, | \le  2L/R.
\ee
Thus we have our claim:
\begin{eqnarray}
\left |\,\frac{d_X(\pi(v),\pi(w))}{R} - | f(v)-f(w)| \,\right|&=& \left|\,\frac{\bar{d}(v,w)}{R} - \frac{|v-w|}{R}\,\right|\, +\, \frac{2L}{R}\\
&=& \left( |\,\bar{d}(v,w) - |v-w| \,|\, +\, 2L \right)/R\\
&\le &( L +\, 2L )/R =(3L)/R.
\end{eqnarray}
So by Theorem~\ref{thm-correspondence} we have
\be
d_{GH}((B_{Rr}(x_0), d_X/R), (B_r(0), d_E)) \le 2 (3L)/R \to 0 \textrm{ as } R \to \infty.
\ee
Note that we do not need a subsequence, nor does this depend on the base point, $x_0$.
\end{proof}

\section{Approximating the Smocking Pseudometric}

In this section, we analyze the smocking pseudometric.  In the first subsection
we prove a key lemma which will allow us to estimate the smocking distances
between points when one can only approximate the smocking distances between intervals.
The next few subsections, we find the smocking pseudometric
for three of our smocked spaces: $X_T$, $X_+$,
and $X_\square$.  The proofs for the other smocked spaces, $X_\diamond$, 
$X_=$ and $X_H$, are significantly more difficult, so we postpone them 
to our next paper \cite{SWIF-smocked}.   

\subsection{ Key Lemma by Prof.~Sormani}

\begin{lem} \label{lem-dil-to-approx}
Given an $N$ dimensional smocked space 
parametrized by points in intervals as in (\ref{param-by-points}),
with smocking depth $h\in (0,\infty)$,
and smocking length $L=L_{max}\in (0,\infty)$, 
if one can find a Lipschitz
function $F: {\mathbb{E}}^N \to [0, \infty)$ such that
\be
|\,d(I_j, I_{j'})\,- \,[F(j)-F(j')] \,| \,\le \,C, 
\ee
then
\be
|\,\bar{d}(x, x')\,- \,[F(x)-F(x')] \,| \,\le \, 2h+ C + 2 \dil(F) (h+L)
\ee
where $dil(F)$ is the dilation factor or Lipschitz constant of $F$:
\be
\dil(F) = \sup \left\{\frac{|F(a)-F(b)|}{|a-b|}\,:\,\, a\neq b \in {\mathbb E}^N\right\}.
\ee
\end{lem}

\begin{proof}
Given any $x, x' \in {\mathbb{E}}^N$, by the definition of smocking depth,
we have closest points in closest intervals, $z\in I_j$ and $z'\in I_{j'}$, such that
\be
d(x, I_j)=\bar{d}(x, z)=|x-z| \le h \textrm{ and } d(x', I_{j'})=\bar{d}(x', z')=|x'-z'| \le h.
\ee
Since our smocked space is parametrized by points in intervals
we have
\be
|z-j|\le L \textrm{ and } |z'-j'|\le L.
\ee
So
\be \label{75}
|x-j|\le L+h \textrm{ and } |z'-j'|\le L+h.
\ee
By the definition of the smocking pseudometric we have 
\be
\bar{d}(z,z')=d(\pi(z), \pi(z'))=d(I_j, I_{j'}).
\ee
Thus by the $\bar{d}$ triangle inequality twice we have
\begin{eqnarray}
\qquad\qquad  \bar{d}(x, x') &\le &\bar{d}(x, z) + \bar{d}(z,z') +\bar{d}(z', x')  \\
&\le& h + d(I_j, I_{j'}) +h.\\
d(I_j, I_{j'})&=& \bar{d}(z,z') \le
\bar{d}(x, z) + \bar{d}(x, x')+\bar{d}(z', x')  \\
&\le& h + d(x,x') +h. 
\end{eqnarray}
So
\be
|\bar{d}(x, x') -d(I_j, I_{j'})| \le 2h.
\ee
We are given that
\be
|\,d(I_j, I_{j'})\,- \,[F(j)-F(j')] \,| \,\le \,C,
\ee
so by the triangle inequality we have
\be\label{76}
|\bar{d}(x, x') \,- \,[F(j)-F(j')] \,| \,\le \,2h+C.
\ee
By the definition of dilation and (\ref{75}), we know
\begin{eqnarray}
|F(j)-F(x)| &\le& \dil(F)\, |j-x| \,\,\le \,\,\dil(F) \,(h+L)\\
|F(j')-F(x')| &\le& \dil(F)\, |j'-x'| \,\,\le\,\, \dil(F) \,(h+L).
\end{eqnarray}
Thus
\begin{eqnarray}
| \,[F(j)-F(j')] \,-\,[F(x)-F(x')] \,| &\le&  | F(j)-F(x)|+|F(j')-F(x')| \qquad\\
&\le& 2 \,\dil(F) \,(h+L).
\end{eqnarray}
Combining this with (\ref{76}) we have
\be
|\bar{d}(x, x') \,- \,[F(x)-F(x')] \,| \,\le \,2h+C+ 2 \,\dil(F) \,(h+L).
\ee
\end{proof}


\subsection{ Estimating the Distances $d_+$ by Prof.~Sormani, Shanell, Vishnu, and Hindy}

In this subsection we estimate the distance between the plus stitches in $X_+$.  First, by examination
it appears that the optimal paths of segments joining one plus stitch to another is found by traveling
first vertically and then horizontally as in Figure~\ref{fig:plus-geods}.  The sum of the lengths of vertical segments between
two plus stitches is $1/3$ of the vertical distance between the centers of the plus stitches and
the sum of the lengths of horizontal segments between
two plus stitches is $1/3$ of the horizontal distance between the centers of the plus stitches.  This
intuitively allows us to guess the formula for the distance in lemma below.  To prove the lemma rigorously
one must ensure that there are no other shorter paths.

\begin{figure}[h]
\includegraphics[width=.6\textwidth]{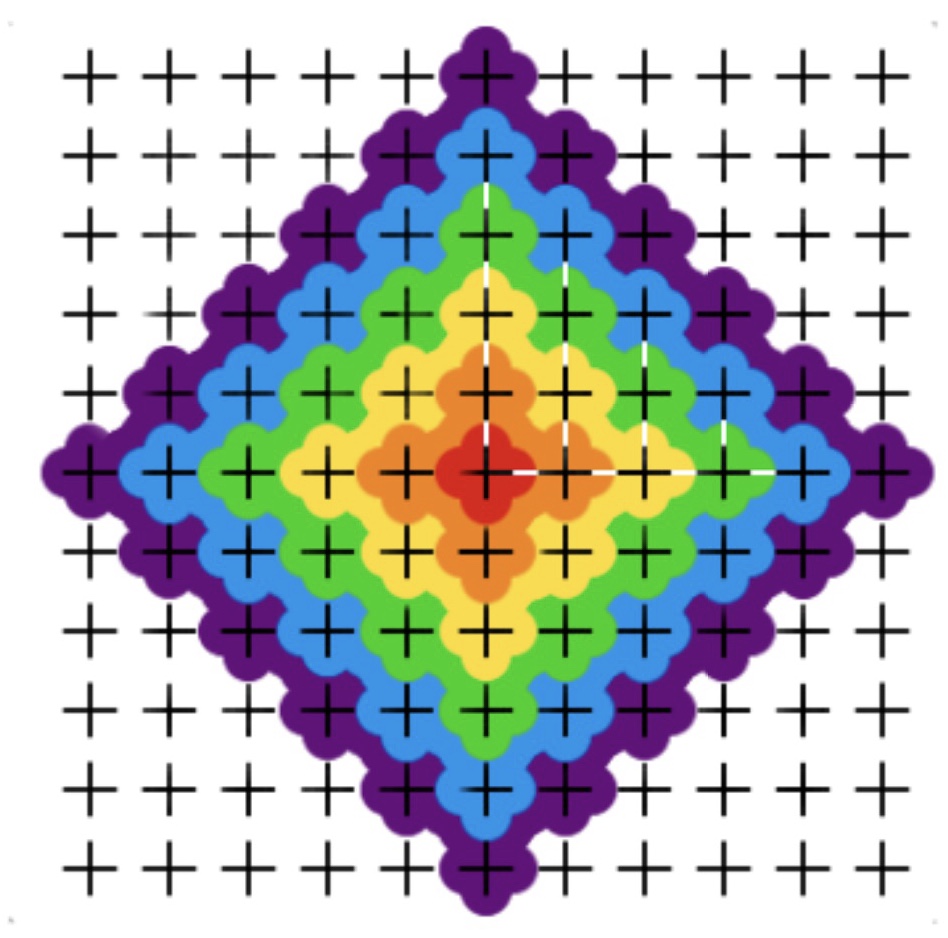}
\caption{Optimal paths of segments from various $I_{(j_1,j_2)}$ to $I_{(0,0)}$ are
drawn in white.}
\label{fig:plus-geods}
\end{figure}

\begin{lem}\label{lem-interval-dist-+}
For any two pairs $(j_1,j_2),(k_1,k_2)\in J_+$, we have
the distance between smocking stiches
\begin{equation}\label{interval-dist-+}
d_{+}(I_{(j_{1},j_{2})},I_{(k_1,k_2)})=\frac{|k_1-j_{1}|+|k_2-j_{2}|}{3}
\end{equation}
\end{lem}

\begin{proof} 
We proceed by inducting on the sum $N=(|j_1-k_1|+|j_2-k_2|)/3$, which is in  ${\mathbb{N}}$
because $j,k\in 3{\mathbb{Z}}\times 3{\mathbb{Z}}$.   By symmetry at any stage 
we can assume $(j_1,j_2)=0$ and $k_1\ge k_2 \ge 0$.  Define 
\be
J_N=\left\{(j_1,j_2):\,\, \frac{|j_{1}|+|j_{2}|}{3}\leq N\right\}.
\ee

For $N=1$,  (\ref{interval-dist-+}) can be demonstrated directly
because the  distance is achieved by a single horizontal segment:
\be
d_+( (3,0), (0,0)) = |(2,0)-(1,0)| = 1 = (3+0)/3.
\ee
Assume that (\ref{interval-dist-+}) holds for all $(k_1,k_2), (j_1,j_2)\in J_+$ satisfying 
\be
(|j_1-k_1|+|j_2-k_2|)/3=N
\ee
for some whole number $N\geq1$.   Now we must prove  (\ref{interval-dist-+})
for all $(k_1,k_2), (j_1,j_2)\in J_+$ satisfying 
\be
(|j_1-k_1|+|j_2-k_2|)/3=N+1.
\ee
By symmetry, we can consider only $k_1>k_2>0$.  Then
taking $i_1=k_1-3$ and $i_2=k_2$ we have $(i_1,i_2) \in S_{N}$.  By the triangle
inequality and induction,
\begin{eqnarray*}
d_+(I_{(k_1,k_2)}, I_{(0,0)}) &\le& d_+(I_{(k_1,k_2)}, I_{(i_1,i_2)})+ d_+(I_{(i_1,i_2)}, I_{(0,0)})\\
&=& 1 + (i_1+i_2)/3  = (3+ k_1-3 +k_2)/3  = (|k_1|+|k_2|)/3.
\end{eqnarray*}
The distance between $I_{(k_1,k_2)}$ and $I_{(0,0)}$ is achieved by a sum of lengths
of segments between intervals.  Note that every one of these segments will head towards the
origin, otherwise they could be replaced with one of the same length that does head towards the
origin and the overall sum of segments would be shorter.

We claim it cannot be achieved by a single segment from $I_{(k_1,k_2)}$ to $I_{(0,0)}$ unless $k_1=3$
and $k_2=0$.  If it could, then
the segment would connect the right tip of $I_{(0,0)}$ which is $(1,0)$ to 
the bottom tip of $I_{(k_1,k_2)}$ which is $(k_1, k_2-1)$ since $k_1\ge k_2> 0$.  So
\begin{eqnarray*}  
d_+(I_{(k_1,k_2)}, I_{(0,0)}) &=& |(k_1, k_2-1)-(1,0)|\\
&=& \sqrt{ (k_1-1)^2 + (k_2-1)^2} = \sqrt{k_1^2 + k_2^2 - 2(k_1 + k_2)} \\
&= & \frac{1}{3}\sqrt{9(k_1^2 + k_2^2) - 18(k_1 + k_2)}\\
&= & \frac{1}{3}\sqrt{8(k_1^2 + k_2^2) + (k_1 - k_2)^2 + 2k_1k_2 - 18(k_1 + k_2)}\\
&\geq & \frac{1}{3}\sqrt{k_1^2 + k_2^2  + 2k_1k_2 + 7(k_1^2 + k_2^2) - 18(k_1 + k_2)}\\
&=&\frac{1}{3}\sqrt{k_1^2 + k_2^2  + 2k_1k_2 + 63\left(\left(\tfrac{k_1}{3}\right)^2 + \left(\tfrac{k_2}{3}\right)^2\right) 
- 54\left( \left(\tfrac{k_1}{3}\right) + \left(\tfrac{k_2}{3}\right)\right)}\\
&\geq& \frac{1}{3}\sqrt{k_1^2 + k_2^2 + 2k_1k_2} \qquad \qquad \textrm{ because } k_i/3 \ge 1\\
& = & \frac{|k_1|+|k_2|}{3}.
\end{eqnarray*}
So the distance between the plus sewing stitches is achieved by a 
collection of segments which passes through some $I_{(i_1, i_2)}$.   Since
the segments head towards the origin we have
\be
(i_1, i_2) \in J_N \textrm{ and }  (k_1-i_1, k_2-i_2) \in J_N.
\ee
Since this interval $I_{(i_1, i_2)}$ lies along a shortest collection of segments we have
\begin{eqnarray*}
d_+(I_{(k_1,k_2)}, I_{(0,0)})&=&d_+(I_{(k_1,k_2)}, I_{(i_1,i_2)})+ d_+(I_{(i_1,i_2)}, I_{(0,0)})\\
&=&(|k_1-i_1|+|k_2-i_2|)/3 + (|i_1-0|+|i_2-0|)/3 \textrm{ by the ind. hyp. }\\
&=&(k_1-i_1+k_2-i_2)/3 + (i_1-0+i_2-0)/3 \textrm{ by } k_j>i_j>0,\\
&=& (k_1+k_2)/3= (|k_1|+|k_2|)/3.
\end{eqnarray*}
\end{proof}

\subsection{ Estimating the Distances $d_\square$ by Prof.~Sormani, Leslie, Emilio, and Aleah}

In this section we prove Proposition~\ref{prop-dist-square}.  

\begin{prop}\label{prop-dist-square}
For any two pairs $(i_1,i_2),(j_1,j_2)\in J_\square$, we have
the distance between square smocking stitches
\be\label{dsq0}
d_{\square}(I_{(i_1, i_2)}, I_{(j_1, j_2)}) = 2\sqrt{2} \,\min \left\{\tfrac{|i_1 - j_1|}{3},\tfrac{|i_2 - j_2|}{3} \right\} + 2 \left| \tfrac{|i_1 - j_1|}{3} - \tfrac{|i_2 - j_2|}{3} \right|.
\ee
\end{prop}

\begin{proof}
Thanks to the symmetry of the placement of the squares, $d_{X_{sq}} = d_{\square} = d$ is translation invariant. 
Thus we will fix $I_{(0,0)}$ and see how to compute $d(I_{(0,0)}, I_{(j_1,j_2)})$ for $j_2\ge j_1> 0$, as this will allow us to compute the distance between any two squares in the lattice, as 
\be
d(I_{(0,0)}, I_{(j_1, j_2)}) = d(I_{(0,0)}, I_{(|j_1|,|j_2|)})=d(I_{(0,0)}, I_{(|j_2|,|j_1|)}),
\ee 
once again because of the symmetry of the square lattice.   So we need only prove that for $j_2\ge j_1\ge 0$
\be
d_{\square}(I_{(0,0)}, I_{(j_1, j_2)}) = \sqrt{8} \left( \frac{j_1}{3} \right) + 2 \left( \frac{j_2-j_1}{3}  \right).
\ee
We will work up to this general formula slowly.  First consult
Figure~\ref{fig:sqmetric} for adjacent square stitches.

\begin{figure}[h]
 \includegraphics[width=.6 \textwidth]{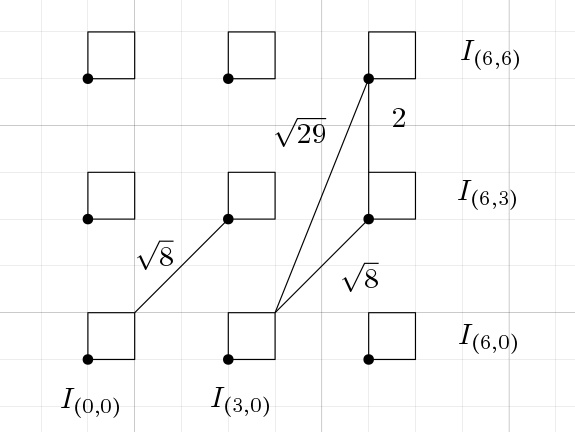}
\caption{Key segments needed to estimate $d_\square(I_j, I_k)$ .}
\label{fig:sqmetric}
\end{figure}

We claim 
\be\label{dsq1}
d(I_{(i_1,i_2)}, I_{(i_1, i_2+3)})= d(I_{(0,0)}, I_{(0,3)}) = 2.
\ee
It is quite easy to see that this will be the Euclidean distance between these two squares. Since we can compute the distance by looking at the distance between the closest corners of the two squares. In the case of $I_{(0,0)}$ and $I_{(3,0)}$ we will simply have to find the shortest distance between the points $(1,1)$ and $(3,1)$. This gives us 
\be
||(3 - 1, 1 - 1)|| = \sqrt{2^2 + 0} = 2.
\ee

We claim that
\be\label{dsq2}
d(I_{(i_1,i_2)}, I_{(i_1+3, i_2+3)})=d(I_{(0,0)}, I_{(3,3)}) = \sqrt{8}.
\ee
We use the Pythagorean theorem here, as the diagonal is clearly shorter than a path going the is first horozontal and then vertical. Since we can compute the distance by looking at the distance between the closest corners we will simply have to find the shortest distance between the points $(1,1)$ and $(3,3)$. This gives us $||(3-1,3-1)|| = ||(2,2)|| = \sqrt{2^2 + 2^2}$, thus we arrive at a distance of $\sqrt{8}$.

We claim
\be \label{dsq3}
d(I_{(i_1,i_2)}, I_{(i_1+3, i_2+6)})=d(I_{(0,0)}, I_{(3,6)}) = \sqrt{8} + 2.  
\ee
Clearly the only contenders for shortest paths are the straight line from $I_{(0,0)}$ to $I_{(3,6)}$ 
of length $\sqrt{29}$
and the path that takes one diagonal jump connecting $I_{(0,0)}$ and $I_{(3,3)}$ plus the straight line $I_{(3,3)}$ to $I_{(3,6)}$ of length $\sqrt{8}+2 < \sqrt{29}$.

We claim that $j_2 \in 3\mathbb{Z}, j_2 > 3$, 
\be\label{dsq4}
d(I_{(0,0)}, I_{(3,j_2)}) = \sqrt{8} + 2(j_2-3)/3
\ee
which is achieved by taking one diagonal segment of length $\sqrt{8}$ followed by $(\frac{j_2}{3} - 1)$
vertical segments of length $2$.  We must verify that there are no shorter paths.

We will prove this by induction on $n$ taking $j_2=3n$.  Note the base case has
been proven for $n=1$ in (\ref{dsq2}) and also $n=2$ in (\ref{dsq3}).  Suppose we
have (\ref{dsq4}) for all $j_2\le 3(n-1)$.  We prove it for $j_2=3n$.   Note that since
(\ref{dsq4}) holds for all $j_2\le 3(n-1)$, this means that it is always shorter to take a single 
diagonal jump connecting $I_{(0,0)}$ and $I_{(3,3)}$ followed by vertical segments, than it
is to take any diagonals of slope greater than $1$, except possibly to go
on a single straight segment from $I_{(0,0)}$ to $I_{(3,3n)}$.  The length of that
single segment is
 \be
 d(I_{(0,0)}, I_{(3,3n)}) = ||(1,1),(3,3n)|| = ||(2,3n-1)|| = \sqrt{9n^2 - 6n + 5}.
 \ee 
 Now we square the distance of taking two jumps $\sqrt{8} + 2(n-1)$ to get $4n^2 - 4n + 8\sqrt{2}$. 
 So to prove our induction step, we just have to prove that for $n > 1$, 
 \be\label{dsq5}
 9n^2 - 6n + 5 > 4n^2 - 4n + 8 \sqrt{2}.
 \ee
 We will prove (\ref{dsq5}) by induction.  Plugging in $n = 2$ we have 
  \be
  29 > 24 > 8 + 8 \sqrt{2},
  \ee
   that proves the base case. 
   Now assuming the induction hypotheses
   \be
   9n^2 - 6n + 5 > 4n^2 - 4n + 8 \sqrt{2}
   \ee 
   holds we wish to prove
   \be
   9(n+1)^2 - 6(n+1) + 5> 4(n+1)^2 - 4(n+1) + 8 \sqrt{2}.
   \ee
   To see this note that
   \begin{eqnarray*}
   9(n+1)^2 - 6(n+1) + 5&=&
   (9n^2 -6n + 5) +18n +3 \\
   &>& (4n^2 - 4n + 8 \sqrt{2}\,) + 18n + 3  \textrm{ by the ind. hyp.}\\
   &>& 4n^2  -4n  +8\sqrt{2} + 8n \\
   &=& 4(n+1)^2 - 4(n+1) + 8 \sqrt{2} \textrm{ by } n>1.
   \end{eqnarray*}
   Thus (\ref{dsq5}) holds for all $n$.  Which implies that the
   single segment is not shorter so we have (\ref{dsq4}).
      
 Before we continue to larger values of $j_2$, we observe that any
 line $\{(x,y):\, y= m(x-1)+1\}$ starting at $(1,1)$ that does not
hit a square with $j_1=3$ as seen in 
Figure~\ref{fig:square-slope} must have slope 
\be \label{dsq6}
3/2 \le m \le 5/3.   
\ee  
This follows from the observation that the line must
pass above  $I_{(3,i_2)}$ and below $I_{(3,i_2+3)}$ for some $i_2\ge 3$.  Thus when this line
crosses $x=3$ and $x=4$ we have
\be
i_2+1 < m(3-1)+1  \textrm{ and }  m(4-1)+1 < i_2+3.
\ee
which implies (\ref{dsq6}).

\begin{figure}[h]
\includegraphics[scale=.59]{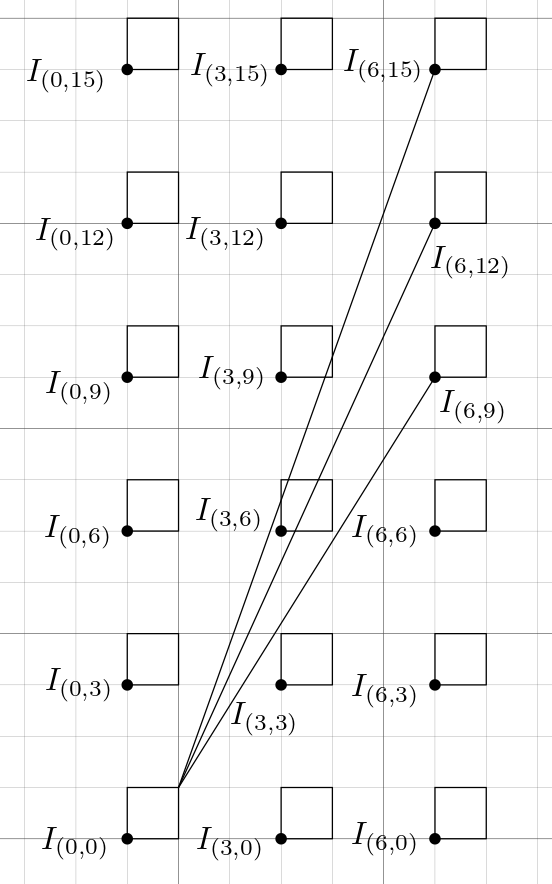}
\caption{Lines with slope greater than 2 must intersect the second column of squares}
\label{fig:square-slope}
\end{figure}
     
We claim that for $j_1=6$ and $j_2 \in 3\mathbb{Z}, j_2 \ge 6 $, 
\be\label{dsq6}
d(I_{(0,0)}, I_{(6,j_2)}) = 2 \sqrt{8} + 2 (j_2-6)/{3} 
\ee
which is achieved by taking two diagonal segment of length $\sqrt{8}$ followed by $(j_2-6)/3$
vertical segments of length $2$.  We must verify that there are no shorter paths.   By
(\ref{dsq4}) shifted, we know:
\be
d(I_{(3,i_2)}, I_{(6,j_2)}) =d(I_{(0,0)}, I_{(3,j_2-i_2)}) = \sqrt{8} + 2(j_2-i_2)/{3} 
\ee
So any shortest path of segments between $I_{(0,0)}$ and  $I_{(6,j_2)}$ which touches a
square on the middle column, $I_{(3,i_2)}$, must consist of only diagonals of length $\sqrt{8}$
and verticals of length $2$.  Such a path has the length in (\ref{dsq6}).  
We need only show that
a direct segment running from $I_{(0,0)}$ to $I_{(6, j_2)}$ is longer than a segment
stopping at a square in the middle.  We know such a segment would have slope 
as in (\ref{dsq6}).  Thus our only concern is when $j_2=9$.
However we see that
the direct segment has longer length too:
\be
||(6,9) - (1,1)|| = \sqrt{5^2 + 8^2} = \sqrt{89} > 2\sqrt{8} + 2
\ee
 because 
 $$
 \left(2\sqrt{8} + 2\right)^2 = 32 + 8\sqrt{8} + 4 = 32 + 16 \sqrt{2} + 4 < 32 + 32 + 4 = 68 < 89 = \left(\sqrt{89}\right)^2.
 $$
 
We complete the proof of the proposition inductively,  still assuming $j_2\ge j_1\ge 0$:\\
{\em Statement(n)}: $ \forall j_2-i_2\ge j_1-i_1=3n$ 
\be
d_{\square}(I_{(i_1, i_2)}, I_{(j_1, j_2)}) = 2\sqrt{2}\, (j_1-i_1)/{3} + 2\,((j_2-i_2)-(j_1-i_1))/{3} 
\ee
{\em which is achieved by taking $n=(j_1-i_1)/3$ diagonal segment of length $\sqrt{8}$ 
followed by $(j_2-j_1)/3$
vertical segments of length $2$}.  

We've already 
proven two base cases $n=1$ and $n=2$.  Assuming we have the induction hypothesis,
{\em Statement(n)}, we now show {\em Statement(n+1)}.  
We first observe that
any shortest path between the squares which passes through a square on the way must
consist of a pair of shortest paths satisfying the induction hypotheses, and so all the segments 
are diagonal segments of length $\sqrt{8}$  or vertical of length $2$.  The only possible way to have a shorter
path is to go directly between the endpoints without hitting any squares in between.   
We claim that for 
$n \ge 3$ there are no such direct paths.   To see this we shift $I_{(i_1,i_2)}$
to $I_{(0,0)}$ so the line segment can be written as $\{(x,y):\, y= m(x-1)+1\}$.
By (\ref{dsq6}), such a direct
path would have to have slope $m\in [3/2,2]$.  We
must also avoid hitting the square $I_{(6,9)}$
as depicted in  Figure~\ref{fig:square-slope}.  
However any line segment which passes
beneath square $I_{(6,9)}$ has slope
\be
m_{under}\le (9-1)/(7-1)=8/6=4/3 < 3/2
\ee
and any line segment which passes above 
square $I_{(6,9)}$ has slope
\be
m_{above}\le (10-1)/(6-1)=9/5> 2.
\ee
So there are no direct paths when $n\ge 3$.   Thus we have proven the
induction step and the proposition follows.
\end{proof}

\subsection{ Estimating the Distances $d_T$ by Dr.~Kazaras, Moshe, and David}

In this subsection, we prove a formula for the distance between the stitched line segments in $X_T$ [Proposition~\ref{prop-dist-T}]. 
Recall that the $d_T$-distance between two points is the infimum of lengths of paths composed of straight line paths. We will show that every path between $I_{(0,0)}$ and another stitch $I_{(i_1,i_2)}$ is at least as long as a path which is constructed via a composition of horizontal and vertical lines of length 1 between neighboring stitches. 
See Figure~\ref{fig:T-geodesics} for an exploration as to how one might jump from
one interval to the next. 

\begin{figure}[h]
\includegraphics[width=.4\textwidth]{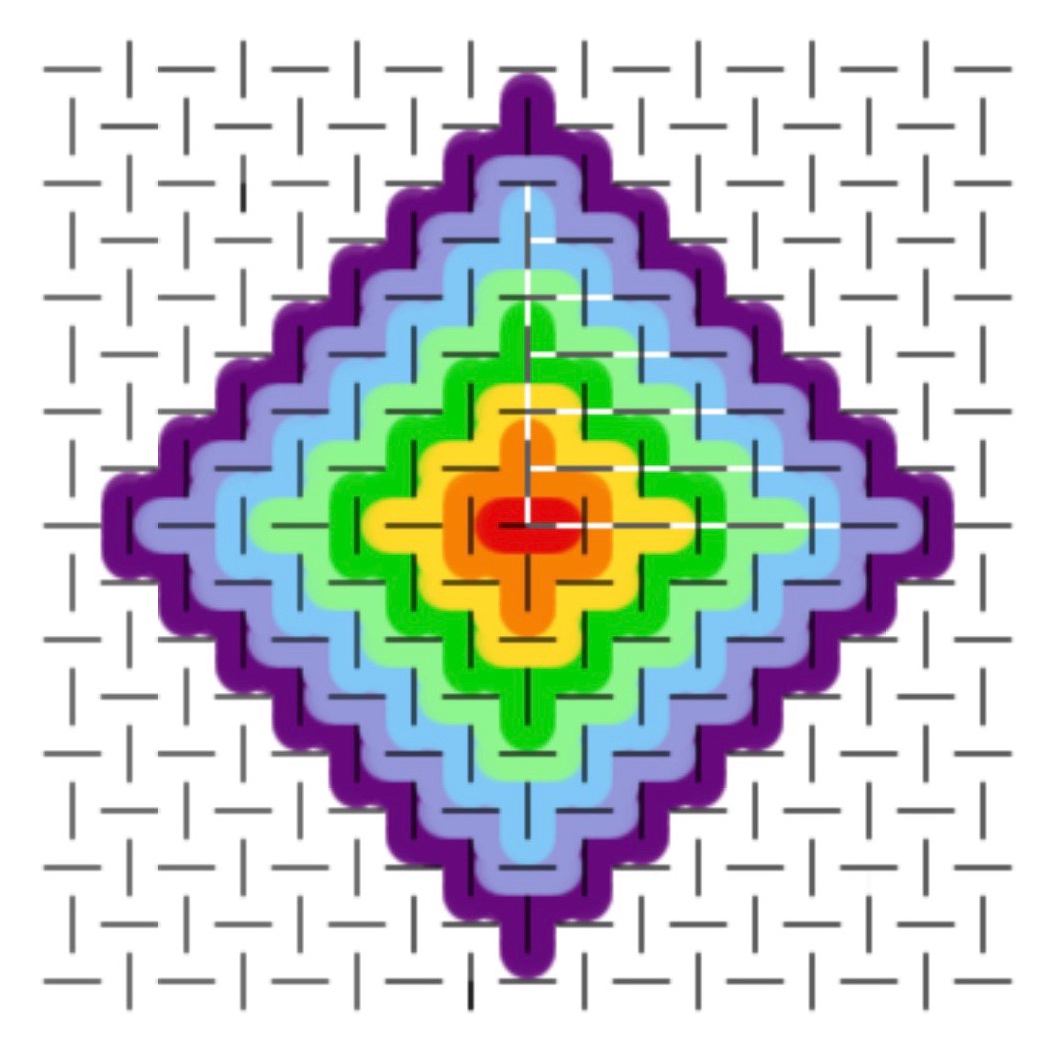}
\caption{Optimal paths of segments from various $I_{(j_1,j_2)}$ to $I_{(0,0)}$ are
drawn in white.}
\label{fig:T-geodesics}
\end{figure}

In this subsection, we prove a formula for the distance between the stitched line segments in $X_T$. 
Recall that the $d_T$-distance between two points is the infimum of lengths of paths composed of straight line paths. We will show that every path between $I_{(0,0)}$ and another stitch $I_{(i_1,i_2)}$ is at least as long as a path which is constructed via a composition of horizontal and vertical lines of length 1 between neighboring stitches. 
See Figure~\ref{fig:T-geodesics} for an exploration as to how one might jump from
one interval to the next.

To begin our discussion, we characterize $d_T$-minimizing line segments leaving $I_{(0,0)}$.

\begin{lem} \label{lem-starting-point-T}
Let $L\subset\mathbb{E}^2$ be a line segment minimizing the $d_T$-distance from $I_{(0,0)}$ to $I_{(k_1,k_2)}$ with $k_1,k_2\geq0$. Then either $L$ has infinite slope and begins at $(0,0)\in I_{(0,0)}$ or $L$ has non-negative slope and begins at $(1,0)\in I_{(0,0)}$.
\end{lem}

\begin{proof}
We only consider the case when the $L$ does not intersect any other stitch $I_{(i_1,i_2)}$ since we may apply our argument to the first such segment of a general line. 
First assume $L$ is vertical. If $L$ does not leave from $(0,0)$, then it must terminate at stitch $I_{(0,4)}$ and have length $4$. However, $d_T(I_{(0,0)},I_{(0,4)})$
is evidently $2$, giving a contradiction.

Now consider the case where $L$ has finite slope. Since $k_1,k_2\geq0$, $L$ must have non-negative slope.
If $L$ leaves from point $(a,0)$ for $a<1$, then $L$ does not even locally minimize distance from $I_{(0,0)}$.
\end{proof}

\begin{lem} \label{lem-length-1-T}
Let $\gamma\subset\mathbb{E}^2$ be a collection of lines minimizing the $d_T$-distance from $I_{(i_1,  i_2)}$ to $I_{(k_1,k_2)}$
with $k_1\geq i_1$ and $k_2\geq i_2$. Then $\gamma$ will intersect $I_{(i_1+2,  i_2)}$, $I_{(i_1,  i_2+2)}$ or $I_{(i_1+2, i_2+2)}$.
\end{lem}

\begin{proof}
We first consider the case where $I_{(i_1,i_2)}=I_{(0, 0)}$, which is a horizontal line segment. Notice that the line with slope $1$ from $(1,0)$ intersects both $I_{(0,2)}$ and $I_{(2,2)}$. As the slope tends towards $\infty$ the line intersects $I_{(2,2)}$ and as the slope tends towards $0$ the line intersects the neighboring  vertical line $I_{(0,2)}$. So if the first segment of $\gamma$ has positive slope, Lemma \ref{lem-starting-point-T} says it must leave $(1,0)$ and hence intersect $I_{(2,0)}$ or $I_{(2,2)}$. If the first segment of $\gamma$
has infinite slope, Lemma \ref{lem-starting-point-T} says it must leave from $(0,0)$ and hence intersect $I_{(0,2)}$.

If $\gamma$ begins from some other horizontal stitch $I_{(i_1,i_2)}$, we can translate $I_{(0, 0)}$ to the horizontal line segment $I_{(i_1,  i_2)}$ under a $d_T$-isometry, which yields that a line with positive slope from $I_{(i_1,  i_2)}$ must intersect $I_{(i_1+2,i_2)},I_{(i_1,  i_2+2)}$ or $I_{(i_1+2,  i_2+2)}$. 

Now we consider the case where $I_{(i_1,i_2)}$ is a vertical stitch. Consider the transformation 
\be
T:\mathbb{E}^2\to\mathbb{E}^2,\quad\quad T(x_1,x_2)=(x_2+2,x_1),
\ee
which is the composition of reflection about the line $x_1=x_2$ and a horizontal translation. $T$ preserves the stitching set and hence yields an isometry of $X_T$.
Applying this transformation to $\gamma$, we can apply the previous case, since $T$ takes vertical stitches to horizontal stitches. This completes the proof of Lemma \ref{lem-length-1-T}.
\end{proof}

Now we are ready to state and proof the formula for $d_T$-distances between 
stitches in $X_T$.

\begin{prop}\label{prop-dist-T}  
For any two pairs $(j_1,j_2), (k_1,k_2)\in J_T$, we have 
\be\label{interval-dist-T}
d_T(I_{(j_1,j_2)},I_{(k_1,k_2)})=\frac{|j_1-k_1|+|j_2-k_2|}{2}.
\ee
\end{prop}

\begin{proof}
By making use of the reflectional and translational symmetry of the stitches that define $X_T$,
it suffices to demonstrate formula (\ref{interval-dist-T}) holds for $j_1=j_2=0$ and $k_1,k_2\geq0$.
We will show that 
\be\label{proof-interval-dist-T}
d_T(I_{(0,0)},I_{(k_1,k_2)})=\frac{k_1+k_2}{2}
\ee
by inducting on the whole number $N=\frac{k_1+k_2}{2}$.
For a whole number $N>0$, consider the collection
\be
B_N=\{(k_1,k_2)\in J_+\colon k_1,k_2\geq0,\frac{k_1+k_2}{2}= N\}.
\ee
Formula (\ref{proof-interval-dist-T}) holds for all $(k_1,k_2)\in B_1$ since $I_{(k_1,k_2)}$ can be reached by unit-length segments emanating from $I_{(0,0)}$ and the tubular neighborhood
of radius $1$ about $I_{(0,0)}$ contains no other stitches. 

Now consider a whole 
number $N>1$ and assume 
formula (\ref{proof-interval-dist-T}) holds for all points in $B_n$ where $n=1,2,\ldots,N$.
Let $(k_1,k_2)\in B_{N+1}$ and let $\gamma$ be a distance-minimizing path from 
$I_{(k_1,k_2)}$ to $I_{(0,0)}$. Consider $\gamma_0\subset\mathbb{E}^2$, the 
component of $\gamma$ which emanates from $I_{(k_1,k_2)}$. Notice that $\gamma_0$
is a single line segment which realizes the $d_T$-distance between $I_{(k_1,k_2)}$ and
the stitch it terminates in, which we denote by $I_{(i_1,i_2)}$. 
Applying Lemma \ref{lem-length-1-T} to $\gamma_0$, 
we find that $I_{(i_1,i_2)}$ must be one of three possibilities: $I_{(k_1-2,k_2)}, I_{(k_1,k_2-2)},$ or $I_{(k_1-2,k_2-2)}$. We consider these cases.

In the cases where $I_{(i_1,i_2)}=I_{(k_1-2,k_2)}$ or $I_{(k_1,k_2-2)}$, then $(i_1,i_2)\in B_N$ and
\begin{align}
d_T(I_{(0,0)},I_{(k_1,k_2)})&=d_T(I_{(0,0)},I_{(i_1,i_2)})+d_T(I_{(i_1,i_2)},I_{(k_1,k_2)})\\
{}&=N+1 =\frac{k_1+k_2}{2}\notag
\end{align}
where we have used the inductive hypothesis. This verifies the 
desired formula (\ref{proof-interval-dist-T}) in this case.

In the case where $I_{(i_1,i_2)}=I_{(k_1-2,k_2-2)}$, then $(i_1,i_2)\in B_{N-1}$ and we may use the inductive hypothesis to find
\begin{align}
d_T(I_{(0,0)},I_{(k_1,k_2)})&=d_T(I_{(0,0)},I_{(i_1,i_2)})+d_T(I_{(i_1,i_2)},I_{(k_1,k_2)})\\
{}&=(N-1)+2 =\frac{k_1+k_2}{2}\notag.
\end{align}
This finishes the proof of the lemma.
\end{proof}

\section{Rescaling Smocked Spaces}

Recall that in Theorem~\ref{thm-ps-R} we proved that the tangent cone at
infinity of a pulled thread space is unique and is Euclidean space with the
standard Euclidean metric.  Later in this section we will show that we also
obtain unique tangent cones for various smocked spaces, that are normed spaces
but not Euclidean space.  The first section has
a theorem which will be a key ingredient
used in those proofs. In the next few subsections, we find the tangent cone
at infinity for three of our smocked spaces:  $X_T$, $X_+$,
and $X_\square$.  The proofs for the other smocked spaces, 
$X_\diamond$, $X_=$ and $X_H$, are significantly more difficult, so we postpone them 
to our next paper \cite{SWIF-smocked}.

\subsection{Main Theorem}

\begin{thm}\label{thm-smocking-R}
Suppose we have an $N$ dimensional smocked space, $(X, d)$, as in Definition~\ref{defn-smock}
 such that
\be
\left|\,\bar{d}(x, x')\,- \,[F(x)-F(x')] \,\right| \,\le \, K \qquad \forall x,x' \in {\mathbb{E}}^N
\ee
where $F: {\mathbb{E}}^N \to [0,\infty)$ is a norm.
Then $(X,d)$ has a unique tangent cone at infinity, 
\be
({\mathbb{R}}^N, d_F) \textrm{ where }
d_F(x,x')=||x-x'||_F=F(x-x').
\ee
\end{thm}

\begin{proof}
Take any $x_0\in X$.  By shifting the smocking set, $S$, we may assume 
that $\pi(0)=x_0$ where
$\pi: {\mathbb{E}}^N\to X$ is the smocking map.   

We need to show that for all $r>0$
\be
\lim_{R\to \infty}d_{GH}((B_{Rr}(x_0), d_X/R), (B_r(0), d_F)) =0
\ee
and we will do this by finding a correspondence for each $R,r>0$.
Let 
\be 
U_{s}(x_0)= \pi^{-1}(\bar{B}_{s}(x_0) )= \{u \in  {\mathbb{E}}^N: \, \bar{d}(u,0)\le s\}.
\ee   
Note that by the fact that 
\be
\bar{d}(u,0) \le F(u) +K  \textrm{ when } \pi(u)=x,
\ee
We have
\be 
U_{Rr}(x_0) \supset \{u: \, F(u)\le Rr-K\} = F^{-1}[0, Rr-K]. 
\ee 
We set up a correspondence
\be
\mathcal{C}_R = \{ (\pi(w), f(w)) : \,\, w\in U_{Rr}(x_0)\} \subset B_{Rr}(x_0) \times B_r(0).
\ee
where $\pi$ is the pulled thread map and 
\be
f: U_{Rr}(x_0) \to \bar{B}_r(0)=F^{-1}[0,r]
\ee
is defined to be surjective:
\begin{equation}
   f(w)=%
   \begin{cases}
     r w / (Rr-K) &\text{if }F(w) < rR-K \\
       r w / F(w)  &\text{if } F(w)\ge rR-K.
   \end{cases}
\end{equation}
This is a correspondence because $\pi:  U_{Rr}(x_0) \to B_{Rr}(x_0)$ and
$f: U_{Rr}(x_0) \to \bar{B}_r(0)$ are surjective.   

We claim that $\mathcal{C}_R$ is $\epsilon_R$ almost distance preserving:
\be
|\,d_X(\pi(v),\pi(w))/R - ||f(v)-f(w)||_F \,| \le \epsilon_R
\textrm{ where }\epsilon_R\to 0 \textrm{ and }R\to \infty.
\ee 
Observe that 
\be
\left|\left| f(w)- \left(\tfrac{r}{(Rr-K)} \right) w\, \right|\right|_F =  \begin{cases}
     0 &\text{if }F(w) < rR-K \\
       \left(||w||_F -(Rr-K)\right) \left(\tfrac{r}{(Rr-K)} \right) &\text{if } F(w)\ge rR-K.
   \end{cases} 
\ee
For $w\in U_{Rr}(x_0)$,  we have $||w||_F \le Rr + K$, 
so
\be
\left|\left| f(w)- \left(\tfrac{r}{(Rr-K)} \right) w\, \right|\right|_F \le \tfrac{2rK}{(Rr-K)}
\ee
and
\be
\left|\left| \left(\tfrac{r}{(Rr-K)} \right) w\,  - \tfrac{1}{R} \, w \right|\right|_F
\le\left( \tfrac{r}{(rR-K)} -\tfrac{1}{R}\right) ||w||_F \le \tfrac{K }{R(rR-K)} (Rr+K).
\ee
Thus for all $v,w \in U_{Rr}(x_0)$ we have
\be
 \left| \, ||f(v)-f(w)||_F  - || \tfrac{v}{R}-\tfrac{w}{R}||_F \,\right| \le \Psi(R; K,r)
 \ee
 where
 \be
 \Psi(R;K,r)=\tfrac{2rK}{Rr-K} + \tfrac{K(Rr+K) }{R(rR-K)} \to 0 \textrm{ as } R\to \infty.
\ee
Taking $\epsilon_R=(1/R) K  \, +\, \Psi(R;K,r)$, 
we see that $\mathcal{C}$  is $\epsilon_R$ almost distance preserving:
\begin{eqnarray*}
|\,d_X(\pi(v),\pi(w))/R -  ||f(v)-f(w)||_F \,|&=& |\,\bar{d}(v,w)/R - ||v-w||_F/R\,|\, +\, \Psi(R;K,r)\\
&=& \tfrac{1}{R} |\,\bar{d}(v,w) - ||v-w||_F \,|\, +\, \Psi(R;K,r) \le \epsilon_R
\end{eqnarray*}
with
$
\lim_{R\to \infty}\epsilon_R=\lim_{R\to\infty} (1/R) K  \, +\, \Psi(R;K,r) =0.
$
So by Theorem~\ref{thm-correspondence} we have
\be
d_{GH}((B_{Rr}(x_0), d_X/R), (B_r(0), d_E)) \le 2 (3K)/R \to 0 \textrm{ as } R \to \infty.
\ee
We do not need a subsequence nor did this limit depend on the base point, $x_0$.
\end{proof}

\subsection{ The Tangent Cone at Infinity of $X_+$  by Dr.~Kazaras, Shanell, Emilio, and Moshe}

In this section we prove the  following theorem which is
depicted in Figure~\ref{fig:plus-lim}.

\begin{thm} \label{thm-tan+}
The tangent cone at infinity of $(X_+, d_+)$ is $({\mathbb{R}}^2, d_{F+})$,
where
\be \label{F-def-+}
d_{F+}( \vec{x},\vec{y})= F_+(\vec{x}- \vec{y})
\textrm{ where }
F_+(\vec{x})= ||\vec{x}||_F=\frac{|x_1| + |x_2|}{3}.
\ee
Note that in the limit space, the shape of a ball, 
$
\{\vec{x}\in\mathbb{R}^2\colon \,\,F_+(\vec{x}) = R\},
$
 is a diamond with vertices at $(\pm 3,0)$ and $(0, \pm 3)$.
\end{thm}

\begin{figure}[h]
\includegraphics[width=.8 \textwidth]{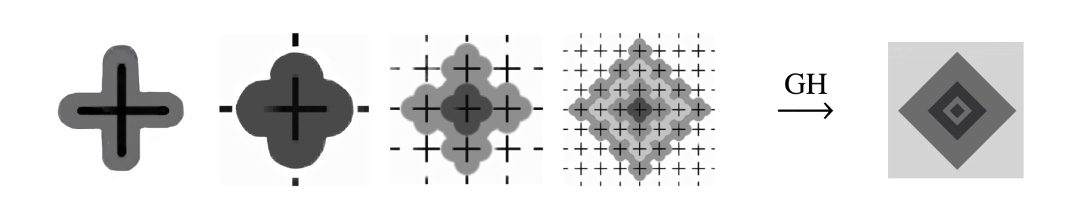}
\caption{Converging to the Tangent Cone at Infinity.}
\label{fig:plus-lim}
\end{figure}

We will prove this by applying Theorem~\ref{thm-smocking-R}.
Recall that the function $F_+$ was first found in  Lemma \ref{lem-interval-dist-+}, 
\be
d_+(I_{(j_1,j_2)}, I_{(j_1',j_2')}) = \frac{|j_1 - j_1'| + |j_2 - j_2'|}{3} =  F_+(\vec{j} - \vec{j'}).
\ee
We will prove a series of lemmas to show we have all the hypotheses needed to apply  
Theorem~\ref{thm-smocking-R} taking $F=F_+$.

\begin{lem}
$F_+$ as in (\ref{F-def-+}) is a norm.   
\end{lem}

\begin{proof}
It is definite:
\be
0=F_+(\vec{x})= \frac{|x_1| + |x_2|}{3} \,\, \iff \,\, x_1=x_2=0
\ee
It scales:
\be
F_+(R\vec{x})= \frac{|Rx_1| + |Rx_2|}{3}= |R|  \frac{|x_1| + |x_2|}{3}= |R| F_+(\vec{x}).
\ee
It satisfies the triangle inequality:
\be
F_+(\vec{x}+\vec{y})= \frac{|x_1+y_1| + |x_2+y_2|}{3}= \frac{|x_1| + |x_2|}{3} + \frac{|y_1| + |y_2|}{3}=
F_+(\vec{x})+F_+(\vec{y}).
\ee
\end{proof}

Next, we will estimate
\be
\text{dil}(F_+) = \max_{\vec{a} \neq \vec{b}\in\mathbb{R}^2} \frac{| F(\vec{a}) - F(\vec{b})|}{|\vec{a} - \vec{b}|}.
\ee

\begin{lem}\label{dilF-+}
The function $F_+$ satisfies $\text{dil}(F_+) \leq 1$.
\end{lem}

\begin{proof}
Recall that 
\be
||\vec{x}||_{taxi} = |x_1| + |x_2|
\ee
 is the taxicab norm, and 
\be
||\vec{x}||_{\mathbb{E}^2} = \sqrt{(x_1)^2 + (x_2)^2}
\ee
 is the Euclidean norm. We have
\begin{align}
\text{dil}(F_+) &= \max_{\vec{a} \neq \vec{b}} \frac{1}{3} \frac{|a_1 - b_1 + a_2 - b_2|}{\sqrt{(a_1 - b_1)^2 + (a_2 - b_2)^2}} \notag\\
{}&\leq\max_{\vec{a} \neq \vec{b}} \frac{1}{3} \frac{|a_1 - b_1| + |a_2 - b_2|}{\sqrt{(a_1 - b_1)^2 + (a_2 - b_2)^2}} 
= \max_{\vec{a} \neq \vec{b}} \frac{1}{3} \frac{||\vec{a} - \vec{b}||_{taxi}}{|| \vec{a} - \vec{b}||_{\mathbb{E}^2}}.\notag
\end{align}
It is well-known that $||\vec{x}||_{\mathbb{E}^2} \leq ||\vec{x}||_{taxi} \leq \sqrt{2} || \vec{x} ||_{\mathbb{E}^2}$. Thus $\frac{1}{||\vec{x}||_{\mathbb{E}^2}} \leq \frac{\sqrt{2}}{||\vec{x}||_{taxi}}$, which gives us
\begin{align}
\max_{\vec{a} \neq \vec{b}} \frac{1}{3} \frac{||\vec{a} - \vec{b}||_{taxi}}{||\vec{a} - \vec{b}||_{\mathbb{E}^2}} \leq \max_{\vec{a} \neq \vec{b}} \frac{\sqrt{2}}{3}\frac{||\vec{a} - \vec{b}||_{taxi}}{||\vec{a} - \vec{b}||_{taxi}} 
= \frac{\sqrt{2}}{3} \leq 1\notag,
 \end{align}
 finishing the proof of Lemma \ref{dilF-+}.
 \end{proof}
 
Now we are ready to estimate the $\bar{d}_+$-distance between arbitrary points.

\begin{lem}\label{lem-point-dist-+}
For any two points $\vec{x} = (x_1, x_2), \vec{y} = (y_1, y_2) \in \mathbb{R}^2$, we have
\be\label{dist-estimate-+}
F_+(\vec{x} - \vec{y}) - (4h + 2L) \leq \bar{d}_+((x_1, x_2), (y_1, y_2)) \leq 
F_+(\vec{x} - \vec{y}) + (4h + 2L)
\ee
where $L$ and $h$ are the smocking length and depth, respectively.
\end{lem}

\begin{proof}
By the definition of the smocking depth, there exists some $I_{\vec{j}} = I_{(j_1,j_2)}$ and some $I_{\vec{j'}} = I_{(j_1',j_2')}$ such that $d_+(\vec{x}, I_{\vec{j}}) \leq h$ and $d_+(\vec{y}, I_{\vec{j'}}) \leq h$.

We first claim that
\be\label{lem-point-dist-claim1+}
|\bar{d}_+(\vec{x},\vec{y}) - \bar{d}_+(\vec{j}, \vec{j'})| \leq 2h.
\ee
Note that 
\be
\bar{d}_+(\vec{x},\vec{y}) \leq \bar{d}_+(\vec{x}, \vec{j}) + \bar{d}_+(\vec{j}, \vec{y}) \leq \bar{d}_+(\vec{x}, \vec{j}) + \bar{d}_+(\vec{j}, \vec{j'}) + \bar{d}_+(\vec{j'}, \vec{y}).
\ee
Thus 
\be
\left| \bar{d}_+(\vec{x},\vec{y}) - \bar{d}_+(\vec{j}, \vec{j'}) \right| \leq |\bar{d}_+(\vec{x}, \vec{j})| + |\bar{d}_+(\vec{j'}, \vec{y})| \leq 2h,
\ee
establishing (\ref{lem-point-dist-claim1+}).

Next, we claim that
\be\label{jtox-+}
|F_+(\vec{j} - \vec{j'}) - F_+(\vec{x} - \vec{y})| \leq 2(h+L).
\ee
Indeed, by definition of $\text{dil}(F_+)$, we have
\be
 \frac{|F(\vec{j} - \vec{j'}) - F(\vec{x} - \vec{y})|}{|(\vec{j}- \vec{j'}) - (\vec{x} - \vec{y})|} \leq \text{dil}(F).
\ee
By our choice of $\vec{j}$, $\vec{j'}$ and using Lemma \ref{dilF-+}, it follows that
\begin{align}
 |F_+(\vec{j} - \vec{j'}) - F_+(\vec{x} - \vec{y})|\,\, &\leq\,\, |(\vec{j}- \vec{j'}) - (\vec{x} - \vec{y})| \, \text{dil}(F_+) \\
 &\leq\,\, (|\vec{j} - \vec{x}| + |\vec{j'} - \vec{y}|) \, \text{dil}(F_+)\,\, \leq \,\,2(h+L)\notag,
 \end{align}
establishing inequality (\ref{jtox-+}).

Combining inequalities (\ref{lem-point-dist-claim1+}), (\ref{jtox-+}), and the triangle inequality, we have 
\begin{align}
|\bar{d}_+(\vec{x}, \vec{y}) - F(\vec{x} - \vec{y})| &\leq |\bar{d}_+(\vec{x},\vec{y})| - \bar{d}_+(\vec{j},\vec{j'})| + |F(\vec{j} - \vec{j'}) - F(\vec{x} - \vec{y})| \le 4h+2L,\notag
\end{align}
finishing the proof of Lemma \ref{lem-point-dist-+}.
\end{proof}

\begin{lem}\label{limit-norm-+}
For any two points $\vec{x}=(x_1,x_2), \vec{y}=(y_1,y_2)$, we have
\[
F_+(\vec{x}-\vec{y})=\lim_{R \to \infty} \frac{1}{R}\left[ \bar{d}_+((Rx_1, Rx_2),(Ry_1,Ry_2)) \right].
\]
\end{lem}

\begin{proof}
By Lemma \ref{lem-point-dist-+}, we have 
\be
\tfrac1R \left|\bar{d}_+((Rx_1, Rx_2), (Ry_1, Ry_2))-F(R\vec{x}-R\vec{y})\right|\leq\tfrac1R(4h + 2L)\notag.
\ee
Using the formula for $F_+$, we find
\be
\left|\tfrac1R \bar{d}_+((Rx_1, Rx_2), (Ry_1, Ry_2))-F_+(\vec{x}-\vec{y})\right|\leq\tfrac1R(4h + 2L)\notag.
\ee
Taking the limit of both sides,
$$
\bar{d}_{\infty}(\vec{x},\vec{y}) = F_+(\vec{x},\vec{y}) = \frac{1}{3} \left[ \frac{|x_1 - y_1|}{3} + \frac{|x_2 - y_2|}{3} \right],
$$
establishing the first part of Lemma \ref{limit-norm-+}.
\end{proof}

Combining these lemmas with Theorem~\ref{thm-smocking-R}
we conclude Theorem~\ref{thm-tan+}.

\subsection{ The Tangent Cone at Infinity of $X_\square$ by Dr.~Kazaras and Hindy}

In this section we prove the tangent cone at infinity of the $\square$ smocked space is 
a normed space whose unit ball is an octagon.  See Theorem~\ref {thm-tan-square}
which is depicted in Figure~\ref{fig:square-lim}.

\begin{figure}[h]
\includegraphics[width=.8 \textwidth]{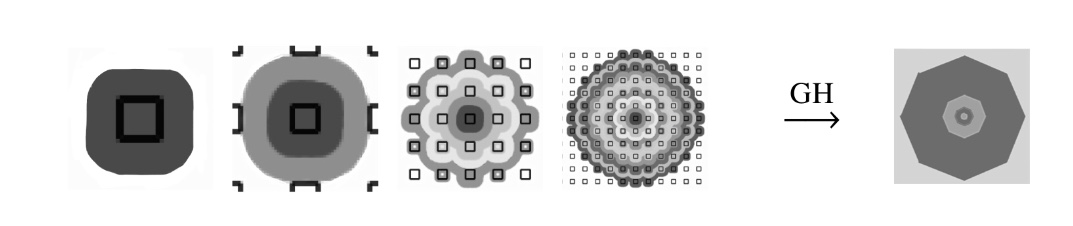}
\caption{Converging to the Tangent Cone at Infinity.}
\label{fig:square-lim}
\end{figure}

\begin{thm} \label{thm-tan-square}
The tangent cone at infinity of $(X_\square, d_\square)$ is $({\mathbb{R}}^2, d_{F_\square})$,
where
\be\label{F-def-square}
d_{F_\square}( \vec{x},\vec{y})= F_\square(\vec{x}- \vec{y})
\textrm{ and }
F_\square(\vec{x})= ||\vec{x}||_\square=\tfrac{2\sqrt{2}}{3}\min \left( |x_1|,|x_2| \right) + \tfrac{2}{3}  \left||x_1| - |x_2| \right| .
\ee
In the limit space, the shape of a ball, 
$
\{\vec{x}\in\mathbb{R}^2\colon \,\,F_\square(\vec{x}) = R\},
$
 is an octagon.
\end{thm}

Notice that, according to Lemma \ref{prop-dist-square}, for $\vec{j},\vec{j'}\in J_\square$
\be
d_{\square}(I_{\vec{j}}, I_{\vec{j'}}) =  F_\square(\vec{j} - \vec{j'}).
\ee
We will prove a series of lemmas to show we have all the hypotheses needed to apply  
Theorem~\ref{thm-smocking-R} taking $F=F_\square$.

\begin{lem}
The function $F_{\square}$ is a norm.
\end{lem}

\begin{proof}
It is definite:
\be
0=F_\square(\vec{x})= 2\sqrt{2}\min \left( \tfrac{|x_1|}{3},\tfrac{|x_2|}{3} \right) + 2 \left| \tfrac{|x_1|}{3} - \tfrac{|x_2|}{3} \right| 
\ee
if and only if $ \min |x_i|=0 \textrm{ and } |x_1|=|x_2|=0 $ which holds exactly when $x_1=x_2=0.$

\noindent
It scales:
\begin{align*}
F_\square(R\vec{x})&= 2\sqrt{2}\min \left( \tfrac{|Rx_1|}{3},\tfrac{|Rx_2|}{3} \right) + 2  \left| \tfrac{|Rx_1|}{3} - \tfrac{|Rx_2|}{3} \right|  \\
&= 2\sqrt{2}|R|\min \left( \tfrac{|x_1|}{3},\tfrac{|x_2|}{3} \right) + 2|R| \left| \tfrac{|x_1|}{3} - \tfrac{|x_2|}{3} \right|  
\quad = \quad |R| F_\square(\vec{x}).
\end{align*}

It satisfies the triangle inequality:
Using the fact that 
\begin{equation}
\big{|}|A|-|B|\big{|}=\max(|A|,|B|)-\min(|A|,|B|)
\end{equation}
for any numbers $A$ and $B$, we can rewrite the formula for $F_\square$ as follows:
\begin{equation}
    F_\square(\vec{x}) = \tfrac{2\sqrt{2} - 2}{3} \min( |x_1|, |x_2|) + \tfrac{2}{3}\max(|x_1|, |x_2|).
\end{equation}
Using this formula and applying the Euclidean triangle inequality, we obtain
\begin{align}
     F_\square(\vec{x}+\vec{y})&=\tfrac{2(\sqrt{2}-1)}{3}\min \left(|x_1 + y_1|,|x_2 + y_2| \right) + \tfrac{2}{3}\max\left(|x_1 + y_1|,|x_2 + y_2| \right)\notag\\
    &\leq \tfrac{2(\sqrt{2}-1)}{3}\min \left(|x_1| + |y_1|,|x_2| + |y_2| \right) + \tfrac{2}{3}\max\left(|x_1| +| y_1|,|x_2 |+| y_2| \right) \label{F_square_triangle_inequality}.
\end{align}
Note that \eqref{F_square_triangle_inequality} is an expression of the form $\tfrac{2(\sqrt{2}-1)}{3}a + \tfrac23b$ where $a+b = |x_1| + |y_1|  + |x_2| +|y_2|$. Since $\tfrac{2(\sqrt{2}-1)}{3}< \tfrac23$, the quantity $\tfrac{2(\sqrt{2}-1)}{3}a + \tfrac23b$ increases if $a$ is decreased and $b$ is increased, while $a+b$ is kept constant. Combining this observation with  \eqref{F_square_triangle_inequality}, we have
\begin{align*}
     F_\square(\vec{x}+\vec{y})&\leq\tfrac{2(\sqrt{2}-1)}{3}\Bigg(\min \left(|x_1|,|x_2|\right)+ \min\left(|y_1| , |y_2| \right)\Bigg) + \tfrac{2}{3}\Bigg(\max\left(|x_1|,|x_2 |\right) + \max\left(|y_1| ,| y_2| \right)\Bigg)\\
     &= F_\square(\vec{x}) + \tfrac32 F_\square(\vec{y}).
\end{align*}
\end{proof}

Next, we will estimate $\text{dil}(F_\square)$.

\begin{lem}\label{dilF-square}
The function $F_\square$ satisfies $\text{dil}(F_\square) \leq \tfrac{4}{3}$.
\end{lem}

\begin{proof}
Recall that $|x_1|+|x_2|\leq \sqrt{2}||\vec{x}||_{\mathbb{E}^2}$. For two points $\vec{a}=(a_1,a_2),\vec{b}=(b_1,b_2)$, we want to estimate the quantity
\begin{align}
\frac{| F_\square(\vec{a}) - F_\square(\vec{b})|}{||\vec{a} - \vec{b}||_{\mathbb{E}^2}} &
	=\frac{2}{3} \frac{\left| | \sqrt{2}(\min(|a_1|,|a_2|)-\min(|b_1|,|b_2|))+||a_1|-|a_2||-||b_1|-|b_2||\right|}{||\vec{a} - \vec{b}||_{\mathbb{E}^2}} \notag \\
{}&\leq \frac{2\sqrt{2}}{3} \frac{\left| | \sqrt{2}(\min(|a_1|,|a_2|)-\min(|b_1|,|b_2|))+||a_1|-|a_2||-||b_1|-|b_2||\right|}{|a_1-b_1|+|a_2-b_2|} \label{eq-dilF-square}
\end{align}
By the symmetry of expression (\ref{eq-dilF-square}) in $\vec{a}$ and $\vec{b}$, it suffices to consider the following two cases: first when $|a_1|\geq|a_2|$, $|b_1|\geq|b_2|$
and second when $|a_1|\geq|a_2|$, $|b_1|\leq|b_2|$.

First, assume that $|a_1|\geq|a_2|$ and $|b_1|\geq|b_2|$. Then we may estimate (\ref{eq-dilF-square}) by
\begin{align}
\frac{| F_\square(\vec{a}) - F_\square(\vec{b})|}{||\vec{a} - \vec{b}||_{\mathbb{E}^2}} &
 \leq\frac{2\sqrt{2}}{3} \frac{\left|  \sqrt{2}(|a_2|-|b_2|)+|a_1|-|a_2|-|b_1|+|b_2|\right|}{|a_1-b_1|+|a_2-b_2|}\\
 &\leq \frac{2\sqrt{2}}{3}\frac{\left||a_1|-|b_1|+|a_2|-|b_2|\right|}{|a_1-b_1|+|a_2-b_2|}
 \leq\frac{2\sqrt{2}}{3}\frac{|a_1-b_1|+|a_2-b_2|}{|a_1-b_1|+|a_2-b_2|} =\frac{2\sqrt{2}}{3}\notag,
\end{align}
which gives the desired estimate.

Now assume that $|a_1|\geq|a_2|$ and $|b_1|\leq|b_2|$. Then we may estimate (\ref{eq-dilF-square}) by
\begin{align}
\frac{| F_\square(\vec{a}) - F_\square(\vec{b})|}{||\vec{a} - \vec{b}||_{\mathbb{E}^2}} &
\leq \frac{2\sqrt{2}}{3} \frac{\left|  \sqrt{2}(|a_2|-|b_1|)+|a_1|-|a_2|+|b_1|-|b_2|\right|}{|a_1-b_1|+|a_2-b_2|}\notag\\
&
 	 =\frac{2\sqrt{2}}{3} \frac{\left|  (\sqrt{2}-1)|a_2|-(\sqrt{2}-1)|b_1|+|a_1|-|b_2|\right|}{|a_1-b_1|+|a_2-b_2|}\notag\\
{}&\leq\begin{cases}
\frac{2\sqrt{2}}{3} \frac{(\sqrt{2}-1)|a_1|-(\sqrt{2}-1)|b_1|+|a_1|-|b_1|}{|a_1-b_1|+|a_2-b_2|}&\text{ if }F_\square(\vec{a})-F_\square(\vec{b})\geq0\\
\frac{2\sqrt{2}}{3} \frac{-(\sqrt{2}-1)|a_2|+(\sqrt{2}-1)|b_2|-|a_2|+|b_2|}{|a_1-b_1|+|a_2-b_2|}&\text{ otherwise}
\end{cases}\notag\\
{}&\leq\begin{cases}
\frac{4}{3} \frac{|a_1-b_1|}{|a_1-b_1|+|a_2-b_2|}&\text{ if }F_\square(\vec{a})-F_\square(\vec{b})\geq0\\
\frac{4}{3} \frac{|a_2-b_2|}{|a_1-b_1|+|a_2-b_2|}&\text{ otherwise}
\end{cases}\qquad \leq\frac43\notag,
\end{align}
which gives the desired estimate and completes the proof of Lemma \ref{dilF-square}.
\end{proof}
 
Now we are ready to estimate the $\bar{d}_\square$-distance between arbitrary points.

\begin{lem}\label{lem-point-dist-square}
For any two points $\vec{x} = (x_1, x_2), \vec{y} = (y_1, y_2) \in \mathbb{R}^2$, we have
\be\label{dist-estimate-square}
F_\square(\vec{x} - \vec{y}) - (\tfrac{14}{3}h+\tfrac{8}{3}L) \leq \bar{d}_\square((x_1, x_2), (y_1, y_2)) \leq 
F_\square(\vec{x} - \vec{y}) + (\tfrac{14}{3}h+\tfrac{8}{3}L)
\ee
where $L$ and $h$ are the smocking length and depth, respectively.
\end{lem}

\begin{proof}
By the definition of the smocking depth, there exists some $I_{\vec{j}} = I_{(j_1,j_2)}$ and some $I_{\vec{j'}} = I_{(j_1',j_2')}$ such that $d_\square(\vec{x}, I_{\vec{j}}) \leq h$ and $d_\square(\vec{y}, I_{\vec{j'}}) \leq h$.

We first claim that
\be\label{lem-point-dist-claim1square}
|\bar{d}_\square(\vec{x},\vec{y}) - \bar{d}_\square(\vec{j}, \vec{j'})| \leq 2h.
\ee
Note that 
\be
\bar{d}_\square(\vec{x},\vec{y}) \leq \bar{d}_\square(\vec{x}, \vec{j}) + \bar{d}_\square(\vec{j}, \vec{y}) \leq \bar{d}_\square(\vec{x}, \vec{j}) + \bar{d}_\square(\vec{j}, \vec{j'}) + \bar{d}_\square(\vec{j'}, \vec{y}).
\ee
Thus 
\be
\left| \bar{d}_\square(\vec{x},\vec{y}) - \bar{d}_\square(\vec{j}, \vec{j'}) \right| \leq |\bar{d}_\square(\vec{x}, \vec{j})| + |\bar{d}_\square(\vec{j'}, \vec{y})| \leq 2h,
\ee
establishing (\ref{lem-point-dist-claim1square}).

Next, we claim that
\be\label{jtox-square}
|F_\square(\vec{j} - \vec{j'}) - F_\square(\vec{x} - \vec{y})| \leq \tfrac{8}{3}(h+L).
\ee
Indeed, by definition of $\text{dil}(F_\square)$, we have
$$
 \frac{|F_\square(\vec{j} - \vec{j'}) - F_\square(\vec{x} - \vec{y})|}{|(\vec{j}- \vec{j'}) - (\vec{x} - \vec{y})|} \leq \text{dil}(F_\square).
$$
By our choice of $\vec{j}$, $\vec{j'}$ and using Lemma \ref{dilF-square}, it follows that
\begin{align}
 |F_\square(\vec{j} - \vec{j'}) - F_\square(\vec{x} - \vec{y})| &\leq |(\vec{j}- \vec{j'}) - (\vec{x} - \vec{y})| \, \text{dil}(F_\square) \notag\\
 {}&\leq (|\vec{j} - \vec{x}| + |\vec{j'} - \vec{y}|) \, \text{dil}(F_\square) \notag\\
 {}&\leq \tfrac{8}{3}(h+L)\notag,
 \end{align}
establishing inequality (\ref{jtox-square}).

Combining inequalities (\ref{lem-point-dist-claim1square}), (\ref{jtox-square}), and the triangle inequality, we have 
\begin{align}
|\bar{d}_\square(\vec{x}, \vec{y}) - F(\vec{x} - \vec{y})| &\leq |\bar{d}_\square(\vec{x},\vec{y})| - \bar{d}_\square(\vec{j},\vec{j'})| + |F_\square(\vec{j} - \vec{j'}) - F_\square(\vec{x} - \vec{y})| \notag\\
{}&\leq \tfrac{14}{3}h+\tfrac{8}{3}L,\notag
\end{align}
finishing the proof of Lemma \ref{lem-point-dist-square}.
\end{proof}

Combining Lemma \ref{lem-point-dist-square} with Theorem~\ref{thm-smocking-R}, we conclude Theorem~\ref{thm-tan-square}.

\subsection{ The Tangent Cone at Infinity of $X_T$  by Prof.~Sormani, Victoria, David, and Aleah}

In this section, we prove the tangent cone at infinity of the T smocked space is 
a normed space whose unit ball, $\{\vec{x}\in\mathbb{R}^2\colon ||\vec{x}||_{\infty,T} = R\}$, is a square. 
See Theorem~\ref {thm-tan-T}
which is depicted in Figure~\ref{fig:T-lim}.

\begin{figure}[h]
\includegraphics[width=.8 \textwidth]{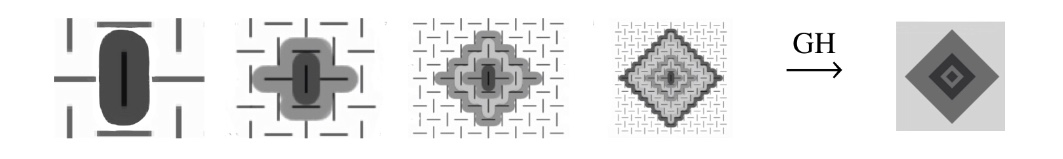}
\caption{Converging to the Tangent Cone at Infinity.}
\label{fig:T-lim}
\end{figure}

\begin{thm} \label{thm-tan-T}
The tangent cone at infinity of $(X_T, d_T)$ is $({\mathbb{R}}^2, d_{F_T})$,
where
\be \label{F-def-T}
d_{F_T}( \vec{x},\vec{y})= F_T(\vec{x}- \vec{y})
\textrm{ where }
F_T(\vec{x})=||\vec{x}||_T=\frac{|x_1|+|x_2|}{2}
\ee
Note that in the limit space, the shape of a ball, 
$
\{\vec{x}\in\mathbb{R}^2\colon ||\vec{x}||_{T} = R\},
$
 is a square with corners at $(\pm 2,0)$ and $(0, \pm 2)$.
\end{thm}

Recall that, according to Lemma \ref{prop-dist-T}, for $\vec{j},\vec{j'}\in J_T$
\be
d_{T}(I_{\vec{j}}, I_{\vec{j'}}) =  F_T(\vec{j} - \vec{j'}).
\ee
We will prove a series of lemmas to show we have all the hypotheses needed to apply  
Theorem~\ref{thm-smocking-R} taking $F=F_T$.

\begin{lem}\label{dilF-T}
The function $F_T$ satisfies $\text{dil}(F_T) \leq 1$.
\end{lem}

\begin{proof}
Recall $||\vec{x}||_{taxi} = |x_1| + |x_2|$ is the taxi norm, and 
$||\vec{x}||_{\mathbb{E}^2} = \sqrt{(x_1)^2 + (x_2)^2}$ is the Euclidean norm. 
It is well-known that $||\vec{x}||_{taxi} \leq \sqrt{2} || \vec{x} ||_{\mathbb{E}^2}$.
Let $\vec{a}=(a_1,a_2)$ and $\vec{b}=(b_1,b_2)$ be points in the plane. 
We may estimate
\begin{align}
\frac{|F_T(\vec{a})-F_T(\vec{b})|}{||\vec{a}-\vec{b}||_{\mathbb{E}^2}}&
	\leq\frac{||a_1|+|a_2| -|b_1|-|b_2||}{||\vec{a}-\vec{b}||_{\mathbb{E}^2}} 
\leq\sqrt{2}\,\left(\frac{|a_1-b_1|+|a_2-b_2|}{||\vec{a}-\vec{b}||_{taxi}}\right)
=1.\notag
\end{align}
 \end{proof}
 
Now we are ready to estimate the $\bar{d}_T$-distance between arbitrary points.

\begin{lem}\label{lem-point-dist-T}
For any two points $\vec{x} = (x_1, x_2), \vec{y} = (y_1, y_2) \in \mathbb{E}^2$, we have
\be\label{dist-estimate-T}
F_T(\vec{x} - \vec{y}) - (4h + 2L) \leq \bar{d}_T((x_1, x_2), (y_1, y_2)) \leq 
F_T(\vec{x} - \vec{y}) + (4h + 2L)
\ee
where $L$ and $h$ are the smocking length and depth of $X_T$, respectively.
\end{lem}

\begin{proof}
By the definition of the smocking depth, 
\be
\exists \,\,I_{\vec{j}} = I_{(j_1,j_2)} \textrm{ and } I_{\vec{j'}} = I_{(j_1',j_2')}
\textrm{ s.t. } d_T(\vec{x}, I_{\vec{j}}) \leq h \textrm{ and }d_T(\vec{y}, I_{\vec{j'}}) \leq h.
\ee
We first claim that
\be\label{lem-point-dist-claim1T}
|\bar{d}_T(\vec{x},\vec{y}) - \bar{d}_T(\vec{j}, \vec{j'})| \leq 2h.
\ee
To see this, note that 
\be
\bar{d}_T(\vec{x},\vec{y}) \leq \bar{d}_T(\vec{x}, \vec{j}) + \bar{d}_T(\vec{j}, \vec{y}) \leq \bar{d}_T(\vec{x}, \vec{j}) + \bar{d}_T(\vec{j}, \vec{j'}) + \bar{d}_T(\vec{j'}, \vec{y}).
\ee
Thus 
\be
\left| \bar{d}_T(\vec{x},\vec{y}) - \bar{d}_T(\vec{j}, \vec{j'}) \right| \leq |\bar{d}_T(\vec{x}, \vec{j})| + |\bar{d}_T(\vec{j'}, \vec{y})| \leq 2h,
\ee
establishing (\ref{lem-point-dist-claim1T}).

Next, we claim that
\be\label{jtox-T}
|F_T(\vec{j} - \vec{j'}) - F_T(\vec{x} - \vec{y})| \leq 2(h+L).
\ee
Indeed, by definition of $\text{dil}(F_T)$,
\be
 \frac{|F_T(\vec{j} - \vec{j'}) - F_T(\vec{x} - \vec{y})|}{|(\vec{j}- \vec{j'}) - (\vec{x} - \vec{y})|} \leq \text{dil}(F_T).
\ee
By our choice of $\vec{j}$, $\vec{j'}$ and using Lemma \ref{dilF-T}, it follows that
\begin{align}
 |F_T(\vec{j} - \vec{j'}) - F_T(\vec{x} - \vec{y})| &\leq |(\vec{j}- \vec{j'}) - (\vec{x} - \vec{y})| \, \text{dil}(F_T) \notag\\
 {}&\leq (|\vec{j} - \vec{x}| + |\vec{j'} - \vec{y}|) \, \text{dil}(F_T) \notag\\
 {}&\leq 2(h+L)\notag,
 \end{align}
establishing inequality (\ref{jtox-T}).

Combining inequalities (\ref{lem-point-dist-claim1T}), (\ref{jtox-T}), and the triangle inequality,
\begin{align}
|\bar{d}_T(\vec{x}, \vec{y}) - F_T(\vec{x} - \vec{y})| &\leq |\bar{d}_T(\vec{x},\vec{y})| - \bar{d}_T(\vec{j},\vec{j'})| + |F_T(\vec{j} - \vec{j'}) - F_T(\vec{x} - \vec{y})| \notag\\
{}&\leq 4h+2L,\notag
\end{align}
finishing the proof of Lemma \ref{lem-point-dist-T}.
\end{proof}

\begin{lem}\label{limit-norm-T}
For any two points $\vec{x}=(x_1,x_2), \vec{y}=(y_1,y_2)$, we have
\be
F_T(\vec{x}-\vec{y})=\lim_{R \to \infty} \tfrac{1}{R}\left[ \bar{d}_T((Rx_1, Rx_2),(Ry_1,Ry_2)) \right].
\ee
Moreover, $||\vec{x}||_{\infty,T}=F_T(\vec{x})$ defines a norm on $\mathbb{R}^2$.
\end{lem}

\begin{proof}
By Lemma \ref{lem-point-dist-T}, we have 
\be
\tfrac1R |\bar{d}_T((Rx_1, Rx_2), (Ry_1, Ry_2))-F_T(R\vec{x}-R\vec{y})|\leq\tfrac1R(4h + 2L).
\ee
Inspecting the formula for $F_T$, one can see that $F_T(\lambda\vec{x})=\lambda F_T(\vec{x})$ for any $\lambda\geq0$. It follows that
\be
|\tfrac1R \bar{d}_T((Rx_1, Rx_2), (Ry_1, Ry_2))-F_T(\vec{x}-\vec{y})|\leq\tfrac1R(4h + 2L).
\ee
Taking the limit of both sides,
\be
\bar{d}_{\infty,T}(\vec{x},\vec{y}) = F_T(\vec{x}-\vec{y}),
\ee
establishing the first part of Lemma \ref{limit-norm-T}.

Finally, since $||\cdot||_{\infty,T} = d_{\infty,T}(\cdot,(0,0))$ is a positive multiple of the taxicab norm, it is also a norm.
\end{proof}

Combining these lemmas with Theorem~\ref{thm-smocking-R}
we conclude Theorem~\ref{thm-tan-T}.

\section{Open Problems}

In this section we propose a collection of open problems.  Please be sure to contact
Prof.~Sormani if you are interested in working on some of these problems so that
you are not in competition with other teams.

\subsection{For Undergraduates and Masters students}

\begin{problem}
As suggested in the introduction: students may consider the
balls, distances, and tangent cones at infinity for other periodic smocking patterns as in
as in Figure~\ref{fig:other-patterns}.
You might consider the various herringbone smocking patterns that are commonly sewn.  You might also consider the patterns we
have already studied but change the separation constants and lengths.  
\end{problem}

\begin{figure}[h]
\includegraphics[width=.3 \textwidth]{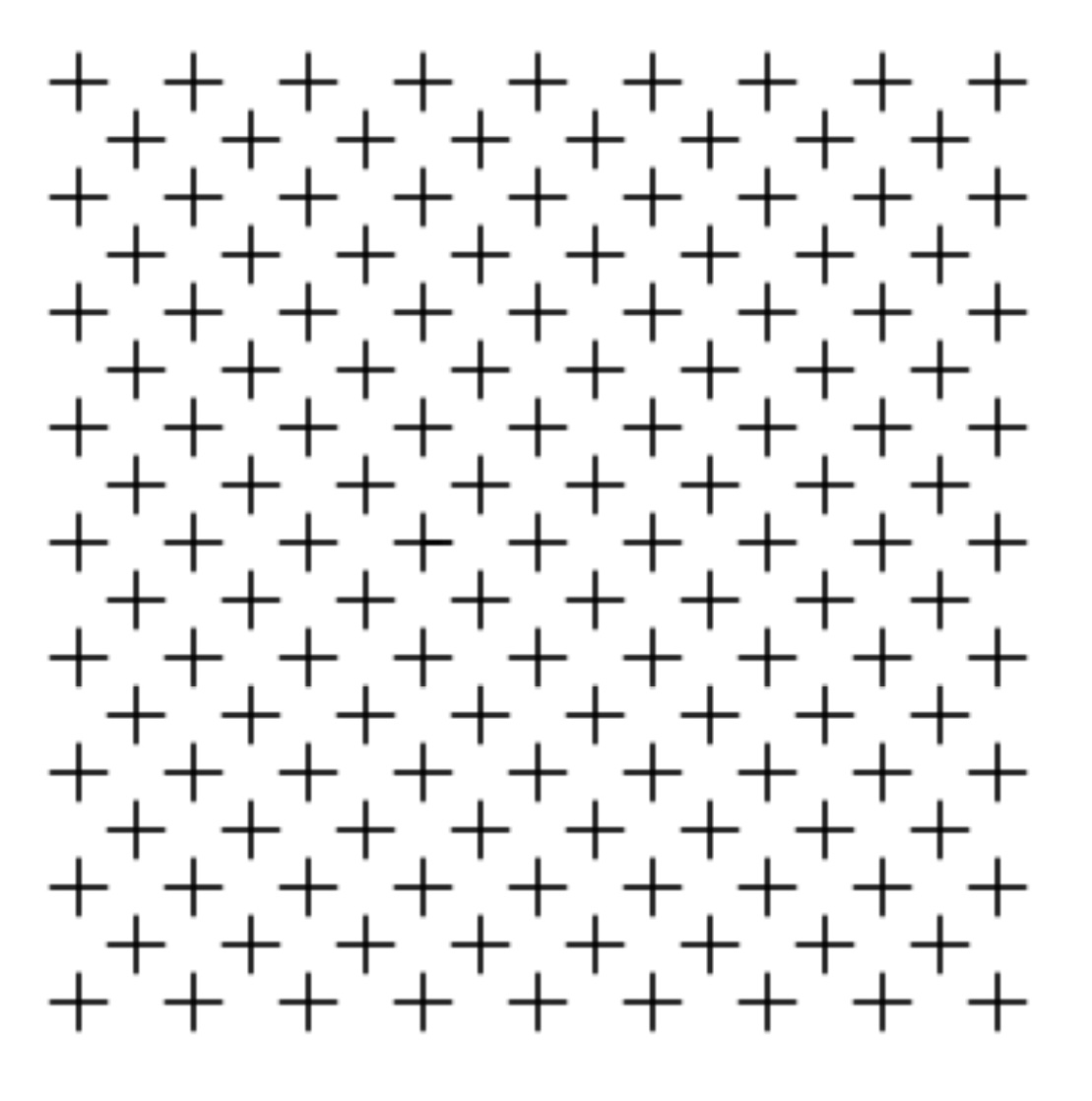} 
\includegraphics[width=.3 \textwidth]{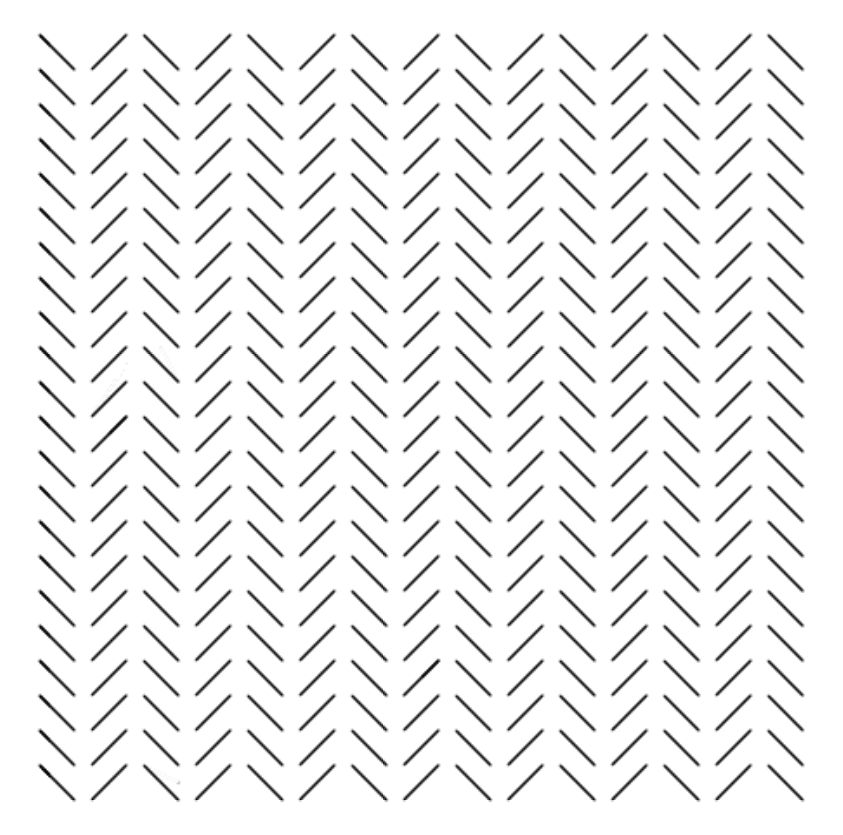} 
\includegraphics[width=.3 \textwidth]{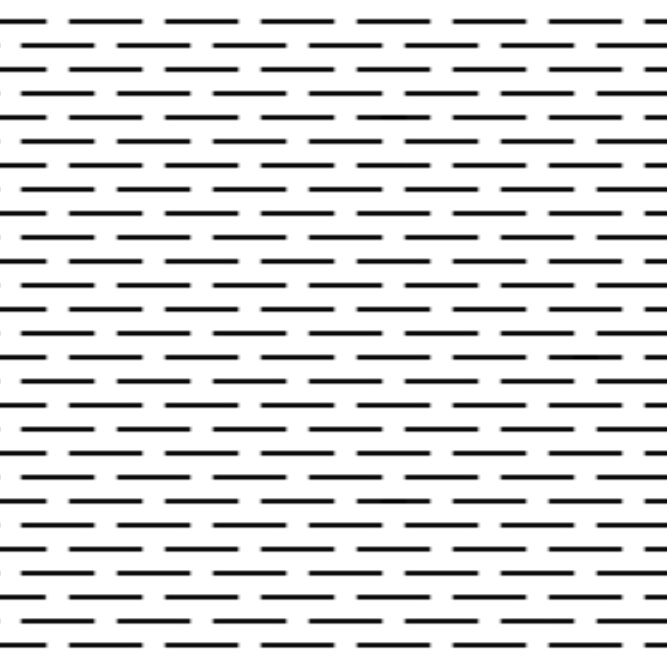} \\
\includegraphics[width=.3 \textwidth]{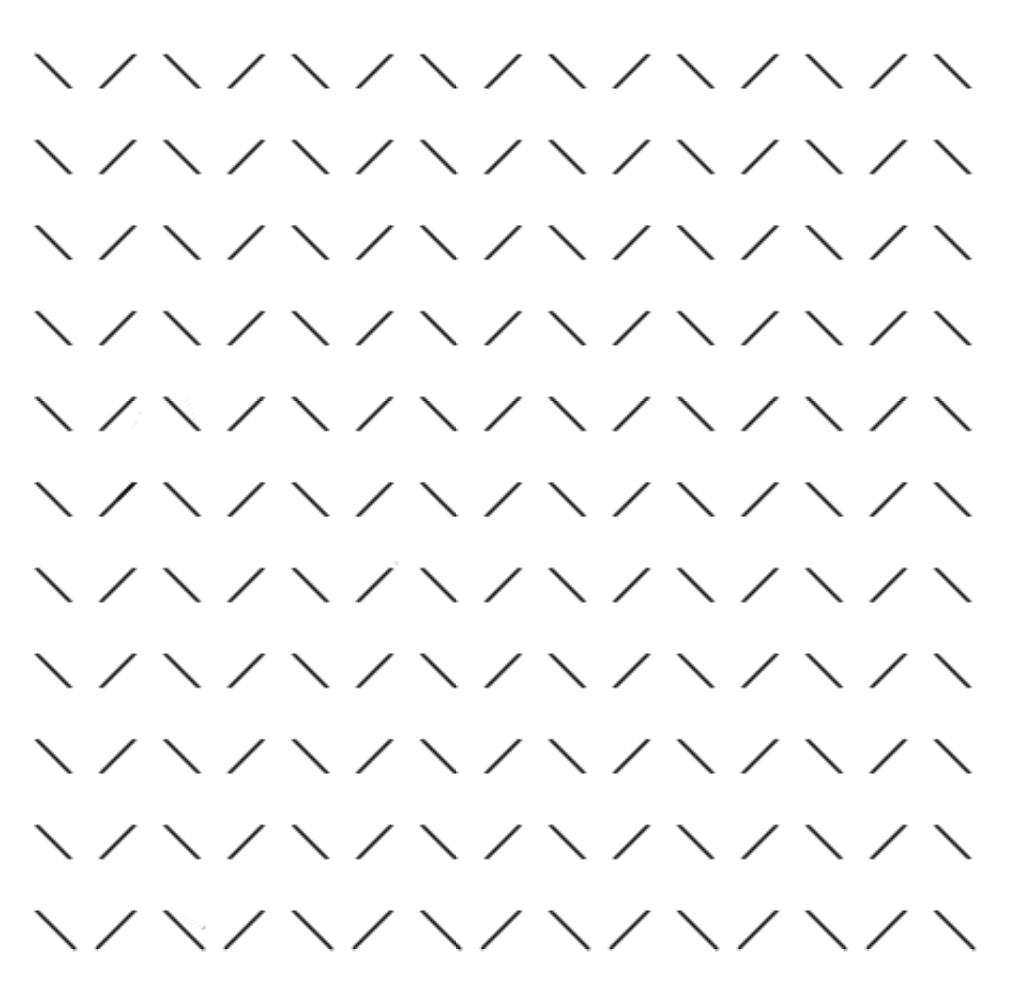} 
\includegraphics[width=.3 \textwidth]{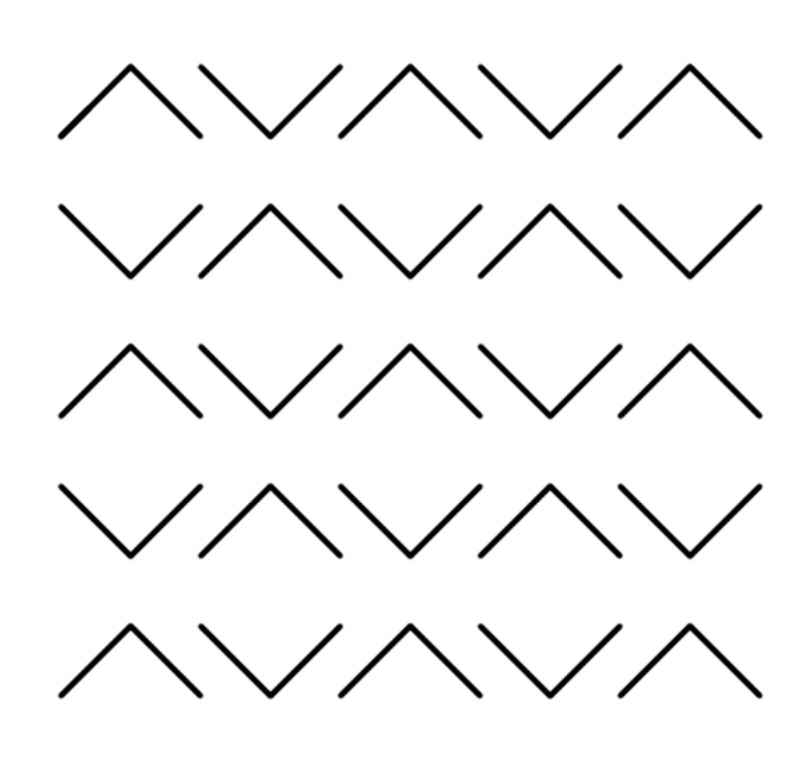} 
\includegraphics[width=.3 \textwidth]{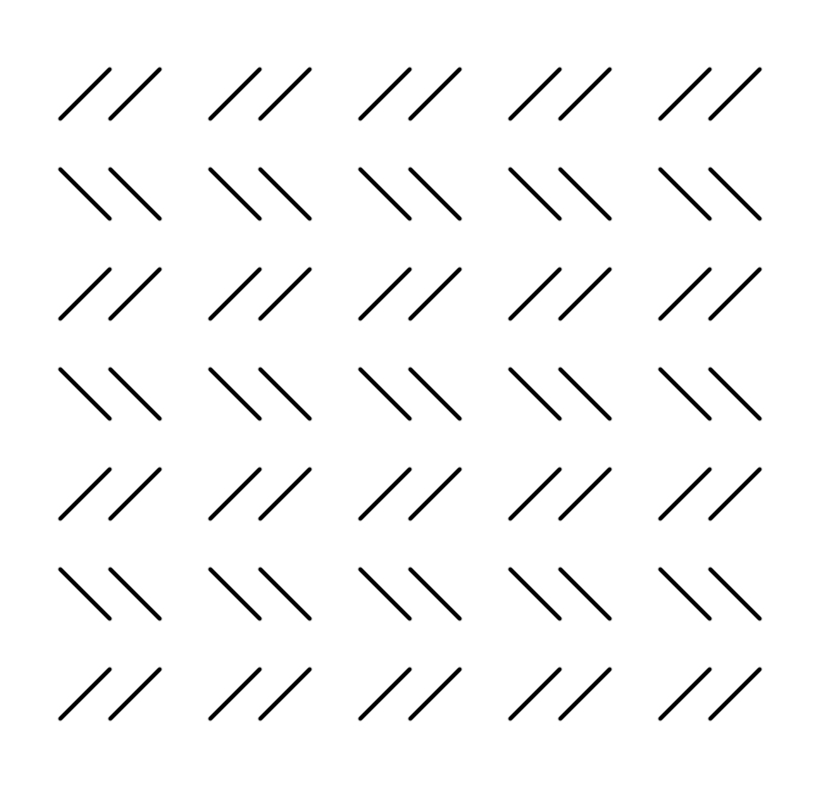} 
\caption{Additional smocked spaces: .}
\label{fig:other-patterns}
\end{figure}

\begin{problem}
Three dimensional smocking patterns will be explored in \cite{SWIF-smocked}.   You may wish to examine others.
\end{problem}

\begin{problem}
Students might also explore the Gromov-Hausdorff limits of sequences of different periodic smocking patterns.  
What happens for example if you examine a sequence of patterns similar to one of our patterns as the
length is taken to $0$ with a fixed index set $J$ as in Figure~\ref{fig:Lto0}?    Does it have a Gromov-Haudorff limit?  Can
you prove a general theorem about the Gromov-Hausdorff limits and what they are?
\end{problem}

\begin{figure}[h]
\includegraphics[width=.8 \textwidth]{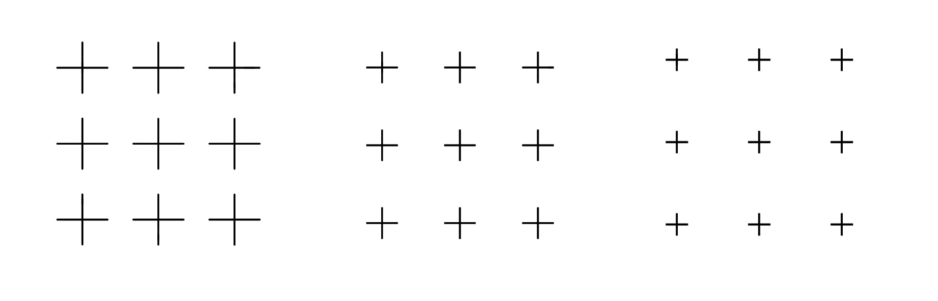} 
\caption{Taking smocking lengths to $0$.   What is the GH limit?}
\label{fig:Lto0}
\end{figure}

\begin{problem}
Students might also explore the Gromov-Hausdorff limit of a sequence of patterns obtained by rescaling one
of our patterns horizontally but not vertically as in Figure~\ref{fig:horiz}?   Does it have a Gromov-Haudorff limit?
\end{problem}

\begin{figure}[h]
\includegraphics[width=.8 \textwidth]{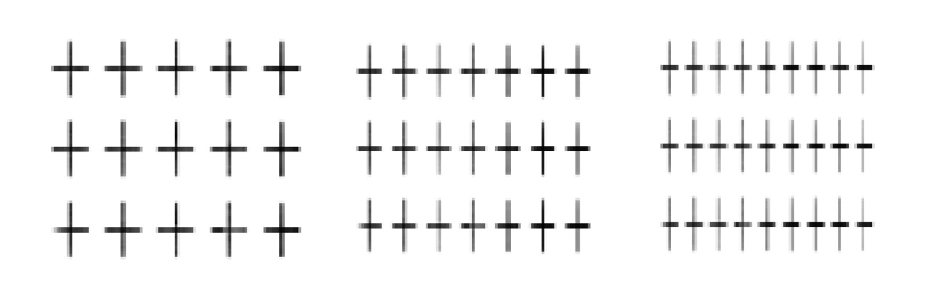} 
\caption{Rescaling horizontally, does the GH limit exist?}
\label{fig:horiz}
\end{figure}

\begin{problem}
Show that the smocking space $(X_Q, d_Q)$ defined by the smocking intervals
\be
S_Q= \{ \, \{j\} \times [0,j]\, : \, j\in J_Q \, \} \textrm{ where } J_Q=\{ j_k: \, k\in {\mathbb{N}}\} \textrm{ with }
\lim_{k\to \infty} j_k/j_{k+1}=0
\ee
does not have a unique tangent cone at infinity and the tangent
cone is not always a normed space.  First show that the GH limit when  rescaling by $R_k=j_k \to \infty$ is
a pulled thread space with a interval of unit length at $\{1\}\times [0,1]$.  Then rescale instead
by $R_k=(j_k+j_{k+1})/2$ and show the GH limit is Euclidean space.  What other possible tangent cones
at infinity do you find?
\end{problem}

\subsection{More Advanced Questions}

\begin{problem}
Under what conditions is the rescaled limit of a smocked metric space a unique normed vector space?
When does it exist?  When is it unique?  When is it a normed vector space?   Here we are asking you to
find hypotheses which allow you to apply Theorem~\ref{thm-smocking-R} without having to carefully
analyze the formula for the distance between intervals in the space.
\end{problem}

\begin{problem}
Note that in this paper two of our tangent cones at infinity, that of $X_T$ and $X_+$ were isometric to one another as they
are both just rescalings of taxispace (see Example~\ref{taxi-same}).   Is there a way to assess a pair of 
smocking metric spaces initially to determine if they have the same tangent cone at infinity without conducting a
complete derivation of their distance functions?  
\end{problem}

\begin{problem}
If you have a sequence of smocked spaces defined using smocking sets,
$S_k\subset {\mathbb{E}}^N$, and the sets $S_K$ converge in the Hausdorff sense
to a smocking set $S_\infty$, 
\be
d_H(S_k, S_\infty) \to 0,
\ee
then do the corresponding smocked metric spaces converge
\be
d_{GH}(X_k, X_\infty) \to 0?
\ee
Prove this or find a counter example.
\end{problem}

\begin{problem}
You may read in \cite{BBI} or \cite{Gromov-metric} to learn Gromov's Compactness Theorem,
which describes when a sequence of metric spaces has a GH converging subsequence.  This
theorem requires a uniform upper bound on diameter and the number of balls  of any given radius.
Smocking spaces do not satisfy these hypotheses and yet we saw the rescaled sequences
converged.  When does a sequence of smocked metric space have a GH converging subsequence? 
Do uniform bounds on the smocking constants suffice?   What properties do the limit spaces have?
\end{problem}  

\begin{problem}
Given a finite dimensional normed vector space, can one find a smocked space whose
unique tangent cone at infinity is that given space?
\end{problem}


\end{document}